\documentclass[12pt]{article}    
\usepackage{amsmath, amssymb, amsthm,pdfsync}
\usepackage{fullpage}
\usepackage{verbatim}
\usepackage{graphicx}
\usepackage{subfig, caption}
\usepackage{cancel}
\usepackage{enumerate}
\usepackage{enumitem}
\usepackage{wrapfig}
\usepackage{mathtools}
\usepackage{xcolor}
\usepackage{epigraph}
\usepackage{cleveref}
\usepackage{overpic}
\usepackage{esint}
\usepackage{dsfont}
\usepackage{color}
\usepackage{amsrefs}
\usepackage{mathrsfs}
\usepackage{stmaryrd}
\usepackage[all]{xy}
\usepackage[mathcal]{eucal}
\usepackage{csquotes}

\definecolor{darkgreen}{rgb}{.25,.65,.25}

\newcommand{\ind}{\mathds{1}}

\setlist[itemize]{itemsep=-1mm}

\newtheorem{theorem}{Theorem}

\newtheorem{lem}[theorem]{Lemma}
\newtheorem{prop}[theorem]{Proposition}
\newtheorem{corollary}[theorem]{Corollary}

\newtheorem{rem}[theorem]{Remark}

\newcommand{\vertiii}[1]{{\left\vert\kern-0.25ex\left\vert\kern-0.25ex\left\vert #1 
    \right\vert\kern-0.25ex\right\vert\kern-0.25ex\right\vert}}
\newcommand{\e}{\epsilon}

\newcommand{\N}{\mathbb{N}}

\newcommand{\R}{\mathbb{R}}

\def\epsilon{\varepsilon}
\def\tilde{\widetilde}

\newcommand{\ds}{\displaystyle}

\DeclareMathOperator{\supp}{supp}

\DeclareMathOperator{\sign}{sign}
\DeclareMathOperator{\erfc}{erfc}

\numberwithin{equation}{section}
\numberwithin{theorem}{section}

\Crefname{assumption}{Assumption}{Assumptions}
\Crefname{theorem}{Theorem}{Theorems}
\Crefname{lem}{Lemma}{Lemmas}
\Crefname{cor}{Corollary}{Corollaries}
\Crefname{prop}{Proposition}{Propositions}
\Crefname{theorem}{Theorem}{Theorems}
\Crefname{conjecture}{Conjecture}{Conjectures}

\begin{document}

\title{{\bf{Propagation in a Fisher-KPP equation with non-local advection}}\thanks{The project leading to this publication was performed within the framework of Excellence Initiatives of Aix-Marseille University A*MIDEX and Universit\'e de Lyon (ANR-11- IDEX-0007), and Excellence Laboratories Archimedes LabEx (ANR-11-LABX-0033) and LABEX MILYON (ANR-10-LABX-0070), French ``Investissements d'Avenir" programmes operated by the French National Research Agency (ANR). It has received funding from the ANR NONLOCAL project (ANR-14-CE25-0013) and from the European Research Council under the European Union's Seventh Framework Programme (FP/2007-2013) ERC Grant Agreement n.~321186~- ReaDi~- Reaction-Diffusion Equations, Propagation and Modelling, and under the European Union Horizon 2020 Research and Innovation Programme (Grant Agreement n. 639638). CH was partially supported by the National Science Foundation Research Training Group grant
DMS-1246999.}}

\author{Fran{\c{c}}ois Hamel$^{\,\hbox{\small{a}}}$ and Christopher Henderson$^{\,\hbox{\small{b}}}$\\
\\
\footnotesize{$^{\,\hbox{\small{a}}}$ Aix Marseille Univ, CNRS, Centrale Marseille, I2M, Marseille, France}\\
\footnotesize{$^{\,\hbox{\small{b}}}$ The University of Chicago, Department of Mathematics, Chicago, IL, USA}}

\date{}

\maketitle

\begin{abstract}
We investigate the influence of a general non-local advection term of the form $K*u$ to propagation in the one-dimensional Fisher-KPP equation.  This model is a generalization of the Keller-Segel-Fisher system.  When $K \in L^1(\R)$, we obtain explicit upper and lower bounds on the propagation speed which are asymptotically sharp and more precise than previous works.  When $K \in L^p(\R)$ with $p>1$ and is non-increasing in $(-\infty,0)$ and in $(0,+\infty)$, we show that the position of the ``front'' is of order $O(t^{1/p})$ if $p < \infty$ and $O(e^{\lambda t})$ for some $\lambda > 0$ if $p = \infty$ and $K(+\infty)>0$.  
We use a wide range of techniques in our proofs.
\end{abstract}

\tableofcontents


\section{Introduction and main results}\label{sec:results}

The model we consider is
\begin{equation}\label{eq:the_equation}
	\begin{cases}
		u_t + [ (K*u) u]_x = u_{xx} + u(1-u), \qquad &\text{ in } (0,+\infty)\times \R,\\
		u(0,x) = u_0(x) \geq 0, &\text{ in } \R,
	\end{cases}
\end{equation}
where $u_0$ is compactly supported and non-negative and satisfies
\begin{equation}\label{hypu0}
0 < \|u_0\|_{L^\infty(\R)} \leq 1.
\end{equation}
The unknown, $u$,  typically represents the population density of a species.  Throughout the paper, $K$ is a bounded odd function which is monotonic except at $x = 0$, in the sense that it is monotonic in $(-\infty,0)$ and in $(0,+\infty)$. The notation $K*u$ stands for the convolution $K*u(t,\cdot)$ of $K$ with $u(t,\cdot)$. We also assume that $K$ does not change sign in $(-\infty,0)$ and in $(0,+\infty)$. We define the ``jump'' of $K$ as
\begin{equation}\label{eq:J}
	J := \lim_{x\to 0^-} 2K(x).
\end{equation}
We provide a cartoon picture of examples of $K$ and the value of $J$ in~\Cref{fig:K}. We mainly focus on three cases, leading to completely different behaviours: (i)~when $K \in L^1(\R)$, (ii)~when $K\in L^p(\R)$ with $p\in(1,+\infty)$ and (iii)~when $\lim_{x\to+\infty} K(x) > 0$. 

\begin{figure}[h]
\begin{center}
\begin{overpic}[scale=.31]
	{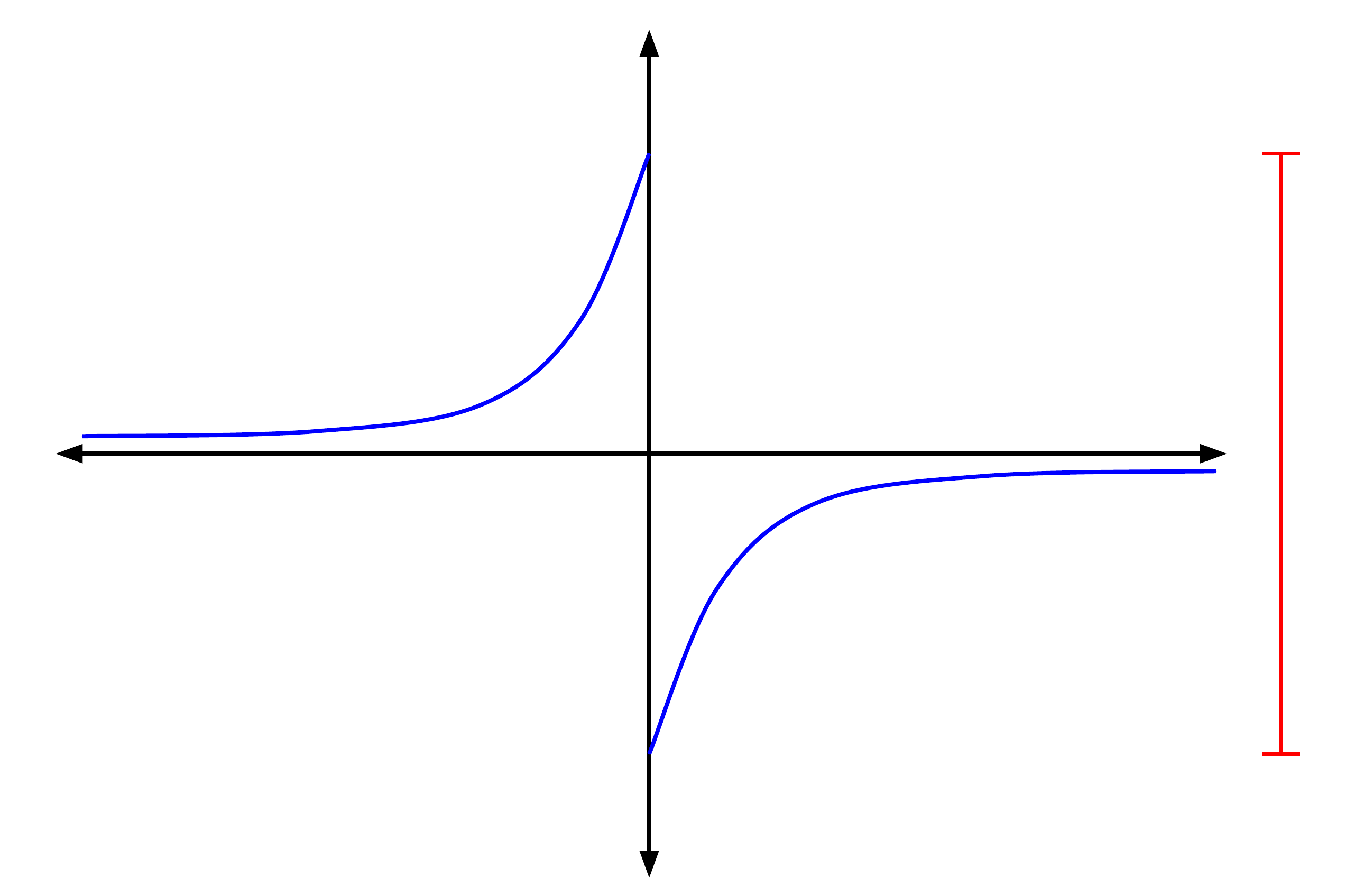}
	\put(4,28){$x$}
	\put(23, 39){\color{blue} $K$}
	\put(95,30){\color{red} $J$}
\end{overpic}
\begin{overpic}[scale=.31]
	{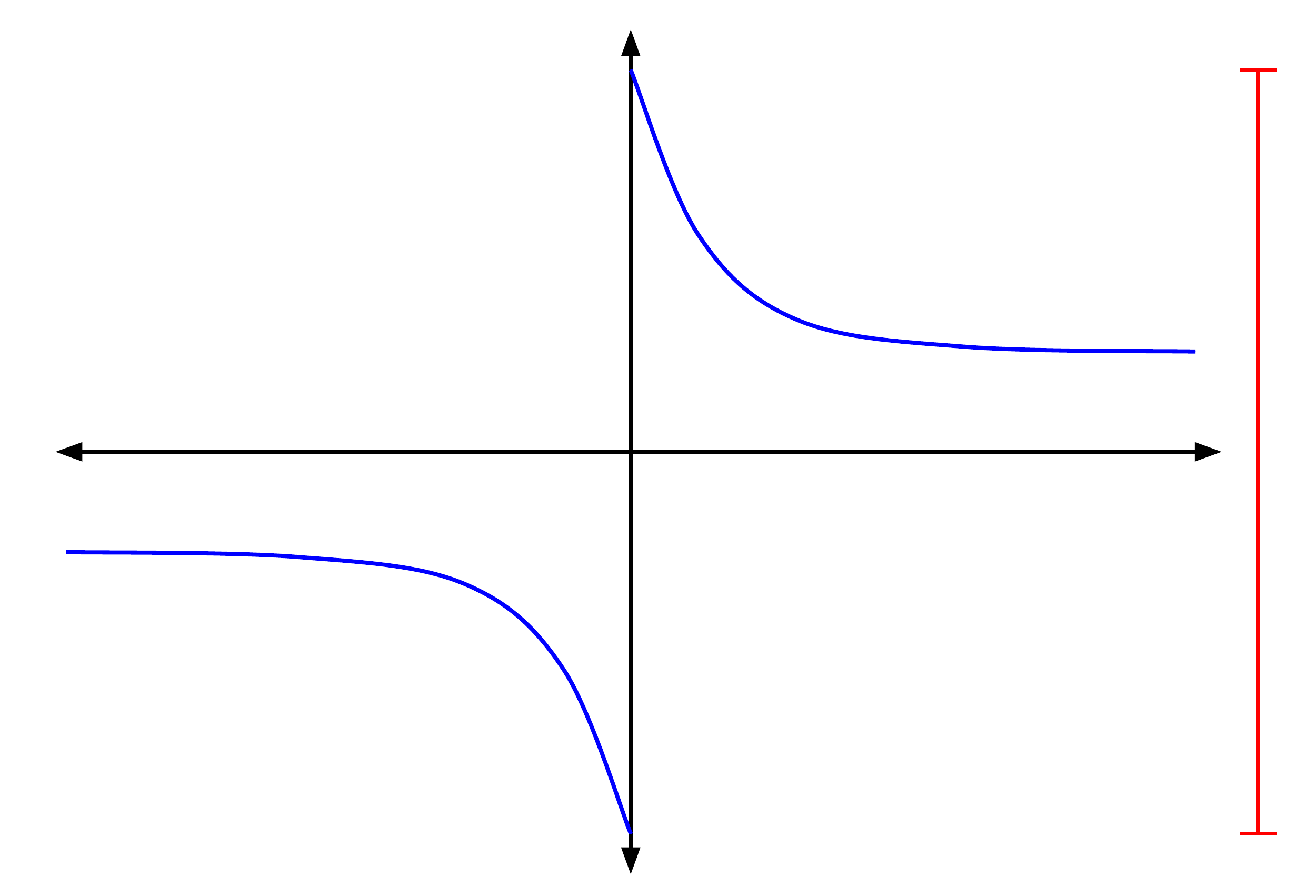}
	\put(4,29){$x$}
	\put(58, 48){\color{blue} $K$}
	\put(97,30){\color{red} $-J$}
\end{overpic}
\caption{A cartoon of $K$ when (left) $K\in L^1(\R)$ and is increasing everywhere except at $x = 0$ and (right) when $K \in L^\infty(\R)$ and is decreasing everywhere except at $x = 0$.}
\label{fig:K}
\end{center}
\end{figure}
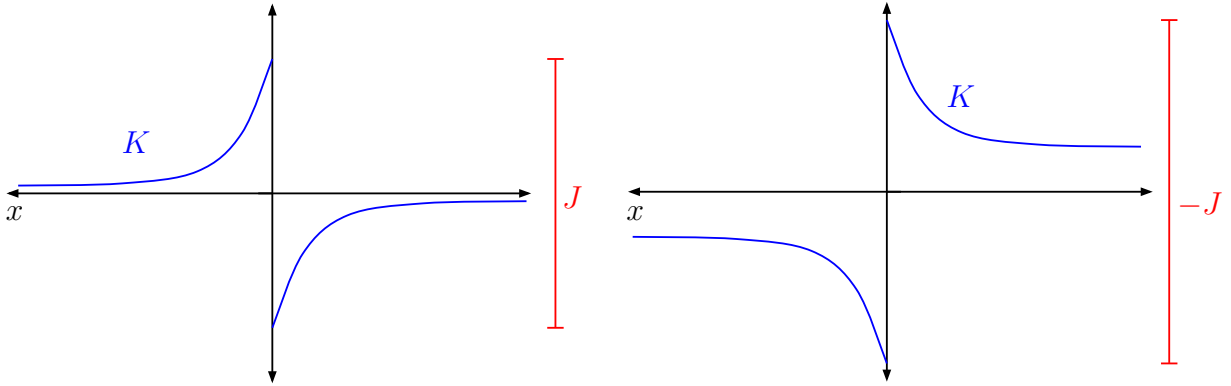

Our motivation for considering models of this type is two-fold. Firstly, reaction-advection-diffusion models where the advection is a non-linear term have garnered interest recently, see e.g.~\cite{ BKV, BerestyckiConstantinRyzhik, ConstantinLewickaRyzhik,CRRV, Crooks,CrooksMascia, Henderson_Boussinesq, Lewicka, LewickaMucha, MalhamXin, TexierPicardVolpert,  VCKRR, VladimirovaRosner} for a sampling of such works.  However, in many cases, precise bounds on propagation are intractable.  As such,~\eqref{eq:the_equation} serves as a toy model which may provide insight into other models.  Secondly, when $K(x) = -\chi\sign(x) e^{-|x|/\sqrt{d}}/(2d)$ for some positive constants $\chi$ and $d$ (as in the left part of~\Cref{fig:K}), the above is the Keller-Segel-Fisher model dating back at least to~\cite{Tello,TelloWinkler}. This model, which can be written as
\begin{equation}\label{ks}\left\{\begin{array}{rcl}
u_t+\chi\,(uv_x)_x-u_{xx} & = & u(1-u),\vspace{3pt}\\
-d\,v_{xx}+v & = & u,\end{array}\right.
\end{equation}
has been the subject of considerable interest recently.  Since there have been numerous works regarding questions of regularity and stability for different choices of domains and dimensions, we point the interested to~\cite{Tello,TelloWinkler,Winkler} and the many works which cite these.  Closer to our interest, there have been a few works studying the propagation speeds for Keller-Segel-Fisher systems~\cite{NadinPerthameRyzhik,SalakoShen1,SalakoShen2}.   The case where $\chi < 0$ in~\eqref{ks} corresponds to negative chemotaxis, where a species secretes a chemorepellent causing the species to spread.  This behavior has been observed in, for example, slime molds~\cite{KeatingBonner} and {\em entamoeba histolytica}, the bacteria that causes dysentery~\cite{ZakiAndrewInsall}.  Mathematically this behavior has attracted less interest due, perhaps, to the ease with which well-posedness may be established.  However, we mention the more applied works of~\cite{BenAmar, Grima, Sengupta}, showing pattern formation and the effect on diffusion.

Our goal is to obtain explicit propagation bounds for solutions to~\eqref{eq:the_equation}. When $K \in L^1(\R)$, we obtain a lower bound on the propagation of $2t$. While we believe that this is sharp, obtaining a matching bound is still an open question.  We obtain an upper bound of the form $c^*t$ where $c^*$ is given explicitly by the expression~\eqref{eq:c^*} in Corollary~\ref{cor:c^*} below. After \Cref{thm:localized_advection_bound} we discuss heuristically why we believe that the front is at $2t$ asympotically, and after \Cref{cor:c^*}, we discuss the relationship of our bounds on $c^*$ with known results. When $K(x) \sim x^{-\alpha}$ as $x\to+\infty$ for some $\alpha \in (0,1)$, we show that the front is at $O(t^{1/\alpha})$, and when $\lim_{x\to+\infty} K(x) > 0$, the front moves exponentially.

Throughout the paper, the solutions $u$ are understood as classical $C^{1;2}_{t;x}((0,+\infty)\times\R)$ solutions, with $u(t,\cdot)\to u_0$ as $t\to0^+$ in $L^1(\R)$. As a matter of fact, if $u$ is a continuous, non-negative, bounded solution in $(0,T)\times\R$ for some $T>0$ and if $K$ belongs to $L^1(\R)$, then the general assumptions on $K$ imply that $K*u$ and $(K*u)_x$ belong to $L^{\infty}((0,T)\times\R)$ with
\begin{equation}\label{eq:convolution_bounds_prelim}
\begin{split}
	&\|K*u\|_{L^{\infty}((0,T)\times\R)}\le \|K\|_{L^1(\R)}\|u\|_{L^{\infty}((0,T)\times\R)}
		\qquad \text{ and }\\
	&\|(K*u)_x\|_{L^{\infty}((0,T)\times\R)}\le 2|J|\,\|u\|_{L^{\infty}((0,T)\times\R)}.
\end{split}
\end{equation}
Therefore, standard parabolic estimates imply that, for every $\varepsilon\in(0,T)$, $u$ and $u_x$ are H\"older continuous in $(\varepsilon,T)\times\R$ and then that $u$ is a classical solution in $(0,T)\times\R$, with $u$, $u_t$, $u_x$ and $u_{xx}$ bounded and H\"older continuous in $(\varepsilon,T)\times\R$ (see also Section~\ref{sec:well_posed} below in the case when $K$ is only in $L^{\infty}(\R)$). The strong maximum principle then yields $u>0$ in $(0,T)\times\R$ (see also Lemma~\ref{lemmax} below). The same conclusions hold in $(0,+\infty)\times\R$ if $u$ is assumed to be continuous and bounded in $(0,+\infty)\times\R$. We note that, with the positivity of $u$ established, we may optimize the estimates in~\eqref{eq:convolution_bounds_prelim}
\begin{equation}\label{eq:convolution_bounds}
\begin{split}
	&\|K*u\|_{L^{\infty}((0,T)\times\R)}\le \frac{1}{2}\|K\|_{L^1(\R)}\|u\|_{L^{\infty}((0,T)\times\R)}
		\qquad \text{ and }\\
	&\|(K*u)_x\|_{L^{\infty}((0,T)\times\R)}\le |J|\,\|u\|_{L^{\infty}((0,T)\times\R)}.
\end{split}
\end{equation}
To obtain~\eqref{eq:convolution_bounds} we use explicitly the cancellations due to the monotonicity and positivity properties of $K$ and $u$ in order to pick up an extra factor of $1/2$ in both terms.

Lastly, for a function $g:E\to\R^k$ defined on a subset $E$ of $\R^m$ and for $p\in[1,\infty]$, we denote $\|g\|_p$ the $L^p(E)$ norm of the Euclidean norm $|g|$ of $g$. For instance, if $u$ is a bounded solution of~\eqref{eq:the_equation}, then $\|u\|_{\infty}$ stands for $\|u\|_{L^{\infty}((0,+\infty)\times\R)}$ while $\|u_0\|_{\infty}$ stands for $\|u_0\|_{L^{\infty}(\R)}$. When there is ambiguity, we write the domain explicitly.


\subsection{The effect of $L^1$ kernels}

Our first main result discusses the case when $K\in L^1(\R)$ as well as the particular case when $K = \overline{K}'$ for some kernel $\overline K \in W^{1,1}(\R)$, meaning that $K$ converges to $0$ fast enough at $\pm\infty$ (we recall that $K$ is always assumed to be odd and monotonic except at $0$). In this case, we have the following result.

\begin{theorem}\label{thm:finite_speed}
Suppose that $K\in L^1(\R)$ is odd and monotonic except at~$0$ and that $u$ is a bounded classical solution of~\eqref{eq:the_equation} with non-negative and non-zero compactly supported initial data $u_0$ satisfying~\eqref{hypu0}. Then $u>0$ in $(0,+\infty)\times\R$ and the following holds:
\begin{enumerate}[label=(\roman*)]
		\item for every $c\in (0,2)$, there exists $\delta>0$ depending only on $c$, $K$ and $\|u\|_{\infty}$ such that
\begin{equation}\label{liminfdelta1}
\liminf_{t\to+\infty} \inf_{|x| < ct} u(t,x)\geq \delta;
\end{equation}
in addition, if $K$ is compactly supported, then
\begin{equation}\label{liminfdelta2}
\liminf_{t\to+\infty} \inf_{|x| < 2t - (3/2)\log(t)} u(t,x)>0;
\end{equation}
		\item if $K=\overline{K}'$ with $\overline{K}\in W^{1,1}(\R)$, then there exists $c^*>0$, given explicitly in~\eqref{eq:c^*} below, such that
\begin{equation}\label{limsup0}
\forall\,c>c^*,\quad\limsup_{t\to+\infty} \sup_{|x| > ct} u(t,x)= 0.
\end{equation}
\end{enumerate}
\end{theorem}

We make a few observations about \Cref{thm:finite_speed}.  Firstly, we point out that the well-posedness of similar models have been handled extensively, see e.g.~\cite{NadinPerthameRyzhik,SalakoShen1, SalakoShen2,Tello, TelloWinkler}.  When $J < 1$ in \Cref{thm:finite_speed}, a maximum principle argument gives immediately that
$$\|u\|_\infty \leq \max\{1,(1-J)^{-1}\},$$
see Lemma~\ref{lemmax} below. While this relies on the fact that $\|u_0\|_\infty \leq 1$, we note that the condition $\|u_0\|_\infty \leq 1$ is not restrictive; indeed, if $\|u_0\|_\infty>1$, it is straightforward to replace the above with $$\limsup_{s\to+\infty}\|u\|_{L^{\infty}((s,+\infty)\times\R)}\le\max\{1,(1-J)^{-1}\}.$$
Uniform bounds in $L^\infty$ have been obtained in various other settings.  In particular, solutions to the Keller-Segel-Fisher model in $\R$ are always uniformly bounded in $L^\infty$.  We also expect a similar result to hold for the general kernels considered in this paper, but as this is not the focus of our study, we do not address it here.

Secondly, we mention that the lower bound~(i) in \Cref{thm:finite_speed} does not require the monotonicity of $K$.  Indeed, it requires only that $K \in L^1(\R)$ and that $u$ is bounded in $(0,+\infty)\times\R$.

Thirdly, we mention that the upper bound~\eqref{limsup0} in \Cref{thm:finite_speed}-(ii) does not require the assumption that $K = \overline K'$ for some $\overline K\in W^{1,1}(\R)$.  If $K$ were merely $L^1$ our proof still provides an upper bound; however, the expression characterizing $c^*$ (see~\eqref{eq:c^*} below) is the minimum of three terms, two of which depend on $\|\overline K\|_1$.  Hence, in many cases, we obtain better bounds on $c^*$ by considering this more specific setting.  We note that the choice of $K$ in the Keller-Segel-Fisher model satisfies this condition; indeed, $K(x) = - \chi \sign(x) e^{-|x|/\sqrt d}/(2d) = (d/dx)(\chi e^{-|x|/\sqrt d}/(2\sqrt d)$.

Finally, when $J < 1/2$, we may apply arguments from~\cite{SalakoShen1, SalakoShen2} to bootstrap our lower bound to give uniform convergence to~$1$ (that is,~\eqref{liminfdelta1} may be replaced by $\lim_{t\to\infty} \sup_{|x|< ct}|u(t,x) - 1|= 0$ and a similar, though more difficult to state, result may replace~\eqref{liminfdelta2}). This is handled in \Cref{cor:convergence_to_1} below.

\subsubsection*{A general result for the lower bounds in Theorem~\ref{thm:finite_speed}}

In order to prove \Cref{thm:finite_speed}, we prove some general results, the first of which is a lower bound in $\R^n$ that shows that a ``localized'' non-local advection term cannot slow down propagation. In the sequel, for any $x\in\R^n$ and $r>0$, we denote $B_r(x)$ the open Euclidean ball with center $x$ and radius $r$.

\begin{theorem}\label{thm:localized_advection_bound}
Let $F: C(\R^n)\cap L^{\infty}(\R^n) \to (C^1(\R^n)\cap W^{1,\infty}(\R^n))^n$ have the property that, for any $R > 0$ sufficiently large, there exist two constants $C_R>0$ and $\epsilon_R>0$, such that $\lim_{R\to+\infty}\epsilon_R=0$ and
\begin{equation}\label{hypF}
\|F(v)\|_{L^\infty(B_R(x_0))} + \|\nabla \cdot F(v)\|_{L^\infty(B_R(x_0))}\leq C_{R} \|v\|_{L^\infty(B_{2R}(x_0))} + \epsilon_R\|v\|_{L^{\infty}(\R^n)}
\end{equation}
for all $v\in C(\R^n)\cap L^{\infty}(\R^n)$ and $x_0\in\R^n$. We also assume that $F$ is locally H\"older continuous, in the sense that, for every $M>0$, there exist $\alpha\in(0,1)$ and $C>0$ such that $\|F(v)-F(w)\|_{(W^{1,\infty}(\R^n))^n}\le C\|v-w\|_{L^\infty(\R^n)}^\alpha$ for all $v,\,w\in C(\R^n)\cap L^{\infty}(\R^n)$ with $\|v\|_{L^{\infty}(\R^n)}\le M$ and $\|w\|_{L^{\infty}(\R^n)}\le M$. Let $u:(0,+\infty)\times\R^n\to\R$ be a non-negative, non-zero, bounded classical solution of
\begin{equation}\label{eqF}
u_t + \nabla \cdot (F(u(t,\cdot))\,u) = \Delta u + u(1-u), \quad\text{ in } (0,+\infty)\times \R^n.
\end{equation}
Then, for every $c\in(0,2)$, there exists $\delta>0$ depending only on $c$, $F$, and $\|u\|_{\infty}$ such that
\begin{equation}\label{propagation1}
\liminf_{t\to+\infty} \inf_{|x| < ct} u(t,x)\geq \delta.
\end{equation}
Furthermore, if the assumption~\eqref{hypF} holds with $\epsilon_R = 0$, then
\begin{equation}\label{propagation2}
\liminf_{t\to+\infty} \inf_{|x| < 2t - (n+1/2)\log(t)} u(t,x)>0.
\end{equation}
\end{theorem}

We expect \Cref{thm:localized_advection_bound} and its technique to be useful in attacking other problems where the advection is non-local. We believe that, in the positive chemotaxis setting, i.e.~when $K$ is increasing except at $x=0$, $2$ is the sharp propagation speed.  The reason being that, for $x > 0$, at the front and beyond $u$ is, on average, decreasing, in which case $K*u$ is, on average, negative and so should {\em slow} down the front compared to the same model without advection.  Since the Fisher-KPP equation moves at speed $2$ without advection, we expect propagation in our model to be {\em no faster than} $2$.  In view of~\Cref{thm:localized_advection_bound}, we expect $2$ to be the sharp rate.

Our proof of \Cref{thm:localized_advection_bound} hinges on two novel observations.  First, when $\epsilon_R > 0$, we build a small function $\underline u$ satisfying a similar equation to $u$ but where the advection $F(u(t,\cdot))(x)$ is normalized whenever $|F(u(t,\cdot))(x)|$ is not small.  We may then apply the results of~\cite{BerestyckiHamelNadin} to characterize the propagation of $\underline u$.  On the other hand, when $u(t,x)$ is small enough that $u$ and~$\overline u$ could possibly ``touch'', we use a local-in-time Harnack inequality (see~\cite[Theorem~1.2]{BHR_LogDelay} and see also Lemma~\ref{lem:harnack} below) along with the inequality that $F$ satisfies to show that $F(u(t,\cdot))(x)$ must be small enough to be below the normalization cut-off.  Hence, at this point $u$ and $\underline u$ satisfy the same equation and so the maximum principle rules out this ``touching,'' thus preserving the ordering of $u$ and $\underline u$. As a result, the front of $\underline u$, which travels at speed $2 - \epsilon$, must sit behind the front of $u$, providing the lower bound~\eqref{propagation1}.

The second observation is that, when $\epsilon_R = 0$, we may generalize the local-in-time Harnack inequality \cite[Theorem~1.2]{BHR_LogDelay} to obtain a bound on $|\nabla u(t,x)|$ by a term involving $u(t,x)^{1/p}$ for any $p>1$.  This allows us to absorb the term $\nabla \cdot(F(u(t,\cdot))u)$ into the reaction term so that $u$ is the super-solution to an equation of the form $\underline u_t = \Delta \underline u + f(\underline u)$.  The propagation of $\underline u$ is well-understood and it leads to the desired lower bound~\eqref{propagation2} for $u$.  The improved local-in-time Harnack inequality is stated precisely in \Cref{lem:harnack}.

Before discussing the second general result used in the proof of Theorem~\ref{thm:finite_speed}, which gives the upper bound on the propagation speed, we mention other known lower bounds on the propagation speed. First, in~\cite{NadinPerthameRyzhik}, Nadin, Perthame, and Ryzhik construct a travelling wave solution of~\eqref{ks} (that is,~\eqref{eq:the_equation} with the choice $K(x) = -\chi \sign(x) e^{-|x|/\sqrt{d}}/(2d)$) with a speed $c$ such that $c \in [2, 2 + \chi\sqrt{d}/(d-\chi)]$ when $0<\chi < \min\{1,d\}$.  Their lower bound matches ours; however, their strategy, which amounts to understanding the tail of the traveling wave, does not apply to the Cauchy problem.  Indeed, when they investigate the Cauchy problem, they just show that
\[
	K_1
		\leq \liminf_{t\to+\infty}\, \frac{1}{t}\int_\R u(t,x)dx
		\leq \limsup_{t\to+\infty}\, \frac{1}{t}\int_\R u(t,x)dx
		\leq K_2
\]
for some $0< K_1, K_2$, which are not obtained explicitly.  This shows linear-in-time propagation in a weak sense, but it does not provide an explicit bound of the propagation.  Later, Salako and Shen~\cite{SalakoShen1, SalakoShen2} are able to obtain a lower bound on the propagation of the Cauchy problem when $d=1$ and $0<\chi < 2/(3 + \sqrt{2})$ of $2\sqrt{1 - \chi(1-\chi)^{-1}} - \chi(1-\chi)^{-1}$.  In both cases, our bound for the propagation of the Cauchy problem is an improvement.

\subsubsection*{Some heat kernel estimates for the upper bounds in Theorem~\ref{thm:finite_speed}}

The second general result, used in the proof of Theorem~\ref{thm:finite_speed}-(ii), is an upper bound on the heat kernel when the operator takes a similar form to that of~\eqref{eq:the_equation}.

\begin{prop}\label{prop:heat_kernel}
Suppose that $v:(0,+\infty)\times\R\to\R$ is of class $C^{1;2}_{t;x}((0,+\infty)\times\R)$ and satisfies the bounds
\begin{equation}\label{boundsv0}
	\|v\|_\infty \leq A_0, \quad
	\|v_x\|_\infty \leq A_1, \quad \text{ and }\quad
	\|v_{xx}\|_\infty \leq A_2,
\end{equation}
for some $A_0, A_1, A_2 \in[0,+\infty)$.  Consider the fundamental solution $\Gamma(t,s,x,y)$ to the equation
\begin{equation}\label{eqGamma}
\begin{cases}
	\Gamma_t + (v_x \Gamma)_x = \Gamma_{xx}, \qquad &\text{ in } (s,+\infty)\times\R,\\
	\Gamma(t=s,s,x,y) = \delta_y(x),
\end{cases}
\end{equation}
for any $s\ge0$ and $y\in\R$. Then, for every $\delta>0$, there is a constant $C_\delta>0$ depending on $\delta$, such that $\Gamma$ satisfies the upper bound
\begin{equation}\label{eq:kernel_bound}
\begin{split}
	\Gamma(t,s,x,y) \leq \frac{C_\delta}{\sqrt{t\!-\!s}} \exp \Bigg\{\delta(t\!-\!s)\!+\!
		\min\Bigg[&\frac{A_2(t-s)}{2} \!-\! \frac{(x\!-\!y)^2}{4(1+ \delta + A_0^2)(t\!-\!s)},\\
			&\!\inf_{\epsilon\in(0,1)}\!\!\Bigg(\!\!\frac{A_1^2(t\!-\!s)}{4\epsilon}\! -\! \frac{(x\!-\!y)^2}{4\big(1\!+\!\delta\!+\!\frac{A_0^2}{1\!-\!\epsilon}\big)(t\!-\!s)}\Bigg)\Bigg]\Bigg\}
\end{split}
\end{equation}
for every $t>s\ge0$ and $(x,y)\in\R\times\R$.
\end{prop}

The proof of the upper bound~\eqref{eq:kernel_bound} in~\Cref{prop:heat_kernel} follows the general outline of Fabes and Stroock's~\cite{FabesStroock} proof for heat kernel estimates in $\R^n$ for second order parabolic equations without first and zeroth order terms.  Their proof involves looking at exponentially weighted solutions to the equation and obtaining a general $L^{2^{k+1}}$ estimate in terms of the $L^{2^k}$ norm for all $k\in\N$.  Taking $k$ to infinity yields an $L^\infty$ bound dependent on the growth of the $L^2$ norm, which comes from the estimate when $k = 1$.  Here, we observe that the rate of exponential decay of the eventual estimate comes {\em only} from the $L^2$ estimate.  As such, we optimize this estimate in a way which we may not necessarily do for the general $L^{2^k}$ estimate.  This allows us to obtain estimates which take advantage of the three norms of $v$ and obtain spatial decay which is better than the na\"ive bound.  We note that, while the above estimate is, in general, not sharp, it is quite general.

The last general result gives a bound on the tails of solutions.  As we observe below, it translates to the na\"ive bound on the propagation of the sum of the speeds of propagation for homogeneous Fisher-KPP and of advection, that is $c^* \leq 2 +\|K*u\|_\infty$.  Beyond this, it is useful in many places throughout this manuscript, especially in the setting where $K \notin L^1(\R)$.

\begin{prop}\label{prop:fokker_planck}
	Suppose that $T>0$, $u \in C^{1;2}_{t;x}((0,T]\times \R)$, $v \in C^{0;1}_{t;x}((0,T]\times\R)$, and
	\begin{equation}\label{eq:kfp}
	\begin{cases}
		u_t + (v u)_x = u_{xx} \qquad &\text{ in } (0,T]\times \R,\\
		u(0,x) = u_0(x) \geq 0 &\text{ in } \R,
	\end{cases}
	\end{equation}
	where $u_0 \in L^\infty(\R)$ is non-negative, non-zero, and satisfies $\supp u_0 \subset [-a,a]$ for some $a > 0$. Suppose, further, that $v$ satisfies
	\[
		\sup_{t\in(0,T]}\,\|v(t,\cdot)\|_{L^\infty(\R)}  \leq A,
	\]
	for some $A \in(0,+\infty)$.  Then, for all $|x| \geq A\,T + a + 1$, 
	\[
		u(T,x) \leq \frac{a}{\sqrt{\pi}}\left(\frac{1}{\sqrt{T}} + \frac{A\,\sqrt{T}}{|x|-A\,T-a}\right)e^{-\frac{(|x|-A\,T-a)^2}{4T}}.
	\]
\end{prop}

One could prove \Cref{prop:fokker_planck} in a rather straight-forward manner by utilizing the probabilistic interpretation of the equation.  However, the proof of the local-in-time Harnack inequality requires upper and lower heat kernel estimates that match asymptotically as $|x| \to \infty$.  Since we obtain these estimates from a result of Hill~\cite{Hill}, we utilize this result to prove \Cref{prop:fokker_planck} instead of proving \Cref{prop:fokker_planck} directly.

The bound on $c^*$ follows directly from \Cref{prop:heat_kernel,prop:fokker_planck}, along with~\eqref{eq:convolution_bounds}.

\begin{corollary}\label{cor:c^*}
	Under the assumptions of Theorem~$\ref{thm:finite_speed}$, the speed $c^*$ in~\eqref{limsup0} can be given by the following explicit formula:
\begin{equation}\label{eq:c^*}
\begin{split}
	c^*
		= \min\Bigg\{
			&2 + \frac{\|K\|_1\|u\|_\infty}{2},
			2\sqrt{\left(1 + \frac{|J|\|u\|_\infty}{2}\right)\left(1 + \|\overline K\|_1^2 \|u\|_\infty^2 \right)},\\
		&2\inf_{\epsilon \in (0,1)} \sqrt{\left(1 + \frac{\|K\|_1^2 \|u\|_\infty^2}{16\epsilon}\right)\left(1 + \frac{\|\overline{K}\|_1^2 \|u\|_\infty^2}{1-\epsilon}\right)}
		\Bigg\}.
\end{split}
\end{equation}
\end{corollary}

We make a few observations about \Cref{cor:c^*}.  Firstly, when $K=0$ (and then, $J=0$ and $\overline{K}=0$),~\eqref{eq:c^*} reduces to the well-known formula $c^*=2$ for the {\it local} Fisher-KPP equation~\cite{Fisher,KPP}.

Secondly, it is possible to derive an explicit formula for the term involving an infimum over $\epsilon$.  The resulting formula is complicated and since we are interested in estimates that are asymptotically correct, we leave the formula as above.

Thirdly, when $0\le J<1$, we may make some estimates more explicit. Since $\|u\|_\infty\le(1-J)^{-1}$ in this case, as we show later in \Cref{lemmax}, it follows that
\begin{equation}\label{eq:c^*1}
	\begin{split}
	c^*
		\le \min\Bigg\{
		&2 + \frac{\|K\|_1}{2(1-J)},
		2\sqrt{\frac{2-J}{2(1-J)}\left(1+ \frac{\|\overline K\|_1^2}{(1-J)^2}\right)},\\
		&2\inf_{\epsilon \in (0,1)} \sqrt{\left(1 + \frac{\|K\|_1^2}{16\epsilon(1-J)^2}\right)\left(1 + \frac{\|\overline{K}\|_1^2 }{(1-\epsilon)(1-J)^2}\right)}
		\Bigg\}.
	\end{split}
\end{equation}

Lastly, we point out that, for each term in the minimization characterizing $c^*$ in~\eqref{eq:c^*}, one may easily find choices of $\overline K$ for which that term is the minimum.

Since, with the unwieldy characterization of $c^*$ in~\eqref{eq:c^*1}, it is difficult to see if this is an improvement over known results, we now compare these bounds on $c^*$ with the known results.  First, in~\cite{NadinPerthameRyzhik}, as mentioned above, a traveling wave moving with speed $c \in [2, 2+ \chi \sqrt{d} / (d - \chi)]$ was constructed for $K(x)=-\chi \sign(x) e^{-|x|/\sqrt{d}}/(2d)$ and $0<\chi<\min\{1,d\}$. Given that choice of $K$, one can observe that $0<J = \chi/d<1$, $\|K\|_1 = \chi/\sqrt{d}$, and $\|\overline K\|_1 = \chi$.  We see improvement in a few regimes.  The first term in the right-hand side of the upper bound for~$c^*$~\eqref{eq:c^*1} is essentially the same as the upper estimate $2+ \chi \sqrt{d} / (d - \chi)$ in~\cite{NadinPerthameRyzhik}.\footnote{They are the same up to the $1/2$ factor which is a mere oversight in~\cite{NadinPerthameRyzhik}.} However, the second term is asymptotically better, for example, if $\chi = 1/d$ and $\chi \to 0$: the upper bound in~\cite{NadinPerthameRyzhik} is asymptotic to $2 + \chi^{3/2}$ while ours is, using the middle term in~\eqref{eq:c^*1}, at most asymptotic to $2 + 3\chi^2/2$ (the third term in~\eqref{eq:c^*2} would also lead to a similar asymptotics). The third term is also asymptotically better in the case where, for example, $\chi=d^2$ and $d\to0$. In this case, the upper bound in~\cite{NadinPerthameRyzhik} is asymptotic to $2 +d^{3/2}$ while ours is at most asymptotic to $2 + d^3/16$, after choosing $\epsilon = 1-\sqrt d$. We recall again that no explicit bounds are given for the speed of propagation in the Cauchy problem in~\cite{NadinPerthameRyzhik}.  On the other hand, it is difficult to compare with the results of Salako and Shen~\cite{SalakoShen1, SalakoShen2} as they provide two upper bounds: the first is equivalent to the bound of Nadin, Perthame, and Ryzhik~\cite{NadinPerthameRyzhik} discussed above, and the second is characterized in terms of the solution to an algebraic equation which is not explicitly solvable. However, in the case when $d$ is fixed such that $d>1$, which in the notation of~\cite{SalakoShen1,SalakoShen2} corresponds to $a>1$, this second upper bound converges to $1 + d>2$ as $\chi \to 0$, while ours converges to $2$.  Hence, it follows that our bound is an improvement of the one of~\cite{SalakoShen1,SalakoShen2} in some regimes.

Before proceeding to the discussion of the effect of large kernels, we note that we do not address the problem of the existence of travelling waves.  We believe that the procedure of Nadin, Perthame, and Ryzhik~\cite{NadinPerthameRyzhik} is sufficiently robust to handle our case as well without the addition of new ideas.


\subsection{The effect of large ``negative chemotaxis'' kernels}

One may ask what happens when $K$ is not in $L^1(\R)$.  In the case where $K$ is increasing except at $x = 0$ and is merely $L^\infty(\R)$, this question was investigated by Mimura and Ohara~\cite{MimuraOhara} who found kernels which admitted standing waves.  That is, certain kernels $K$ can {\em stop} propagation.  On the other hand, we ask what happens in the case of decreasing kernels which are not in $L^1(\R)$.

First, we address the well-posedness of the problem.

\begin{prop}\label{prop:well_posed}
If $K$ is an odd kernel in $L^\infty(\R)$ which is non-increasing except at $0$ and is non-positive in $(-\infty,0)$ and non-negative in $(0,+\infty)$, then there exists a unique bounded classical solution $u$ that solves~\eqref{eq:the_equation}-\eqref{hypu0} and that satisfies $0 < u(t,x)\le1$ for all $(t,x) \in (0,+\infty)\times \R$.
\end{prop}

Next, we define
\begin{equation}\label{defP}
P(t) = \int_\R u(t,x)\,dx
\end{equation}
to be the total population at time $t\ge0$.  In the Fisher-KPP equation $u_t=u_{xx}+u(1-u)$, if the initial condition $u_0$ is non-negative, non-zero, bounded and compactly supported, it is known that $P(t)/(ct) \to 2$  as $t\to+\infty$, where $c$ is the minimal traveling wave speed, which is known to be $c = 2$.  We use the notation with $c$ to emphasize what role the speed plays in the asymptotics of $P$.  See~\Cref{fig:P} for a visual depiction of this.
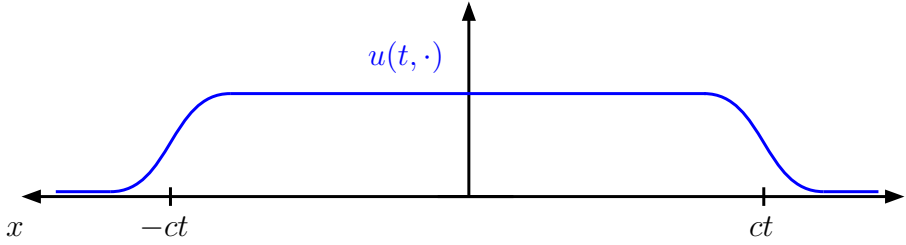
\begin{figure}[t]
\begin{center}
\begin{overpic}[scale=.5]
	{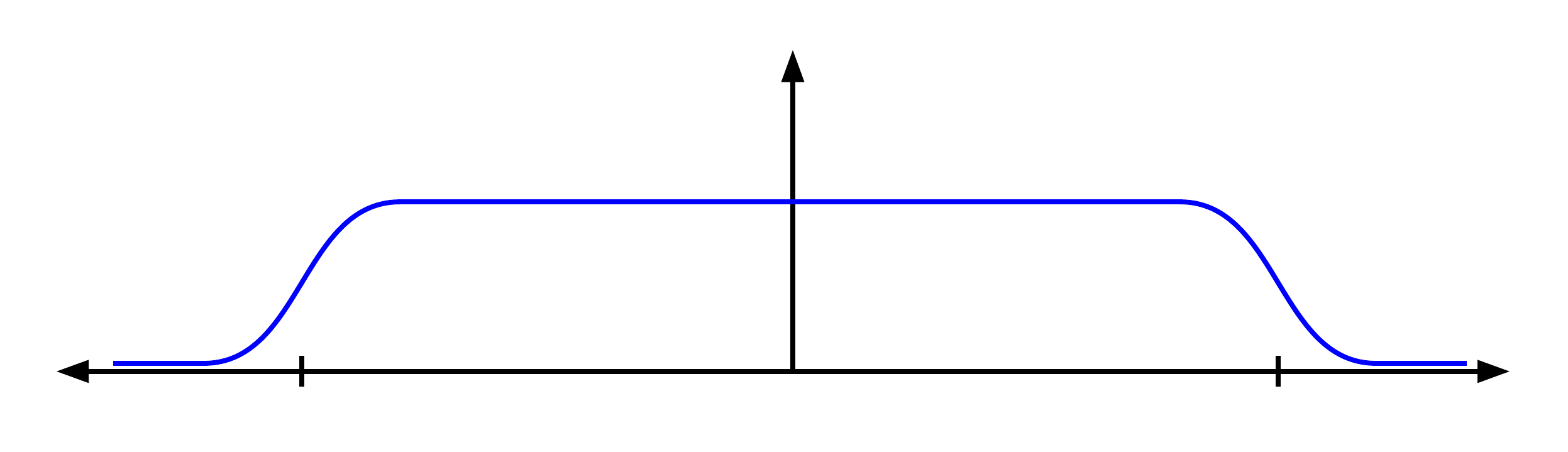}
	\put(2,1){$x$}
	\put(16, 1){$-ct$}
	\put(80,1){$ct$}
	\put(40,19){\color{blue} $u(t,\cdot)$}
\end{overpic}
\caption{A cartoon of $u$.  Notice that $P(t)$ is the area under the blue line that represents $u$.  From this image it is clear that $P(t) \sim 2ct$.}
\label{fig:P}
\end{center}
\end{figure}
This is a key point in our analysis.  In the theorem below and in the sequel, $\R^* := \R\setminus\{0\}$.

\begin{theorem}\label{thm:acceleration}
Suppose that $u$ is a bounded classical solution of~\eqref{eq:the_equation}-\eqref{hypu0} with $\supp(u_0) \subset [-a,a]$ for some $a>0$ and with an odd kernel $K \in L^\infty(\R)$ that is non-increasing except at~$0$ and is non-negative on $(0,+\infty)$ $($and then non-positive on $(-\infty,0))$. Then:
\begin{enumerate}[label=(\roman*)]
	\item if $\lim_{x\to+\infty} K(x) > 0$, there exists $r\in(0,1)$, depending only on $K$, and a constant $C\ge1$, depending only on $P(0)$ and $K$, such that
	$$C^{-1} e^{rt} \leq P(t) \leq Ce^t\ \hbox{ for all }t\ge0;$$
	\item if $K \in L^p(\R)$ with $p\in[1,\infty)$, then there exists a constant $C>0$, depending only on $u_0$, $p$, and $K$, such that
	$$P(t) \leq C\,(t^p +1)\ \hbox{ for all }t\ge0;$$
	\item if $K\in C^1(\R^*)$ is convex on $(0,+\infty)$ $($and then concave on $(-\infty,0))$ and is such that $K(x) \geq A\,(1+x)^{-\alpha}$ on $(0,+\infty)$ for some $\alpha \in (0,1)$ and $A > 0$, then there exists $C>0$, depending only on $u_0$ and $K$, such that
	$$P(t) \geq C(1+t)^{1/ \alpha}\ \hbox{ for all }t\ge0.$$
\end{enumerate}
\end{theorem}

We point out that the cases (ii) and (iii) complement each other.  They match in the sense that if, in addition to the other assumptions on the kernel $K$, it satisfies $K(x) \approx x^{-\alpha}$ as $x\to+\infty$ with $\alpha\in(0,1)$, then $K \in L^{(1/\alpha) +\epsilon}(\R)$ for all $\epsilon>0$ so we obtain that $t^{1/\alpha} \lesssim P(t) \lesssim t^{(1/\alpha)+ \epsilon}$ as $t\to+\infty$, for all $\epsilon>0$.

We note that these estimates are obtained by deriving differential inequalities on the quantities $P(t)$ and $V(t) := \int_\R u(t,x)\,(1-u(t,x))\,dx$, related by the equation $P' = V$, and an estimate on the derivative $(K*u)_x$.  Obtaining propagation speed estimates via integral quantities goes back at least to~\cite{ConstantinKiselevObermanRyzhik} and seems to be quite well suited to understanding acceleration phenomena in situations where there is no comparison principle as in this work.  It is interesting to note that the superlinear acceleration in~\Cref{thm:acceleration} is due entirely to advection.  To our knowledge, this has not been observed in other reaction-diffusion-advection models previously.

These estimates may be bootstrapped to obtain pointwise bounds.  We obtain the following corollary.

\begin{corollary}\label{cor:acceleration_pointwise}
Suppose that $u$ is a bounded classical solution of~\eqref{eq:the_equation}-\eqref{hypu0} with $\supp(u_0) \subset [-a,a]$ for some $a>0$ and with an odd kernel $K \in L^\infty(\R)\cap C^1(\R^*)$ that is non-negative and convex on $(0,+\infty)$.
\begin{enumerate}[label=(\roman*)]
\item If $\lim_{x\to+\infty}K(x)>0$ and if $u_0$ is even, radially non-increasing and of class $C^{2+\beta}(\R)$ for some $\beta\in(0,1)$, then there exists $\lambda>0$ such that
\[
	\lim_{t\to+\infty} \inf_{|x| < e^{\lambda t}} u(t,x) >0.
\]
\item If there exist $A\ge1$ and $\alpha \in(0,1)$ such that $A^{-1} (1+x)^{-\alpha} \leq K(x) \leq A (1+x)^{-\alpha}$ for all $x\in\R$, then there exists a constant $C_0>0$, depending on $u_0$ and~$K$, such that, for any $\mu\in(0,1)$,
\[
	\liminf_{t\to+\infty}\, t^{-1/\alpha}\left(\frac{1}{t}\int_{0}^t \big|\{ x\in\R : u(s,x) \geq\mu\}\big|\,ds\right)\ge C_0.
\]
\end{enumerate}
\end{corollary}

This is proved in two parts: \Cref{cor:acceleration_pointwise1} and \Cref{cor:acceleration_pointwise2}. We note that, in Corollary~\ref{cor:acceleration_pointwise}, if $K(x) \approx x^{-\alpha}$ as $x\to+\infty$ and $u_0$ is even, radially non-increasing and of class $C^{2+\beta}(\R)$ for some $\beta\in(0,1)$, we may easily bootstrap the results of \Cref{cor:acceleration_pointwise} to obtain that
\[
	\lim_{t\to+\infty} \sup_{|x|<C_1 t^{1/\alpha}} |u(t,x) - 1| = 0
\]
for some $C_1>0$, by using the fact that $u(t,x)$ is even in $x$ and non-increasing in $|x|$ for all $t > 0$.


\subsection*{Organization of the paper}

In \Cref{sec:finite_speed}, we prove~\Cref{thm:finite_speed} through the results~\Cref{thm:localized_advection_bound} and \Cref{prop:heat_kernel} up to technical lemmas that we outsource to \Cref{sec:lemmas}.  We also show that the lower bound in \Cref{sec:finite_speed} may be upgraded to convergence to 1 when $J < 1/2$ in \Cref{sec:finite_speed}. In \Cref{sec:acceleration}, we prove \Cref{thm:acceleration} and \Cref{cor:acceleration_pointwise}, including in \Cref{sec:well_posed} a discussion of {\em a priori} bounds for the negative chemotaxis model when $K$ is in $L^{\infty}(\R)$ but may not be in $L^1(\R)$.


\section{Bounds on the speed of propagation for the Cauchy problem}\label{sec:finite_speed}

In this section, we prove the bounds in \Cref{thm:finite_speed}.  We prove the upper bound and the lower bound in \Cref{sec:upper_bound} and \Cref{sec:lower_bound}, respectively. To begin, we show in Section~\ref{sec21} some explicit uniform bounds of $u$ when $J < 1$.


\subsection{Explicit pointwise bounds}\label{sec21}

\begin{lem}\label{lemmax}
Let $K$ and $u$ be as in Theorem~$\ref{thm:finite_speed}$. Then $u(t,x)>0$ for all $t>0$ and $x\in\R$. Furthermore, if $J < 1$,
\begin{equation}\label{ineqmax1}
u(t,x)\leq \max\big\{1,(1-J)^{-1}\big\}\qquad \hbox{ for all }t>0\hbox{ and }x\in\R.
\end{equation}
\end{lem}

\begin{proof}
First of all, the assumptions on $K$ and $u$ imply that the functions $a:=K*u$ and $b:=(K*u)_x$ belong to $L^{\infty}((0,+\infty)\times\R)$ with $\|a\|_\infty\le\|K\|_1\|u\|_\infty/2$ and $\|b\|_\infty\le \,|J|\,\|u\|_\infty$ as noted in~\eqref{eq:convolution_bounds}. The function $u$ can then be written as a bounded classical solution of
$$u_t=u_{xx}-a(t,x)\,u_x-b(t,x)\,u(t,x)+u(1-u)$$
in $(0,+\infty)\times\R$. As already emphasized in Section~\ref{sec:results}, since $u_0$ is non-negative and non-zero, it follows from the maximum principle that $u(t,x)>0$ for all $t>0$ and $x\in\R$. Remember also that, from standard parabolic estimates, the functions $u$, $u_t$, $u_x$ and $u_{xx}$ are then H\"older continuous in $(\epsilon,+\infty)$ for every $\epsilon>0$.\par
Fix any $T>0$.  Since $u_t\le u_{xx}-a(t,x)\,u_x+Cu$ in $(0,+\infty)\times\R$ for some constant $C\ge0$, there holds
\begin{equation}\label{boundu}
u(t,x)\le\|u_0\|_\infty e^{Ct}\ \hbox{ for all }t>0\hbox{ and }x\in\R.
\end{equation}
Assume by way of contradiction that $\|u\|_{L^\infty([0,T]\times \R)}>\max\{1,(1-J)^{-1}\}$. Using the inequality $\|u\|_{L^\infty([0,T]\times \R)}>1 \geq \|u_0\|_\infty$ and equation~\eqref{boundu}, one gets the existence of a sequence $(t_n,x_n)_{n\in\N}$ in $(0,T]\times\R$ such that $u(t_n,x_n)\to\|u\|_\infty=\sup_{(0,T]\times\R}u$ as $n\to+\infty$ and $\liminf_{n\to+\infty}t_n>0$. Up to extraction of a subsequence, the functions $u_n(t,x):=u(t,x+x_n)$ converge in $C^{1;2}_{t;x;loc}((0,T]\times\R)$, to a classical solution $u_{\infty}$ of
$$
(u_\infty)_t=(u_\infty)_{xx}-a_{\infty}(t,x)\,(u_\infty)_x-b_\infty(t,x)\,u_\infty+u_\infty(1-u_\infty)
$$
in $(0,T]\times\R$, with $a_{\infty}=K*u_\infty$ and $b_\infty=(K*u_\infty)_x$.  Furthermore, for all $(t,x)\in(0,T]\times\R$, one has $0\le u_\infty(t,x)\le\|u\|_{\infty}=u_\infty(t_\infty,0)$ for some $t_\infty \in (0,T]$. At the point $(t_\infty,0)$,
\begin{equation}\label{eq:max_principle}
\begin{array}{rcl}
0 & \le & (u_\infty)_t(t_\infty,0)-(u_\infty)_{xx}(t_\infty,0)+a_\infty(t_\infty,0)\,(u_\infty)_x(t_\infty,0)\vspace{3pt}\\
& = & u_\infty(t_\infty,0)\,(1-u_\infty(t_\infty,0))-(K*u_\infty)_x(t_\infty,0)\,u_\infty(t_\infty,0).\end{array}
\end{equation}
If $J\ge0$ and $K\in L^1(\R)$ is monotonic except at $0$, $K$ is then non-decreasing in $(-\infty,0)$ and in $(0,+\infty)$. The function $u_\infty$ being itself non-negative, it follows that $-(K*u_\infty)_x(t_\infty,0)\,u_\infty(t_\infty,0)\le J\,(u_\infty(t_\infty,0))^2$, hence $0\le u_\infty(t_\infty,0)\,(1-u_\infty(t_\infty,0))+J\,(u_\infty(t_\infty,0))^2$. Since $u_\infty(t_\infty,0)=\|u\|_\infty$ was assumed to be larger than $(1-J)^{-1}\,(>0)$, one reaches a contradiction.

If $J < 0$, then similar reasoning shows that $-(K*u_\infty)_x(t_\infty,0) u_\infty(t_\infty,0) \leq 0$.  Hence~\eqref{eq:max_principle} implies $0 \leq u_\infty(t_\infty,0)(1 - u_\infty(t_\infty,0))$.  Since $u_\infty(t_\infty,0) = \|u\|_\infty$ was assumed to be larger than $1$, one reaches a contradiction.

Thus,~\eqref{ineqmax1} is proved.
\end{proof}


\subsection{The upper bound: proofs of Theorem~\ref{thm:finite_speed}-(ii) and Corollary~\ref{cor:c^*}}\label{sec:upper_bound}

\begin{proof}[Proof of part~{\rm{(ii)}} in Theorem~$\ref{thm:finite_speed}$ and Corollary~$\ref{cor:c^*}$, from Propositions~$\ref{prop:heat_kernel}$ and $\ref{prop:fokker_planck}$] We first consider the case when
$$c > 2 + \frac{\|K\|_1 \|u\|_\infty}{2}.$$
We obtain our bound by applying \Cref{prop:fokker_planck} to a suitable super-solution of $u$.  Indeed, let $\overline u$ satisfy
\begin{equation}\label{overlineu}
	\overline u_t + [(K*u)\overline u]_x = \overline u_{xx} + \overline u
\end{equation}
in $(0,+\infty)\times \R$ with $\overline u(0,\cdot) = u_0$.   Hence, together with the non-negativity of $u$, we have that $0 \leq u(t,x) \leq \overline u(t,x)$ for all $t> 0$ and all $x \in \R$.  Let $a>0$ be such that $\supp u_0 \subset [-a,a]$.  For any $t>0$, recalling~\eqref{eq:convolution_bounds} and choosing $v = K*u$ and $A = \|K\|_1\|u\|_\infty/2$, we apply \Cref{prop:fokker_planck} to $e^{-t} \overline u(t,x)$ to obtain a constant $C_K$ depending only on $K$, $\|u\|_\infty$, and $a$ such that, for all $t>1$ and $|x| \ge \|K\|_1\|u\|_\infty t/2 + a + 1$,
\[
	0\le e^{-t} u(t,x)\leq e^{-t}\overline u(t,x)\leq C_K\sqrt{t}\,e^{-(|x| - \|K\|_1 \|u\|_\infty t/2 - a)^2/4t}.
\]

Let $2\delta_c = c - 2 - \|K\|_1 \|u\|_\infty/2>0$.  Fix $T>1$ large enough such that, for all $t \geq T$, $ct > (2 + \delta_c + \|K\|_1 \|u\|_\infty /2)t + a + 1$.  Then, for $t  \geq T$ and $|x| \geq ct$, we have
\[
0\le u(t,x)\leq C_K\sqrt{t}\exp\left\{t - \frac{(2 + \delta_c)^2 t^2}{4t}\right\}=C_K\sqrt{t}\exp\left\{-\frac{(4\delta_c + \delta_c^2)t}{4}\right\}.
\]
Recalling that $\delta_c > 0$, we see that $\sup_{|x|\ge ct} u(t,x) \to 0$ as $t\to+\infty$ for $c > 2 + \|K\|_1 \|u\|_\infty/2$.

Let us now consider the case when
\begin{equation}\label{eq:c^*2}
\begin{split}
	c> \min\Bigg\{&2\sqrt{\left(1 + \frac{|J|\|u\|_\infty}{2}\right)\left(1 + \|\overline K\|_1^2 \|u\|_\infty^2 \right)},\\
		&2\inf_{\epsilon \in (0,1)} \sqrt{\left(1 + \frac{\|K\|_1^2 \|u\|_\infty^2}{16\epsilon}\right)\left(1 + \frac{\|\overline{K}\|_1^2 \|u\|_\infty^2}{1-\epsilon}\right)}
		\Bigg\}.
\end{split}
\end{equation}
The key estimate here is the heat kernel estimate in \Cref{prop:heat_kernel} that follows from a modification of the strategy of Fabes and Stroock~\cite{FabesStroock}. As in the previous paragraph, the function $\overline{u}$ solving~\eqref{overlineu} with $\overline{u}(0,\cdot)=u_0$ satisfies $0\le u(t,x) \leq \overline u(t,x)$ for all $t>0$ and all $x\in\R$.  Let $\Gamma$ be the fundamental solution given in~\eqref{eqGamma} with the choice $v(t,x) = (\overline K * u(t,\cdot))(x)$.  Thus, using also that $u_0$ has compact support included in $[-a,a]$ and ranges in $[0,1]$,

\[
		0\le u(t,x)
			\leq \overline u(t,x)
			= e^t \int_\R \Gamma(t,0,x,y)\,u_0(y)\,dy
			\leq 2\,a\, e^t \max_{y\in[-a,a]}\Gamma(t,0,x,y).
\]
As already emphasized in Section~\ref{sec:results} and~\eqref{eq:convolution_bounds}, the function $v$ is of class $C^{1;2}_{t;x}((0,+\infty)\times\R)$, with
\begin{equation}\label{boundsv}
	\|v\|_\infty\!\le A_0\!:=\|\overline{K}\|_1\|u\|_\infty,\ 
	\|v_x\|_\infty\!\le A_1\!:=\frac{\|K\|_1\|u\|_\infty}{2}
	\hbox{ and }
	\|v_{xx}\|_\infty\!\le A_2\!:=|J|\,\|u\|_\infty.
\end{equation}
It follows from formula~\eqref{eq:kernel_bound} in Proposition~\ref{prop:heat_kernel} that, for any $\delta>0$, there exists a constant $C_\delta>0$ depending on $\delta$, such that, for every $t>0$ and $x\in\R$,
$$\begin{array}{r}
\ds0\le u(t,x) \leq\frac{2\,a\,C_\delta}{\sqrt{t}}\max_{y\in[-a,a]} \exp \Bigg\{ (1+\delta)t +\min\Bigg[\ds \frac{A_2}{2}t-\frac{(x-y)^2}{4(1+\delta+A_0^2)t} ,\qquad\qquad\qquad\ \ \\
\displaystyle \inf_{\epsilon\in(0,1)}\!\!\Bigg(\frac{A_1^2t}{4\epsilon}-\frac{(x-y)^2}{4\big(1+\delta+\frac{A_0^2}{1-\epsilon}\big)t}\Bigg)\Bigg]\Bigg\}.
\end{array}$$
In the above formula, each term in the max-min is of the form
$$\frac{2\,a\,C_\delta}{\sqrt{t}}\,\exp\Big\{(1+\delta+a_1)t-\frac{(x-y)^2}{4(1 + \delta + a_2)t}\Big\}$$
for some $a_1, a_2\geq 0$. More precisely, the pairs $(a_1,a_2)$ are of the type
\begin{equation}\label{a123}
(a_1,a_2)=\left(\frac{A_2}{2},A_0^2\right),
\ds\hbox{ or }
(a_1,a_2)=\left(\frac{A_1^2}{4\epsilon},\frac{A_0^2}{1-\epsilon}\right)\hbox{ with }0<\epsilon<1.
\end{equation}
Hence, our upper bound gives us, for every $\delta>0$ and $c>0$, and for both choices $(a_1,a_2)$ above,
	\[
		0\le \sup_{|x|>ct} u(t,x)\leq \frac{2\,a\,C_\delta}{\sqrt t}\sup_{|x|>ct}\max_{y\in[-a,a]} \exp\left\{ (1+\delta+a_1)t - \frac{(x-y)^2}{4(1 + \delta + a_2)t}\right\}.
	\]
The term on the right tends to zero if
\begin{equation}\label{largespeed}
c > 2\sqrt{(1+\delta+a_1)(1 + \delta + a_2)}.
\end{equation}
Substituting the values of $(a_1,a_2)$ into the right hand side above with $(A_0,A_1,A_2)$ given in~\eqref{boundsv}, minimizing over these values $(a_1,a_2)$ and setting $\delta = 0$, yields exactly the terms in the right-hand side of~\eqref{eq:c^*2}.  If we now choose $c$ as in~\eqref{eq:c^*2}, then we may choose $\delta>0$ small enough so that~\eqref{largespeed} holds for at least one of the choices $(a_1,a_2)$ listed in~\eqref{a123}. Hence, $\sup_{|x|>ct}u(t,x)\to0$ as $t\to+\infty$ for every  $c$ satisfying~\eqref{eq:c^*2}.

The two cases above yield that $\sup_{|x|>ct} u(t,x) \to 0$ as $t \to +\infty$ for all $c> c^*$, where $c^*$ is given by~\eqref{eq:c^*}.  This finishes the proof.
\end{proof}


\subsection{The lower bounds}\label{sec:lower_bound}

In this section, we prove the lower bounds in~\Cref{thm:finite_speed}.  Afterwards, we show that when $K$ is small, these lower bounds may be bootstrapped to show that $u$ converges to $1$.

\subsubsection{Proofs of \Cref{thm:finite_speed}-(i) and Theorem~\ref{thm:localized_advection_bound}}

\begin{proof}[Proof of part~{\rm{(i)}} in Theorem~$\ref{thm:finite_speed}$, from Theorem~$\ref{thm:localized_advection_bound}$] It is immediate, by setting
$$\begin{array}{rcl}
F:C(\R)\cap L^{\infty}(\R) & \to & C^1(\R)\cap W^{1,\infty}(\R)\vspace{3pt}\\
v & \mapsto & K*v.\end{array}$$
Indeed, it is straightforward to check that $F$ is locally H\"older continuous (it is even Lipschitz continuous since $\|F(v)-F(w)\|_{W^{1,\infty}(\R)}\le(\|K\|_1+2|J|)\|v-w\|_{L^{\infty}(\R)}$ for all $v,\,w\in C(\R)\cap L^{\infty}(\R)$) and that it satisfies~\eqref{hypF} with $C_R=\|K\|_1+2\,|J|$ and $\lim_{R\to+\infty}\epsilon_R=0$, and even $\epsilon_R=0$ for $R$ large enough if $K$ is compactly supported.
\end{proof}

Therefore, we just have to prove \Cref{thm:localized_advection_bound}.  There are two key estimates that we use to prove this.  The first is a local-in-time Harnack inequality, allowing us to compare $u$ at any two points at the same time, and the second is an extension of this to the gradient of~$u$.  The first result in Lemma~\ref{lem:harnack} below is a slight generalization of the original version which appears in~\cite[Theorem~1.2]{BHR_LogDelay}, while the second result is new and relies on the first.

\begin{lem}\label{lem:harnack}
Suppose that the assumptions of Theorem~$\ref{thm:localized_advection_bound}$ hold.  For any $t_0>0$, $s_0\ge0$, $R>0$ and $p\in(1,+\infty)$, there exists a constant $C>0$, depending only on $t_0$, $s_0$, $R$, $p$, $F$, and~$n$, such that if $t \geq t_0$, $s \in [0,s_0]$, and $|x-y| \leq R$, then
\begin{equation}\label{*14}
		u(t,x) \leq C\,u(t+s,y)^{1/p}\,\max\{\|u\|_\infty^{1-1/p},\|u\|_\infty\},
\end{equation}
and
	\[
		|\nabla u(t,x)| \leq C\,u(t,y)^{1/p}\,\max\{\|u\|_\infty^{1-1/p},\|u\|_\infty\}\,(1+\|u\|_\infty).
	\]
\end{lem}

We now show how to conclude \Cref{thm:localized_advection_bound} from \Cref{lem:harnack}.  The proof of \Cref{lem:harnack} is in \Cref{sec:harnack}.

\begin{proof}[Proof of Theorem~$\ref{thm:localized_advection_bound}$ from Lemma~$\ref{lem:harnack}$]

First of all, the general assumptions on $F$ and the boundedness of $u$ imply that $\sup_{t>0}\|F(u(t,\cdot))\|_{(W^{1,\infty}(\R^n))^n}<+\infty$. From standard parabolic estimates, it follows that all functions $u$, $u_t$ , $u_{x_i}$ and $u_{x_ix_j}$ belong to $L^q_{loc}((0,+\infty)\times\R^n)$ for every $q\in[1,+\infty)$ and $1\le i,j\le n$, and furthermore that $u$ and $\nabla u$ are H\"older continuous in $[\epsilon,+\infty)\times\R^n$ for every $\epsilon>0$. Since $\|F(u(t,\cdot))-F(u(t',\cdot))\|_{(W^{1,\infty}(\R^n))^n}\le C\|u(t,\cdot)-u(t',\cdot)\|_{L^{\infty}(\R^n)}^{\alpha}$ for some  $C>0$, $\alpha\in(0,1)$ and for all $t,t'\in(0,+\infty)$, and since $\sup_{t>0}\|F(u(t,\cdot))\|_{(W^{1,\infty}(\R^n))^n}<+\infty$, one gets that the functions $(t,x)\mapsto F(u(t,\cdot))(x)$ and $(t,x)\mapsto \nabla\cdot F(u(t,\cdot))(x)$ are H\"older continuous in $[\epsilon,+\infty)\times\R^n$ for every $\epsilon>0$. Finally, Schauder parabolic estimates imply that all functions $u$, $u_t$ , $u_{x_i}$ and $u_{x_ix_j}$ are H\"older continuous in $[\epsilon,+\infty)\times\R^n$ for every $\epsilon>0$. From the strong parabolic maximum principle, one also has $u>0$ in $(0,+\infty)\times\R^n$.\par
Fix any $c \in (0,2)$.  Using the bounds~\eqref{hypF} for $\nabla\cdot F(u(t,\cdot))$ along with the results of \Cref{lem:harnack} applied with $u$, $t_0=1/2$, $s_0 = 0$, and $p=3/2$, we see that, for any $R>0$ large enough, $u$ satisfies
\begin{equation}\label{eq:u_super}
	u_t - \Delta u +F(u(t,\cdot))(x)\cdot\nabla u \geq u(1-\tilde{\epsilon}_R - u -  C_R u^{2/3})
\end{equation}
in $[1,+\infty)\times\R^n$, where $C_R$ is a positive constant depending only on $R$, $F$, and $\|u\|_\infty$ (and therefore $C_R$ depends only on $F$ and $\|u\|_\infty$), and $\tilde{\epsilon}_R=\epsilon_R\|u\|_{\infty}\to0$ as $R\to+\infty$. In the rest of the proof, we fix $R>0$ large enough so that~\eqref{eq:u_super} holds and $2\sqrt{1 - \tilde{\epsilon}_R} > c$, and we also choose $\eta>0$ such that
$$(c+\eta)^2<4(1-\tilde{\epsilon}_R).$$\par
In order to build a sub-solution to~\eqref{eq:u_super}, we define, for $(t,x)\in[1,+\infty)\times\R^n$,
\[
	A(t,x) = \frac{ F(u(t,\cdot))(x)}{ \max\{ 1, |F(u(t,\cdot))(x)|\,\eta^{-1} \}}.
\]
Notice that $A:[1,+\infty)\times\R^n\to\R^n$ is H\"older continuous and $|A(t,x)|\leq \eta$ for all $(t,x)\in[1,+\infty)\times\R^n$.  This advection cut-off is, to our knowledge, novel. Let then $R'>0$ be such that $\epsilon_{R'}\|u\|_{\infty}<\eta/2$. The assumptions~\eqref{hypF} on $F$ and Lemma~\ref{lem:harnack} (applied now with power $p=2$) yield the existence of a constant $C'>0$ (depending on $R'$, $F$ and $\|u\|_\infty$, and therefore on $c$, $F$ and $\|u\|_\infty$) such that
\begin{equation}\label{eq:small_F}
	|F(u(t,\cdot))(x)|\le C'\sqrt{u(t,x)}+\frac{\eta}{2}\ \hbox{ for all }(t,x)\in[1,+\infty)\times\R^n.
\end{equation}
Then, fix $M>1$ such that
\begin{equation}\label{defM}
C'\sqrt{\frac{\|u(1,\cdot)\|_{\infty}+1}{M}}<\frac{\eta}{2}
\end{equation}
and define $\underline u$ to solve
\[\begin{cases}
	\underline u_t - \Delta\underline u + A(t,x)\cdot\nabla\underline u = \underline u( 1 - \tilde{\epsilon}_R - M\underline{u} - C_R \underline u^{2/3}), \quad &\text{ in } [1,+\infty)\times\R^n,\\
	\underline u(1,\cdot)= u(1,\cdot)/M.
\end{cases}\]
Notice that, from standard parabolic estimates, $\underline{u}$ is a classical solution in $[1,+\infty)\times\R^n$ and that all functions $\underline{u}$, $\underline{u}_t$, $\underline{u}_{x_i}$ and $\underline{u}_{x_ix_j}$ are H\"older continuous in $[1,+\infty)\times\R^n$, while
\begin{equation}\label{underlineu}
0\le\underline{u}(t,x)
	<\max\left( \frac{\|u(1,\cdot)\|_\infty}{M}, \frac{1}{M}\right)
	<\frac{\|u(1,\cdot)\|_{\infty}+1}{M}\ \hbox{ for all }(t,x)\in[1,+\infty)\times\R^n
\end{equation}
from the maximum principle.\par
Now we show that we may compare $\underline u$ and $u$. To do so, observe that, for every $(t,x)\in[1,+\infty)\times\R^n$ such that $u(t,x)<(\|u(1,\cdot)\|_{\infty}+1)/M$, one has $|F(u(t,\cdot))(x)|<\eta$ due to~\eqref{eq:small_F} and~\eqref{defM}; hence, $A(t,x)=F(u(t,\cdot))(x)$. Together with~\eqref{underlineu} and the definition of~$\underline{u}$, it follows that $u$ and $\underline u$ satisfy the same equation locally around any point where they could ``touch.''  Thus, one concludes from the maximum principle that
$$0\le\underline{u}(t,x)\le u(t,x)\ \hbox{ for all }(t,x)\in[1,+\infty)\times\R^n.$$
Since the positive real numbers $\tilde{\epsilon}_R$, $c$, and $\eta$ are such that $(c+\eta)^2<4(1-\tilde{\epsilon}_R)$ and since $|A|\le\eta$ in $[1,+\infty)\times\R^n$, we may now apply the results of~\cite[Theorem~1.2]{BerestyckiHamelNadin}, which imply that
\[
	\liminf_{t\to+\infty} \inf_{|x| < ct} u(t,x)
		\geq \liminf_{t\to+\infty} \inf_{|x| < ct} \underline u(t,x)
		\geq \delta,
\]
where $\delta>0$ is the unique positive real number solving $1 - \tilde{\epsilon}_R = M \delta + C_R \delta^{2/3}$. The above arguments imply that $\delta$ depends on $c$, $F$, and $\|u\|_\infty$.

To conclude, we need to consider the case when $\epsilon_R = 0$ in~\eqref{hypF} for all $R>0$ large enough.  Hence, in the above calculations,
$\tilde{\epsilon}_R=\epsilon_R\|u\|_\infty=0$ and
$$	u_t - \Delta u +F(u(t,\cdot))(x)\cdot\nabla u \geq u(1 - u -  C_R u^{2/3})$$
in $[1,+\infty)\times\R^n$. Using again the bounds~\eqref{hypF} for $F(u(t,\cdot))$ along with the results of \Cref{lem:harnack} applied (twice) with $u$, $t_0=1/2$, $s_0=0$, and $p=3/2$, one gets the existence of a positive constant $C'_R$ depending only on $F$ and $\|u\|_\infty$ such that
$$|F(u(t,\cdot))(x)\cdot\nabla u(t,x)|\le C'_Ru(t,x)^{4/3}\ \hbox{ for all }(t,x)\in[1,+\infty)\times\R^n.$$
Hence,
$$	u_t - \Delta u\geq u(1 - u -  C_R u^{2/3}-C'_Ru^{1/3})$$
in $[1,+\infty)\times\R^n$. It then follows from the maximum principle that $u(t,x)\ge\underline{v}(t,x)$ for all $(t,x)\in[1,+\infty)\times\R^n$, where $\underline{v}$ is the classical solution of
\[\begin{cases}
	\underline v_t - \Delta\underline v = \underline v( 1 - \underline{v} - C_R \underline v^{2/3}-C'_R\underline{v}^{1/3}), \quad &\text{ in } [1,+\infty)\times\R^n,\\
	\underline v(1,\cdot)= u(1,\cdot).
\end{cases}\]
This, along with the results of~\cite{Gartner}, implies that
\[
	\liminf_{t\to+\infty}\inf_{|x| < 2t - (n+1/2) \log(t)} u(t,x)
		\geq \liminf_{t\to+\infty}\inf_{|x| < 2t - (n+1/2) \log(t)} \underline v(t,x)>0.
\]
The proof of Theorem~\ref{thm:localized_advection_bound} is thereby complete.
\end{proof}


\subsubsection{Convergence to $1$ when $K$ is small}

In many cases, a weak bound on the infimum may be bootstrapped to convergence to $1$.  This occurs, in particular, when there are non-local terms in the reaction term and when the non-local terms are small enough.  This argument is an adaptation from the work of Salako and Shen~\cite{SalakoShen1, SalakoShen2}.

\begin{corollary}\label{cor:convergence_to_1}
Suppose that the assumptions of Theorem~$\ref{thm:finite_speed}$ hold and $J< 1/2$. Then, for any $c\in(0,2)$, we have that
	\[
		\lim_{t\to+\infty} \sup_{|x| <ct} |u(t,x) - 1| = 0.
	\]
	If $K$ is compactly supported, then
	\[
		\lim_{\substack{(t,L)\to(+\infty,+\infty),\\L<2t-(3/2)\log(t)}}\ \sup_{|x| < 2t - (3/2)\log(t) - L} |u(t,x) - 1| = 0.
	\]
\end{corollary}

\begin{proof}
Let $X(t,L) =2t - (3/2)\log(t) - L$ or $ct$ with $c\in(0,2)$ depending on whether $K$ is compactly supported or not.  Take any sequences $(t_n)_{n\in\N}$ in $(0,+\infty)$, $(L_n)_{n\in\N}$ in $(0,+\infty)$, and $(x_n)_{n\in\N}$ in $\R$ such that $t_n$ and $L_n$ converge to $+\infty$, $X(t_n,L_n) > 0$, and $|x_n| < X(t_n,L_n)$.  Define
\[
	u_n(t,x) = u(t+t_n, x + x_n).
\]
We are finished if we show that $\lim_{n\to+\infty} u_n(0,0) = 1$.\par
As already emphasized in Section~\ref{sec:results}, since $u$ is bounded uniformly in $L^\infty$, parabolic regularity theory implies that $u$ is bounded in $C^{1+\alpha/2;2+\alpha}_{t;x}([\epsilon,+\infty)\times\R^n)$ for any $\alpha \in (0,1)$ and $\epsilon>0$.  Hence, up to extracting a sub-sequence, $u_n$ converges to some limit $u_\infty$ in $C^{1;2}_{t;x;loc}(\R\times\R)$. We observe that, due to \Cref{thm:finite_speed},
\[
	0 < \inf_{(t,x) \in\R\times\R} u_\infty(t,x) \leq \sup_{(t,x) \in \R\times \R} u_\infty(t,x) < +\infty.
\]
In addition, we have that
\begin{equation}\label{eq:u_infty}
	(u_\infty)_t + [(K*u_\infty) u_\infty]_x = (u_\infty)_{xx} + u_\infty(1-u_\infty),\ \ \hbox{ in }\R\times\R.
\end{equation}
In order to finish, we show that $\underline u := \inf u>0$ and $\overline u := \sup u>0$ satisfy $\underline u = \overline u = 1$.  The inequality $\underline u>0$ is due to \Cref{thm:finite_speed}. To that end, up to re-centering the equation and taking another limit, we may assume without loss of generality that there exists $(t_i,x_i)$ such that $\underline u = u_\infty(t_i,x_i)$ and that there exists $(t_s,x_s)$ in $\R\times\R$ such that $\overline u = u_\infty(t_s,x_s)$. From here, we handle the cases $J< 0$ and $J \in [0,1/2)$ separately.

First, assume that $J \in [0,1/2)$. At $(t_i,x_i)$, we claim that
\begin{equation}\label{tixi}\begin{array}{rcl}
0 \geq (u_\infty)_t - (u_\infty)_{xx} + (K*u_\infty) (u_\infty)_x & = & u_\infty(1-u_\infty) - (K*u_\infty)_x u_\infty\vspace{3pt}\\
& \geq &\underline u(1-\underline u) + \underline u J(\underline u - \overline u).\end{array}
\end{equation}
Indeed, since $J\ge0$ and $K$ is thus non-decreasing in $(-\infty,0)$ and $(0,+\infty)$, the last inequality is a consequence of the following one:
\[
	-(K*u_\infty)_x(t_i,x_i)
		= J\,u_\infty(t_i,x_i) - \int_{-\infty}^0 [u_\infty(t_i,x+y) + u_\infty(t_i,x-y)]\,dK(y)
		\geq J\underline u - J \overline u,
\]
where by $\mu:=dK(y)$ we mean the Radon measure such that $\mu((a,b))=\lim_{x\to b^-}K(x)-\lim_{x\to a^+}K(x)$ and $\mu(\{c\})=\lim_{x\to c^+}K(x)-\lim_{x\to c^-}K(x)$ for every $-\infty\le a<b\le+\infty$ and $c\in\R$. Since $\underline u>0$,~\eqref{tixi} yields the inequality $1\le J \overline u + (1-J)\underline u$, which we use later. Similarly, we may argue that at $(t_s,x_s)$ we have
$$0 \leq (u_\infty)_t - (u_\infty)_{xx} + (K*u_\infty) (u_\infty)_x=u_\infty(1-u_\infty) - (K*u_\infty)_x u_\infty\leq\overline u(1-\overline u) + \overline u J(\overline u - \underline u).$$
The above inequality and~\eqref{tixi} give us $\underline u - \underline u^2 + \underline u J(\underline u - \overline u)\leq 0\le \overline u - \overline u^2 + \overline u J(\overline u - \underline u)$, hence
\[
	(1-J) (\overline u + \underline u)(\overline u - \underline u)
		= (1-J)\left(\overline u^2 - \underline u^2\right)
		\leq \overline u - \underline u.
\]
We claim that $\overline u = \underline u$.  If this is not true, then dividing by $\overline u - \underline u>0$ above leads to $(1-J)\,(\overline{u}+\underline{u})\le1$, and using the inequality $1\le J \overline u + (1-J) \underline u$  implies that
\[
	(1-2J)\,\overline u \leq 0.
\]
Since $\overline{u}>0$ and $J < 1/2$, by assumption, this is a contradiction.  Hence $\overline u = \underline u$. Therefore, $u_\infty$ is uniformly equal to a positive constant.  Since $K \in L^1(\R)$ is odd, then $K*u_\infty \equiv 0$.  Hence, from~\eqref{eq:u_infty}, the constant $u_\infty$ satisfies $u_\infty(1-u_\infty)=0$, with $u_{\infty}>0$.  We conclude that $u_\infty \equiv 1$, finishing the proof of the case when $J \in [0,1/2)$.

When $J< 0$, the argument is similar, but less complicated.  At $(t_i,x_i)$, we claim that
\begin{equation}\label{eq:underline_u_bound}
	0 \geq (u_\infty)_t - (u_\infty)_{xx} + (K*u_\infty)(u_\infty)_x= u_\infty(1-u_\infty) - (K*u_\infty)_xu_\infty\geq \underline u(1-\underline u).
\end{equation}
Indeed, since $J < 0$ and $K$ is thus non-increasing in $(-\infty,0)$ and $(0,+\infty)$,
\[\begin{array}{rcl}
	-(K*u_\infty)_x(t_i,x_i) & = & \ds J\,u_\infty(t_i,x_i) - \int_{-\infty}^0 [u_\infty(t_i,x+y) + u_\infty(t_i,x-y)]\,dK(y)\vspace{3pt}\\
	& \ge & \ds J\,\underline{u}+2\underline{u}\int_{-\infty}^0-dK(y)=0.\end{array}
\]
Hence, from~\eqref{eq:underline_u_bound} along with the positivity of $\underline u$, it follows that $\underline u \geq 1$. Hence, using also \Cref{lemmax},
\[
	1 \leq \underline u \leq \overline u \leq 1,
\]
finishing the proof.
\end{proof}


\section{Negative chemotaxis with large tails}\label{sec:acceleration}

In this section we prove Proposition~\ref{prop:well_posed} in Subsection~\ref{sec:well_posed} and then \Cref{thm:acceleration} in the following three subsections.  We handle the three cases separately.  We recall that, throughout this section, $J\le 0$ and $K$ is non-increasing expect at $0$, non-positive on $(-\infty,0)$, and non-negative on $(0,+\infty)$.


\subsection{Well-posedness of the model in \Cref{thm:acceleration}: proof of Proposition~\ref{prop:well_posed}}\label{sec:well_posed}

In this section, our goal is to show that {\em a priori} bounds on $u$ may be established in the case when $K \in L^\infty(\R)$ and is non-increasing except at $x = 0$, non-positive on $(-\infty,0)$, and non-negative on $(0,+\infty)$.  We aim to give enough of a treatment that the reader is re-assured that the problem is well-posed, without belaboring the point.

We first claim that the Cauchy problem~\eqref{eq:the_equation}-\eqref{hypu0} is locally well-posed in spaces with suitable decay at $|x| = +\infty$. This may easily be justified by rigorously constructing a unique strong solution $u$ in $(0,T]\times\R$, via the Banach fixed point in sets of the type
$$E_T=\big\{u\in C([0,T];L^1(\R)): u(0,\cdot)=u_0\hbox{ and } 0\le u(t,x)\le C\,e^{\gamma t-|x|}\hbox{ in }(0,T]\times\R\big\}$$
endowed with the norm $\|u\|_{E_T}=\max_{t\in[0,T]}\|u(t,\cdot)\|_{L^1(\R)}+\|u(t,x)\,e^{-\gamma t+|x|}\|_{L^{\infty}((0,T)\times\R)}$, with a constant $C>0$ depending on $u_0$, and some constants $\gamma>0$ large and $T>0$ small.

It then follows that the functions $K*u$ and $(K*u)_x$ are bounded in $(0,T)\times\R$. Local parabolic regularity in Sobolev spaces (see e.g.~\cite[Theorem~7.22]{Lieberman}), along with Sobolev embedding theorems (see e.g.~\cite[Lemma A3]{Engler}), implies that $u \in C^{(1+\alpha)/2; 1+\alpha}_{t;x}([\epsilon,T]\times \R)$ for all $\alpha \in (0,1)$ and $\epsilon\in(0,T)$.  Examining the form of $(K*u)_x$, it is clear that it is $C^{(1+\alpha)/2;1+\alpha}_{t;x}([\epsilon,T]\times \R)$ since $u$ is and $K$ has bounded variation. In particular, the function $K*u$ is H\"older continuous with respect to $x$ in $[\epsilon,T]\times\R$. 

We now  prove H\"older regularity in time.  Due to the heat kernel bounds of Aronson, \cite[Theorem 10]{Aronson}, which crucially do not require more regularity than $L^\infty$ estimates on the coefficients of the equation, and the compact support of $u_0$, there is a constant $C_T>0$ so that
\begin{equation}\label{gaussian}
	u(t,x) \leq \frac{C_T}{\sqrt t}\,e^{- \frac{x^2}{C_T t}}
\end{equation}
for all $(t,x) \in(0,T]\times \R$.  Fix $\alpha \in (0,1)$, $\epsilon\in(0,T]$ and consider $t_1, t_2 \in [\epsilon,T]$.  We may assume without loss of generality that $0<|t_1 - t_2|<1$ since a bound on $|K*u(t_1,x) - K*u(t_2,x)|/|t_1-t_2|^\alpha$ follows from the $L^\infty((0,T)\times\R)$ bound on $K*u$ if $|t_1-t_2|\ge 1$. Fix $\alpha'\in(\alpha,1)$ and $R=\sqrt{-\alpha'T C_T \log|t_1 - t_2|}>0$. Then we have that
\[\begin{split}
	&|K*u(t_1,x) - K*u(t_2,x)|
		\leq \int_\R |K(y)| |u(t_1,x-y) - u(t_2,x-y)| dy\\
		&\leq \|K\|_\infty \left[ \int_{|y - x| > R} |u(t_1,x-y) - u(t_2,x-y)|dy + \int_{|y-x|\leq R} |u(t_1,x-y) - u(t_2,x-y)|dy \right]\\
		&\leq \|K\|_\infty \left[\int_{|y| \geq R} \left(\frac{C_T}{\sqrt t_1} e^{-y^2/C_Tt_1} + \frac{C_T}{\sqrt t_2} e^{-y^2/C_Tt_2}\right) dy +  2\|u\|_{C^{\alpha';2\alpha'}_{t;x}([\epsilon,T]\times \R)}|t_1-t_2|^{\alpha'}R\right]\\
		&\leq 2C_T^2 \|K\|_\infty \left[ \sqrt t_1 \frac{e^{-R^2/C_Tt_1}}{R} +  \sqrt t_2 \frac{e^{-R^2/C_Tt_2}}{R} + \|u\|_{C^{\alpha';2\alpha'}_{t;x}([\epsilon,T]\times \R)} |t_1 - t_2|^{\alpha'} R \right].
\end{split}\]
Using that $t_1,t_2 \leq T$ as well as the explicit expression for $R$, and absorbing all factors into the constant $C_T$, including $\|K\|_\infty$ and $\|u\|_{C^{\alpha';2\alpha'}_{t;x}([\epsilon,T]\times \R)}$, we obtain
\[
	|K*u(t_1,x) - K*u(t_2,x)|
		\leq C_T
			\left( \frac{|t_1 - t_2|^{\alpha'}}{\sqrt{-\alpha'\log|t_1-t_2|}} + |t_1 - t_2|^{\alpha'} \sqrt{-\alpha'\log|t_1-t_2|}\right)
\]
Using that $\alpha' > \alpha$, this simplifies to the desired inequality:
\[
	|K*u(t_1,x) - K*u(t_2,x)|
		\leq C_T |t_1-t_2|^{\alpha}.
\]

Since we have shown that all the coefficients in the equation for $u$ are H\"older continuous in $[\epsilon,T]\times\R$ for every $\epsilon\in(0,T)$ and with every exponent $\alpha\in(0,1)$, it follows from the classical Schauder estimates for parabolic equations that $u \in C^{1 +\alpha/2; 2+\alpha}_{t;x}([\epsilon,T]\times \R)$ for every $\epsilon\in(0,T)$ and $\alpha\in(0,1)$, see e.g.~\cite[Theorem~4.9]{Lieberman}. Notice also that the Schauder estimates imply that, for every $\epsilon\in(0,T)$, there is a constant $C_{\epsilon,T}>0$ such that
$$|u_t(t,x)|+|u_x(t,x)|+|u_{xx}(t,x)|\le C_{\epsilon,T}\,e^{-|x|}\ \hbox{ for all }(t,x)\in[\epsilon,T]\times\R.$$
In addition, since the functions $K*u$ and $(K*u)_x$ are bounded in $(0,T)\times\R$, since $0$ is a sub-solution and since $J\le0$ here, a maximum principle argument as in \Cref{lemmax} implies that
$$0<u(t,x) \leq 1\ \hbox{ for all }(t,x)\in(0,T]\times \R.$$
Furthermore, the function $P$ defined in~\eqref{defP}, which is by construction continuous in $[0,T]$, is actually of class $C^1((0,T])$ and, by integrating the equation over $\R$ for any $t\in(0,T]$, one infers that
\begin{equation}\label{eq:P_V}
P'(t)=\int_\R u(t,x)\,(1-u(t,x))\,dx\le P(t)
\end{equation}
for all $t\in(0,T]$, hence $P(t)\le P(0)\,e^t=\|u_0\|_1e^t$ for all $t\in[0,T]$. Since $K \in L^\infty(\R)$, then we obtain the bound
\[
	|K*u(t,x)|
		\leq \|K\|_\infty\,\|u_0\|_1\,e^t
\]
for all $(t,x)\in[0,T]\times\R$. Thus, we also obtain that
\begin{equation}\label{ineqK*u}
	|(K*u)_x(t,x)|
		= \left|-J u(t,x) + \int_{-\infty}^0 [u(x-y) + u(x+y)]\,dK(y)\right|
		\leq 2\,|J|
\end{equation}
for all $(t,x)\in(0,T]\times\R$.

The above upper bounds for $u$, $|K*u|$ and $|(K*u)_x|$ do not depend on $T$ and it then follows from standard arguments that the maximal existence time for the solution $u$ is $+\infty$ (hence, $u$ being locally unique is globally unique) and that all above estimates hold for any $T>0$. The proof of Proposition~\ref{prop:well_posed} is thereby complete.\hfill$\Box$


\subsection{The case when $\displaystyle\lim_{x \to +\infty} K(x) = K_\infty > 0$: proof of \Cref{thm:acceleration}-(i)}

\begin{proof}[Proof of Theorem~$\ref{thm:acceleration}$-{\rm{(i)}}]
The upper bound follows from the proof of Proposition~\ref{prop:well_posed}: the function $P$ defined by~\eqref{defP} is continuous in $[0,+\infty)$, of class $C^1((0,+\infty)$, and it satisfies
\begin{equation}\label{ineqP}
P(t)\le P(0)\,e^t\ \hbox{ for all }t\ge0
\end{equation}
(notice that the above inequality holds when $K_{\infty}=0$ as well).

To proceed with the lower bound, we first obtain a refined pointwise upper bound on $u$.  To this end, fix any $\e \in(0,2K_\infty)$ and let
$$w(t,x) = e^{\epsilon t} u(t,x)$$
for $(t,x)\in[0,+\infty)\times\R$.  Notice that $w$ is non-negative, since $u$ is, and solves
\[
	w_t - \epsilon w + (K*u) w_x + (K*u)_x w = w_{xx} + w( 1 - e^{-\epsilon t} w)\quad\hbox{ in }(0,+\infty)\times\R,
\]
with $w(0,x) = u_0(x)$ for all $x\in\R$.  Let $T_\epsilon = \e^{-1}\log(1 + 2K_\infty)>0$ and fix any $T\ge T_\epsilon$.

There are two cases.  The first case is that $\|w\|_{L^\infty((0,T)\times\R)} \leq \|u_0\|_\infty$.  Due to \Cref{prop:well_posed}, $u \in C^{1;2}_{t;x}((0,T]\times\R)$, which implies that $w \in C^{1;2}_{t;x}((0,T]\times \R)$. Then we have that $u(t,x) = e^{-\epsilon t} w(t,x) \leq e^{-\epsilon t} \|u_0\|_\infty \leq e^{-\e t}$ for all $(t,x) \in (0,T]\times \R$.  In particular, this is true with $ t = T$, which implies that
$$u(T, x) \leq\frac{1}{1+2K_\infty}\ \hbox{ for all }x\in \R.$$

The second case is that $\|w\|_{L^\infty((0,T)\times\R)} > \|u_0\|_\infty$. Thanks to the upper bounds~\eqref{gaussian} in Proposition~\ref{prop:well_posed} and the bounds~\eqref{boundu}, which still hold locally in time since the functions $K*u$ and $(K*u)_x$ are bounded in $L^{\infty}((0,T)\times\R)$, there exists $(t_0,x_0) \in (0,T]\times \R$ such that $(t_0,x_0)$ is the location of a maximum of $w$ in $(0,T]\times\R$.  For notational ease, let $M = w(t_0,x_0) = \|w\|_{L^\infty((0,T)\times\R)}>\|u_0\|_\infty>0$.  Then, at $(t_0,x_0)$, we claim that
\begin{equation}\label{eq:M_bound}
\begin{split}
	0 \leq w_t - w_{xx} + (K*u) w_x
		& = M(1+ \epsilon - e^{-\epsilon t_0}M) - (K*u)_x M\\
		&\leq M(1 + \epsilon - e^{-\epsilon t_0}(1+2 K_\infty) M).
\end{split}
\end{equation}
Indeed, the first inequality comes from the fact that $(t_0,x_0)$ is the location of a maximum, while the second inequality is proved as follows.  Since $J$ is non-positive and $K$ is non-increasing on $(-\infty,0)$ and $(0,+\infty)$ we have that $|J| \geq 2K_\infty$, and, hence,
\[\begin{split}
	-(K*u)_x(t_0,x_0)
		&= J\,u(t_0,x_0) - \int_{-\infty}^0 [u(t_0,x_0+y) + u(t_0,x_0-y)] dK(y)\\
		&\leq -|J|\,u(t_0,x_0) + (|J|-2K_\infty) \| u(t_0,\cdot)\|_\infty\\
		&= -|J|\,w(t_0,x_0) e^{-\epsilon t_0} + (|J|-2K_\infty) \| w(t_0,\cdot)\|_\infty e^{-\epsilon t_0}
		= -2K_\infty M e^{-\epsilon t_0}.
\end{split}\]
It follows from~\eqref{eq:M_bound} that $M \leq (1+\e) e^{\e t_0}/(1+2K_\infty)$.  Hence, for all $x\in\R$,
\[
	u(T, x) = e^{-\epsilon T} w(T,x)
		\leq  e^{-\epsilon T} M
		\leq e^{-\epsilon T} \frac{1+\epsilon}{1 + 2K_\infty} e^{\epsilon t_0}
		\leq \frac{1+\epsilon}{1 + 2K_\infty}.
\]

Combining the results of our two cases, we see that
\begin{equation}\label{uT}
\|u(T,\cdot)\|_\infty \leq\frac{1+\e}{1+2K_\infty}\ \hbox{ for all }T\ge T_\e=\frac{\log(1 + 2K_\infty)}{\e}>0.
\end{equation}

The proof of the lower bound of $P$ is now relatively straightforward.  Integrate~\eqref{eq:the_equation} to obtain, for $t \geq T_\e$,
\[
	P'(t) = \int_\R u(t,x)(1-u(t,x)) dx \geq \int_\R u(t,x)\left(\frac{2K_\infty - \e}{1 + 2K_\infty}\right) dx = \left(\frac{2K_\infty - \e}{1 + 2K_\infty}\right)P(t).
\]
For simplicity, let
\begin{equation}\label{defr}
r:=\frac{2K_\infty-\e}{1+2K_\infty}\,\in\Big(0,\frac{2K_\infty}{1+2K_\infty}\Big).
\end{equation}
Integrating the equation above from $T_\e$ to $t$ yields the desired result:
\[
	P(t) \geq e^{r(t-T_\e)} P(T_\e)\ \hbox{ for all }t\ge T_\e.
\]
To control $P(T_\e)$, we again integrate~\eqref{eq:the_equation} to obtain
\begin{equation}\label{eq:P_increasing}
	P'(t) = \int_\R u(t,x)(1-u(t,x))\,dx
		\geq 0
\end{equation}
for all $t>0$. Here we used that $0<u\le1$ in $(0,+\infty)\times\R$ from \Cref{prop:well_posed}.  Hence, we have in particular that $P(t)\ge P(0)$ for all $t\in[0,T_\e]$. Combining this with our bound on $P(t)$ by $P(T_\e)$ yields, for all $t \in [0,+\infty)$,
\[
	P(t) \geq e^{r(t - T_\e)} P(0).
\]
This concludes the proof.
\end{proof}

\begin{rem} Since $\epsilon$ can be arbitrary in $(0,2K_\infty)$ in the above proof, it follows from~\eqref{uT} that
\begin{equation}\label{*20}
\limsup_{t\to+\infty}\|u(t,\cdot)\|_\infty\le\frac{1}{1+2K_\infty}
\end{equation}
and the real number $r$ in~\eqref{defr} can be any real number in the interval $(0,2K_\infty/(1+2K_\infty))$ and can therefore be all the closer to $1$ as $K_\infty$ is large.
\end{rem}

To illustrate that the result of Theorem~\ref{thm:acceleration}-(i) is not capturing a situation where the maximum of $u$ is small while its support is wide, we consider a specific case in the following corollary, which amounts to the first claim in \Cref{cor:acceleration_pointwise}.

\begin{corollary}\label{cor:acceleration_pointwise1}
Under the assumptions of Theorem~$\ref{thm:acceleration}$-$(i)$, suppose further that $K\in C^1(\R^*)$ is convex on $(0,+\infty)$ and that $u_0$ is even and radially non-increasing and of class $C^{2+\alpha}(\R)$ for some $\alpha\in(0,1)$.  Then there exists $\lambda > 0$ such that
$$\lim_{t\to+\infty} \inf_{|x| < e^{\lambda t}} u(t,x) \geq \frac{1}{2\,(1+|J|)}.$$
\end{corollary}

\begin{proof}
Due to the assumptions, we first claim that $u(t,\cdot)$ is even and radially non-increasing for all time $t\ge0$. This follows from the well-posedness of~\eqref{eq:the_equation} shown in \Cref{prop:well_posed} along with comparison principle arguments. Indeed, to be more precise, observe first that, since $K$ is odd, the function $\tilde{u}(t,x):=u(t,-x)$ satisfies the same equation~\eqref{eq:the_equation} as $u$ with the same initial condition $u_0$. Therefore, Proposition~\ref{prop:well_posed} implies that $\tilde{u}$ is equal to $u$, meaning that $u(t,\cdot)$ is even for all $t\ge0$.

Let us now show that $u(t,\cdot)$ is radially non-increasing for every $t\ge 0$. First of all, since $u_0$ is of class $C^{2+\alpha}(\R)$ with $\alpha\in(0,1)$, similar arguments as in the proof of Proposition~\ref{prop:well_posed} lead to the local and then global existence and uniqueness of a solution of~\eqref{eq:the_equation}-\eqref{hypu0} of class $C^{1+\alpha/2;2+\alpha}_{t;x}([0,+\infty)\times\R)$ and satisfying the same estimates as in Proposition~\ref{prop:well_posed}. In other words, the solution $u$ constructed in Proposition~\ref{prop:well_posed} is then of class $C^{1+\alpha/2;2+\alpha}_{t;x}([0,+\infty)\times\R)$. Furthermore, parabolic regularity also implies that the function $v:=u_x$ is a bounded classical $C^{1;2}_{t;x}((0,+\infty)\times\R)\cap C([0,+\infty)\times\R)$ solution of
\begin{equation}\label{eqv}
v_t+(K*u)\,v_x+2\,(K*u)_x\,v+(K*v)_x\,u=v_{xx}+v\,(1-2u)\quad \hbox{ in }(0,+\infty)\times\R
\end{equation}
and that $v(t,x)\to0$ as $|x|\to+\infty$ locally uniformly in $t\in[0,+\infty)$ since $u$ satisfies the same property and the coefficient $K*u$ is bounded locally with respect to $t\in[0,+\infty)$ and the other coefficients $(K*u)_x$ and $(K*v)_x$ are globally bounded. Our goal is to show that
\begin{equation}\label{ux0}
u_x=v\le0\ \hbox{ in }[0,+\infty)\times[0,+\infty),
\end{equation}
meaning that $u(t,\cdot)$ is non-increasing in $[0,+\infty)$ for all $t\ge0$. If $J=0$ then $K\equiv 0$ in $\R$ from our assumptions on $K$; in this case, the boundedness of $u$ and $v$ and the non-positivity of $v(0,\cdot)$ on $[0,+\infty)$ (from our assumption on $u_0$) and of $v(\cdot,0)$ on $[0,+\infty)$ ($v(t,0)=0$ for all $t\ge0$ since $u(t,\cdot)$ is even) directly yield~\eqref{ux0} from the local maximum principle. Assume then in the rest of this paragraph that $J\neq0$, that is, $J<0$ with our hypotheses on $K$. Assume by way of contradiction that there exists $T>0$ such that $\sup_{[0,T]\times[0,+\infty)}v>0$. Denote
$$w(t,x)=e^{-6|J| t}v(t,x).$$
Then $\sup_{[0,T]\times[0,+\infty)}w>0$. From the previous observations and the facts that $w(t,0)=0$ for all $t\ge0$ (by evenness of $u(t,\cdot)$) and $w(0,x)\le0$ for all $x\ge0$ by assumption, one infers the existence of $(t_0,x_0)\in(0,T]\times(0,+\infty)$ such that
$$\eta:=w(t_0,x_0)=\sup_{[0,T]\times[0,+\infty)}w>0.$$
From~\eqref{eqv}, there holds, at $(t_0,x_0)$,
$$\begin{array}{rcl}
0\le w_t+(K*u)\,w_x-w_{xx} & = & -6\,|J|\,w-2\,(K*u)_x\,w-(K*w)_x\,u+w\,(1-2u)\\
& = & -6\,|J|\,\eta-2\,(K*u)_x\,\eta+(1-2u)\,\eta-(K*w)_x\,u.\end{array}$$
Together with~\eqref{ineqK*u}, the bound $0<u\le1$ in $(0,+\infty)\times\R$ and the positivity of $\eta$, one infers that
\begin{equation}\label{ineqt0x0}
0\le-|J|\,\eta-(K*w)_x(t_0,x_0)\,u(t_0,x_0).
\end{equation}
Since $w(t_0,\cdot)$ is odd (because $u(t_0,\cdot)$ is even), a straightforward calculation yields
$$\begin{array}{rcl}
-(K*w)_x(t_0,x_0)=-|J|\,w(t_0,x_0) & + & \ds\int_0^{x_0}\big(K'(x_0+y)-K'(x_0-y)\big)\,w(t_0,y)\,dy\vspace{3pt}\\
& + & \ds\int_{x_0}^{+\infty}\big(K'(x_0+y)-K'(x_0-y)\big)\,w(t_0,y)\,dy.\end{array}$$
Since $K\in L^{\infty}(\R)\cap C^1(\R^*)$ is assumed to be odd in $\R$ and convex on $(0,+\infty)$, it follows that $\kappa(y):=K'(x_0+y)-K'(x_0-y)\ge0$ for all $y\in(0,+\infty)\backslash\{x_0\}$. Using $w(t_0,\cdot)\le\eta$ in $[0,+\infty)$, we obtain
$$-(K*w)_x(t_0,x_0)\le-|J|\,\eta+\eta\,\Big(\int_0^{x_0}\kappa(y)\,dy+\int_{x_0}^{+\infty}\kappa(y)\,dy\Big)=-2\,K(x_0)\,\eta\le0.$$
With~\eqref{ineqt0x0} and the positivity of $u(t_0,x_0)$, one gets that $0\le-|J|\,\eta$, which is a contradiction since $J<0$ and $\eta>0$. Therefore,~\eqref{ux0} holds and the even function $u(t,\cdot)$ is non-increasing in $[0,+\infty)$ for every $t\ge0$.

With this in hand, we now claim that $-(K*u)(t,x) u_x(t,x) \geq 0$ for all $t>0$ and $x\in\R$.  Indeed, if $x\ge0$, then $u_x(t,x) \leq 0$ by our observation that $u(t,\cdot)$ is radially non-increasing, while
\begin{equation}\label{eq:Ku}
	(K*u)(t,x)
		= \int_\R K(y)u(t,x-y)dy
		= \int_0^{+\infty} [ u(t,x-y) - u(t,x+y)] K(y)dy \geq 0.
\end{equation}
The second equality follows from the fact that $K$ is odd.  The inequality is due to the fact that, for every $y>0$, $K(y) \geq 0$, $|x-y| \leq x+y$, and $u(t,\cdot)$ is even and radially non-increasing.  This establishes the inequality $-(K*u)(t,x) u_x(t,x) \geq 0$ for all $t>0$ and $x\ge0$. Similarly, if $t>0$ and $x\le0$, then $u_x(t,x)\ge0$ and $(K*u)(t,x)\le0$ since this time $|x+y|\le|x|+y=|x-y|$ for all $y>0$.

In addition, a straightforward computation implies that, for all $(t,x) \in (0,+\infty) \times \R$,
\begin{equation}\label{eq:Kux}
	(K*u)_x(t,x)
		= -J\,u(t,x) + \int_{-\infty}^0 [u(t,x+y) + u(t,x-y)] dK(y)
		\leq |J|\,u(t,x).
\end{equation}
The combination of~\eqref{eq:Ku} and~\eqref{eq:Kux} yield
\begin{equation}\label{eq:exponential_sub_solution}
	u_t - u_{xx} \geq u(1 - (1 + |J|) u)
		\qquad \text{ in }(0,+\infty)\times \R.
\end{equation}
We use this inequality to create a sub-solution, and, hence, lower bound of $u$ in the sequel.

In order to create a sub-solution, we first obtain a preliminary lower bound of $u$.  Fix $r$ as in \Cref{thm:acceleration} with the choice $\epsilon=K_\infty$, that is, $r=K_\infty/(1+2K_\infty)>0$, and let $M =13$ and $t_0>M$ to be determined below, depending only on $P(0)$ and $K$. In addition, fix $t \geq t_0$.  We claim that there exists $x_t\in\R$ such that
\begin{equation}\label{defxt}
u(t/M,x_t) \geq\frac{e^{rt/(2M)}}{(1 + |x_t|)^{1 + r/4}}.
\end{equation}
If not, then using \Cref{thm:acceleration}-(i), there exists $C_0>0$ depending only on $P(0)$ and $K$ such that
\[
	C_0\,e^{rt/M}
		\leq P(t/M)
		= \int_\R u(t/M,x)dx
		< \int_\R \frac{e^{rt/(2M)}}{(1 + |x|)^{1 + r/4}} dx
		= \frac{8\,e^{rt/(2M)}}{r},
\]
which is clearly a contradiction when $t_0 > (2M/r) \log(8/(rC_0))$.  Hence, enlarging $t_0>M$ if necessary (depending on $M$, $r$ and $C_0$, and therefore on $K$ and $P(0)$), we have that, for all $t \geq t_0$, there exists $x_t\in\R$ satisfying~\eqref{defxt}. Since $u(t,\cdot)$ is even, one can assume without loss of generality that $x_t\ge0$. Recalling that $1 \geq u(t/M,x_t)$ yields $1 \geq e^{rt/(2M)}/(1+ x_t)^{1+r/4}$. This, in turn, implies
\begin{equation}\label{eq:x_t_bound1}
	x_t \geq e^{2rt/(M(4+r))} - 1\ \hbox{ for every }t\ge t_0.
\end{equation}

We now obtain a partially matching upper bound on $x_t$, for every fixed $t\ge t_0$, via an upper bound of $u$.  Indeed, let $a>0$ be such that $\supp u_0 \subset [-a,a]$, and notice that, for all $s \in [0,t/M]$,
\[
	\|K*u(s,\cdot)\|_\infty \leq \|K\|_\infty \|u(s,\cdot)\|_1 = \frac{|J|}{2} P(s) \leq \frac{|J|\,P(0)\,e^{t/M}}{2},
\]
and
\begin{equation}\label{K*ux}
\|(K*u(s,\cdot))_x\|_\infty \leq |J|
\end{equation}
from the bounds in Proposition~\ref{prop:well_posed} and from the second inequality in~\eqref{eq:convolution_bounds}, which still holds here.  Define $\overline u(s,x)$ to be the solution of
\[
	\overline u_s + [(K*u)\overline u]_x = \overline u_{xx} + \overline u,
\]
for all $(s,x) \in (0,t/M]\times \R$, with initial data $\overline u(0,\cdot) = u_0$.  Using the non-negativity of $u$, the comparison principle implies that $u(s,x) \leq \overline u(s,x)$ for all $(s,x) \in [0,t/M]\times\R$. Applying \Cref{prop:fokker_planck} to $e^{-s}\overline u(s,x)$ with $T = t/M\ge t_0/M>1$ and $A = |J|\,P(0)\,e^{t/M}/2$, there is a constant $C_1$ depending only on $J$, $a$, and $P(0)$, such that, for $|x| \ge t\,|J|\,P(0)\,e^{t/M}/(2M) + a + 1$,
\[
	e^{-t/M} u(t/M,x)
		\leq e^{-t/M} \overline u(t/M,x)
		\leq C_1e^{t/M} \sqrt{t/M}\,e^{- \frac{(|x| - t\,|J|\,P(0)\,e^{t/M}/(2M) - a)^2}{4t/M}}.
\]

If we had $x_t\ge e^{2t/M}$ for some $t\ge t_0$ large enough so that $e^{2t/M}\ge t\,|J|\,P(0)\,e^{t/M}/(2M) + a + 1$, then using the above inequality at the non-negative point $x_t$ would yield
\[
	\frac{e^{rt/2M}}{(1 + x_t)^{1 + r/4}}
		\leq u(t/M,x_t)
		\leq C_1e^{2t/M} \sqrt{t/M}\,e^{- \frac{(x_t -t\,|J|\,P(0)\,e^{t/M}/(2M) - a)^2}{4t/M}}.
\]

This leads to a contradiction if $t$ is large enough, depending on $r$, $M$, $C_1$, $P(0)$, and $a$, hence on $K$, $P(0)$, and $a$. Therefore, increasing $t_0$ if necessary, depending only on $K$, $P(0)$, and $a$, this implies that
\begin{equation}\label{eq:x_t_bound2}
	x_t \leq e^{2t/M}
\end{equation}
for every $t\ge t_0$. Since $u$ is even and non-increasing then $u(t,x) \geq u(t,x_t)$ for all $|x| \leq x_t$.  Hence, increasing again $t_0$ if necessary,
\begin{equation}\label{eq:prelim_lower_bound}
	u(t/M,x) \geq \frac{e^{rt/(2M)}}{(1 + x_t)^{1+r/4}}
		\geq \frac{e^{-2t/M}}{2^{1+r/4}}\ge e^{-3t/M}\ \hbox{ for all }t\ge t_0\hbox{ and }x\in[-x_t,x_t].
\end{equation}

For every fixed $t\ge t_0$, we now construct a sub-solution of $u$ on $[t/M, t]\times [\overline x-\pi,\overline x+\pi]$ for any $\overline x \in [-x_t + \pi, x_t - \pi]$.  Define
\[
		\mu = \frac{3 - t^{-1}M\log(2)}{M-1}.
\]
Increasing $t_0$ such that $t_0 \geq M\log(2)/3$ if necessary,  we have that $0\le\mu \leq 1/4$ because $M=13$.  Let
\begin{equation}\label{eq:underline_u}
	\underline u(s,x)
		= \frac{e^{\mu( s - t/M)-3t/M}}{1+|J|}\, \cos\left( \frac{x - \overline x}{2}\right)\!,\ \hbox{ for }(s,x)\in[t/M,t]\times[\overline{x}-\pi,\overline{x}+\pi].
\end{equation}
Then on $[t/M, t]\times [\overline x-\pi,\overline x+\pi]$, using the definition of $\mu$, we see that
\[
	0\le(1+|J|)\,\underline u(s,x)
		\leq e^{(\mu (M-1) - 3)t/M}
		= \frac{1}{2}.
\]
Hence $\underline u$ satisfies
\[
	\underline u_s - \underline u_{xx}
		= \mu \underline u + \frac{1}{4}\,\underline u
		\leq \frac{1}{2}\, \underline u
		\leq \underline{u}\,\left(1 - (1+|J|)\underline u \right)\ \hbox{ in }[t/M,t]\times[\overline{x}-\pi,\overline{x}+\pi].
\]
It follows that $\underline u$ is a sub-solution to~\eqref{eq:exponential_sub_solution} in $[t/M,t]\times[\overline{x}-\pi,\overline{x}+\pi]$. On the other hand, by the definition of $\underline u$~\eqref{eq:underline_u} along with~\eqref{eq:prelim_lower_bound}, one has $\underline u(t/M,\cdot)\leq u(t/M,\cdot)$ in $[\overline{x}-\pi,\overline{x}+\pi]$, while $\underline{u}(s,\overline{x}\pm\pi)=0\le u(s,\overline{x}\pm\pi)$ for all $s\in[t/M,t]$. Thus, one infers from the maximum principle that $\underline{u}\le u$ in $[t/M,t]\times[\overline{x}-\pi,\overline{x}+\pi]$, hence
\[
	\frac{1}{2(1+|J|)}=
		\underline u(t, \overline x)
		\leq u(t,\overline x).
\]
Thus, we have that, when $t$ is sufficiently large,
\[
	\inf_{|x| \leq x_t - \pi} u(t,x) \geq \frac{1}{2(1 +|J|)}.
\]
Let $\lambda = 2r/(M(5 + r))>0$.  From~\eqref{eq:x_t_bound1}, it follows that
\[
	\liminf_{t\to+\infty} \inf_{|x| < e^{\lambda t}} u(t,x) \geq \frac{1}{2(1 +|J|)}.
\]
This concludes the proof.
\end{proof}

\begin{rem} In the proof above, for any $\epsilon>0$, redefining $\mu_\e=(3-t^{-1}M_\e\log(1+\epsilon))/(M_\e-1)$, enlarging $M_\e$ and redefining $\underline u(s,x)= (1+|J|)^{-1}\,e^{\mu_\e( s - t/M_\e)-3t/M_\e}\,\cos((x - \overline x)/(2\gamma_\e))$ in $[t/M_\e,t]\times[\overline{x}-\gamma_\e\pi,\overline{x}+\gamma_\e\pi]$ with $\overline x \in [-x_t + \gamma_\e\pi, x_t - \gamma_\e\pi]$ and $\gamma_\e>0$ large, it is straightforward to see that there exists $\lambda_\e>0$ such that
$$\liminf_{t\to+\infty} \inf_{|x| < e^{\lambda_\e t}} u(t,x) \geq \frac{1}{(1+\epsilon)(1 +|J|)}.$$
In particular, for every family $(y_t)_{t>0}$ of positive real numbers such that $y_t=o(e^{\alpha t})$ as $t\to+\infty$ for every $\alpha>0$, then
$$\liminf_{t\to+\infty} \inf_{|x| < y_t} u(t,x) \geq \frac{1}{1 +|J|}.$$
Notice that this lower bound is coherent with the upper bound~\eqref{*20}, since $|J|\ge 2K_\infty$. Furthermore, if $K$ is constant on $(-\infty,0)$ and on $(0,+\infty)$, then $|J|=2K_\infty$ and $\sup_{|x|<y_t}|u(t,x)-(1+|J|)^{-1}|\to0$ as $t\to+\infty$ for every family $(y_t)_{t>0}$ of positive real numbers such that $y_t=o(e^{\alpha t})$ as $t\to+\infty$ for every $\alpha>0$.
\end{rem}


\subsection{The case when $K \in L^p(\R)$: proof of \Cref{thm:acceleration}-(ii)}

\begin{proof}[Proof of Theorem~$\ref{thm:acceleration}$-{\rm{(ii)}}]
First of all, when $K=0$ a.e. in $\R$, then the desired conclusion with any $p\in[1,\infty)$ follows easily from e.g. \cite{Uchiyama}. In the sequel, we then assume that $K$ is not trivial.  We begin by obtaining a pointwise upper bound on $u$.  First, for every $t\ge0$ and $x\in\R$,
\[\begin{split}
	|(K*u)(t,x)|
		= \left| \int_\R K(x-y)u(t,y)dy\right|
		\leq \|K\|_p \left( \int_\R u(t,y)^\frac{p}{p-1} dy\right)^{1 - \frac{1}{p}}
		\leq \|K\|_p\,P(t)^{1-\frac{1}{p}},
\end{split}\]
since $0 \leq u \leq 1$ by Proposition~\ref{prop:well_posed}. Together with~\eqref{eq:P_V} and~\eqref{eq:P_increasing}, this also implies that, for every $t>0$ and $s\in(0,t+1]$, $\|K*u(s,\cdot)\|_\infty\le\|K\|_p\,e^{1-1/p}\,P(t)^{1-1/p}\le\|K\|_p\,e^{1-1/p}\,P(t)^{1-1/p}$. Let $\overline u$ be the solution to
\[
	\overline u_t + ((K*u) \overline u)_x = \overline u_{xx} + \overline u\quad\hbox{ in }(0,+\infty)\times\R,
\]
with initial data $\overline u(0,\cdot) = u_0$.  As above, $\overline u$ is a super-solution of~\eqref{eq:the_equation} and $u(t,x) \leq \overline u(t,x)$ for all $(t,x) \in [0,+\infty)\times \R$.

Let $a>0$ be such that $\supp u_0 \subset [-a,a]$. Fixing $t>1$, using~\eqref{eq:P_increasing} and applying \Cref{prop:fokker_planck} to $e^{-s} \overline u(s,x)$ with $T=t$ and $A = \|K\|_p\,P(t)^{1-1/p}$, there is a constant $C_K$ depending only on $K$, $P(0)$, and $p$ such that, for all $|x| \ge t\,\|K\|_p\,P(t)^{1-1/p} + a + 1$,
\begin{equation}\label{eq:fast_super_solution}
	e^{-t} u(t,x)
		\leq e^{-t} \overline u(t,x)
		\leq C_K\,P(t)^{1-1/p}\,\sqrt{t}\,e^{- \frac{(|x|-t\|K\|_pP(t)^{1-1/p}-a)^2}{4t}}.
\end{equation}
Let
$$I_t = \big\{x\in\R : |x| < 2\,t\,\|K\|_p\,P(t)^{1-1/p} +2a+2\big\}.$$
Since $0 \leq u \leq 1$, it follows immediately that, for every $t>1$,
\[\begin{split}
	P(t) &\leq \int_{I_t} dx + \int_{I_t^c} C_K\,P(t)^{1-1/p}\,\sqrt{t}\,e^{t- \frac{(|x| - t\|K\|_p P(t)^{1-1/p} - a)^2}{4t}} dx\\
		&\leq |I_t|+ C_K\,P(t)^{1-1/p}\,\sqrt{t}\,e^t\,\frac{4t\,e^{-\frac{(t\,\|K\|_p\,P(t)^{1-1/p} +a+2)^2}{4t}}}{t\,\|K\|_p\,P(t)^{1-1/p} +a+2}\\
					&\leq 2(2\,t\,\|K\|_p\,P(t)^{1-1/p} +2a+2)+ 2\,C_K\,P(t)^{1-1/p}\,t^{3/2}\,e^{t(1- \|K\|_p^2 P(t)^{2(1-1/p)}/4)},
\end{split}\]
where we used the inequality $\int_b^{+\infty}e^{-y^2}dy\le e^{-b^2}/(2b)$ for all $b>0$.

Assume first that $p>1$ and denote $q=p/(p-1)$ the conjugate exponent of $p$. Applying Young's inequality, we obtain, for every $t>1$,
\[
	P(t)	 \leq 4(a+1)\!+\!\left(\frac{3^{p-1}\,4^p\,\|K\|_p^p\,t^p}{p\,q^{p - 1}}\!+\!\frac{P(t)}{3}\right)\!
		 +\! \left( \frac{3^{p-1}\,2^p\,C_K^p\,t^{3p/2}\,e^{pt(1- \|K\|_p^2 P(t)^{2(1-1/p)}/4)}}{p\,q^{p-1}}\!+\!\frac{P(t)}{3}\right).
\]
Re-arranging this inequality  and using~\eqref{ineqP} and $\|K\|_p>0$ clearly yields a constant $C>0$ depending on $a$, $p$, $K$, and $P(0)$ such that $P(t) \leq C(t^p +1)$ for all $t\ge 0$.

When $p=1$, then $\|K*u(s,\cdot)\|_\infty\le\|K\|_1$ for all $s\ge0$. Therefore, with $A=\|K\|_1$ and defining this time $I_t=\big\{x\in\R:|x|<M(t\|K\|_1+a+1)\}$ for some $M\ge 1$ to be chosen, the previous calculations imply similarly that $|x|-t\|K\|_1-a-1\ge(1-1/M)|x|$ for all $x\in I_t^c$ and
$$P(t)\le2M(t\|K\|_1+a+1)+C_KM(M-1)^{-2}\,t^{3/2}\,e^{t-(M-1)^2t\|K\|_1^2/4}$$
for every $t\ge1$, where $C_K$ depends only on $K$. Observe that the second term is less than $1$ for all $t\ge1$, provided $M\ge1$ is fixed large enough, depending only on $K$. Therefore, together with~\eqref{ineqP}, there is a constant $C>0$ depending on $a$, $K$, and $P(0)$, such that $P(t)\le C(t+1)$ for all $t\ge0$. The proof is thereby complete.
\end{proof}


\subsection{The case when $K(x) \geq A(1+x)^{-\alpha}$: proof of \Cref{thm:acceleration}-(iii)}\label{sec:K_alpha}

In order to prove the third case of \Cref{thm:acceleration}, we require the following lemma, which we state now and prove in the sequel.

\begin{lem}\label{lem:minimization}
	Suppose that $\phi$ is a non-negative, non-increasing and integrable function on $(0,+\infty)$.  Let $M>0$ and $X = \big\{w \in L^1(0,+\infty)\cap L^\infty(0,+\infty) : \|w\|_\infty\leq 2,\, \|w\|_1 \leq M\big\}$.  Then
	\[
		\max_{w\in X} \int_0^{+\infty} \phi(y) w(y) dy = 2\int_0^{M/2} \phi(y)dy.
	\]
\end{lem}

\begin{proof}[Proof of Theorem~$\ref{thm:acceleration}$-{\rm{(iii)}}, from Lemma~$\ref{lem:minimization}$]
Let the bulk-burning rate be
\[
	V(t) = \int_\R u(t,x)(1-u(t,x)) dx,\quad\hbox{ for }t\ge0.
\]
This proves useful for estimating $P$. First of all, notice that the a priori estimates listed in the proof of Proposition~\ref{prop:well_posed} imply that the function $V$ is well-defined and continuous on $[0,+\infty)$ and of class $C^1$ on $(0,+\infty)$. Furthermore, we obtain a differential inequality with these quantities as follows: at $t > 0$,
\begin{equation}\label{eq:bulk_burning}
\begin{split}
	V'(t)
		&=\int_\R u_t(1-2u) dx
		=\int_\R \left[ u_{xx} + u(1-u) - ((K*u) u)_x \right] (1 - 2u) dx\\
		&= 2 \int_\R u_x^2 dx + V(t) - 2\int_\R u^2(1-u)dx - 2 \int_\R (K*u) u u_x dx\\
		&\geq V(t) - 2\int_\R u^2(1-u)dx  + \int_\R (K*u)_x u^2 dx,
\end{split}
\end{equation}
where all quantities in the integrals are evaluated at $(t,x)$. A straightforward computation yields
\[
	(K*u(t,\cdot))_x(x) = |J|\,u(t,x) + \int_0^{+\infty} [u(t,x-y) + u(t,x+y)]\,K'(y)\,dy.
\]
For any fixed $t>0$, applying \Cref{lem:minimization} with $\phi(y) = - K'(y)$ (remember that $K$ is here assumed to be $C^1$, convex and bounded on $(0,+\infty)$), $M = P(t)$, and $w(y) = u(t,x-y) + u(t,x+y)$, we see that, for all $x\in\R$,
\[
	(K*u(t,\cdot))_x(x)
		\geq |J|\,u(t,x) + 2\int_0^{P(t)/2} K'(y)\,dy
		= |J|\,u(t,x) - |J| + 2K(P(t)/2).
\]
Using this inequality in~\eqref{eq:bulk_burning} along with the identity $u^2 = u - u(1-u)$, we see that, at every $t >0$,
\[\begin{split}
	V'(t) &\geq V(t) - 2 \int_\R u^2 (1-u)dx + \int_\R \left( |J| u^3 - |J|u^2 + K(P(t)/2) u^2\right) dx\\
		&= V(t) - (|J|+2) \int_\R u^2 (1-u)dx + K(P(t)/2) (P(t)-V(t)).
\end{split}\]
We now use the facts that $\int_\R u^2(1-u)dx \leq V(t)$ and that $K\in L^\infty(\R)$ to re-write this as
\begin{equation}\label{eq:P_differential_inequality1}
	V'(t) + C_K V(t) \geq K(P(t)/2) P(t),
\end{equation}
where $C_K$ depends only on $\|K\|_\infty$ and $|J|$.  On the other hand, $P(t) \geq P(0)>0$ for all $t\ge0$ (see~\eqref{eq:P_increasing}) and we assume here that $K(x)\ge A(1+x)^{-\alpha}$ for all $x>0$, with $\alpha\in(0,1)$ and $A>0$. So there exists a constant $C_0>0$ depending only on $\alpha$, $A$ and $P(0)$, and hence only on $u_0$ and $K$, such that
\[
	K(P(t)/2)\,P(t) \geq C_0\,P(t)^{1-\alpha}
\]
for all $t\ge0$. The combination of this with~\eqref{eq:P_differential_inequality1} yields the key inequality
\begin{equation}\label{eq:P_differential_inequality}
	V'(t) + C_K V(t) \geq C_0\,P(t)^{1-\alpha}\ \hbox{ for all }t>0.
\end{equation}

From~\eqref{eq:P_V}, notice that $P'(t) = V(t)$ for all $t>0$.  We claim that~\eqref{eq:P_differential_inequality} is enough to conclude.  To see this, define
$$\underline P(t) = \epsilon (t+1)^{1/\alpha}$$
for $\epsilon>0$ to be determined.  Then,
\[
	\underline P''(t) + C_K \underline P'(t)
		= \epsilon \frac{1-\alpha}{\alpha^2} (1+t)^{1/\alpha - 2} + \epsilon\frac{C_K}{\alpha} (1+t)^{1/\alpha - 1}
\]
for all $t>0$, while
\[
	C_0\,\underline P(t)^{1-\alpha} = C_0\,\epsilon^{1-\alpha}\,(1+t)^{1/\alpha -1}.
\]
Hence, taking $\epsilon>0$ sufficiently small (depending only on $\alpha$, $C_K$, and $C_0$, and hence depending only on $u_0$ and $K$), we have that, for all $t\ge0$,
\begin{equation}\label{eq:underline_P}
	\underline P''(t) + C_K \underline P'(t)
		< C_0\,\underline P(t)^{1-\alpha}.
\end{equation}
Remember now that $P(1)\ge P(0)>0$. Furthermore, $u(1,\cdot)$ is continuous, ranges in $(0,1]$ by Proposition~\ref{prop:well_posed} and $u(1,x)\to0$ as $|x|\to+\infty$ by~\eqref{gaussian}. Hence, $V(1)>0$. Therefore, decreasing $\epsilon>0$ if necessary (depending only on $u_0$ and $K$), we may assume that $\underline P(1) < P(1)$ and that $\underline P'(1) < V(1)$.

We now claim that
$$\underline P(t) \leq P(t)\ \hbox{ for all }t\geq 1.$$
We prove this by contradiction.  Hence, let
\[
	t_1 = \sup\big\{t\ge1: \underline P(s) \leq P(s)\hbox{ and }\underline P'(s) \leq V(s)\text{ for all } s \in [1,t]\big\}\,>0,
\]
and assume that $t_1 <+\infty$.  We first claim that $\underline P'(t_1) = V(t_1)$.  If $\underline P(t_1) < P(t_1)$ then this follows from the definition of $t_1$.  If $\underline P(t_1) = P(t_1)$, we argue as follows.  The fact that $\underline P(s) \leq P(s)$ for all $s\in[1,t_1]$ implies that $P - \underline P$ has a minimum of zero on $[1,t_1]$ that occurs at $t_1$.  It follows that $(P-\underline P)'(t_1) \leq 0$, which is equivalent to $\underline P'(t_1) \geq P'(t_1) = V(t_1)$.  Here we used that $P' = V$.  Continuity, along with the definition of $t_1$, then yields that $\underline P'(t_1) = V(t_1)$.  It follows that, in both cases $\underline{P}(t_1)<P(t_1)$ or $\underline{P}(t_1)=P(t_1)$, there holds
\begin{equation}\label{eq:P_P'}
	\underline P(t_1) \leq P(t_1)
		\qquad \text{ and } \qquad
	\underline P'(t_1) = V(t_1),
\end{equation}
which finishes the proof of the claim that $\underline P'(t_1) = V(t_1)$. We next claim that
\begin{equation}\label{eq:P''_V'}
	\underline P''(t_1) \geq V'(t_1).
\end{equation}
Indeed, since $\underline P'(s) \leq V(s)$ for all $s\in[0,t_1]$, then we have that $V - \underline P'$ has a minimum of zero on $[1,t_1]$ that occurs at $t_1$.  It follows that $(V - \underline P')'(t_1) \leq 0$, that is, $\underline P''(t_1) \geq V'(t_1)$, as desired. Combining the bounds on the relationship between $\underline P$ and $P$ at $t_1$,~\eqref{eq:P_P'} and~\eqref{eq:P''_V'}, with the differential inequality for $\underline P$~\eqref{eq:underline_P} and the differential inequality for $P$~\eqref{eq:P_differential_inequality}, we have that
\[
	C_0\,\underline P(t_1)^{1-\alpha}
		> \underline P''(t_1) + C_K \underline P'(t_1)
		\geq V'(t_1) + C_K V(t_1)
		\geq C_0\,P(t_1)^{1-\alpha}
		\geq C_0\,\underline P(t_1)^{1-\alpha},
\]
a contradiction.

We conclude that $\underline P(t) \leq P(t)$ for all $t \geq 1$.  This may be re-written as
\[
	P(t) \geq \epsilon\,(1+t)^{1/\alpha},
\]
for all $t\ge1$. Together with~\eqref{ineqP} and~\eqref{eq:P_increasing}, there is a constant $C>0$ depending on $u_0$ and $K$ such that $P(t)\ge C\,(1+t)^{-\alpha}$ for all $t\ge0$, which finishes the proof.
\end{proof}

We now prove \Cref{lem:minimization}.

\begin{proof}[Proof of Lemma~$\ref{lem:minimization}$]
First, taking $w = 2\chi_{[0,M/2]}$, we see that
\[
	\sup_{w\in X} \int_0^{+\infty} \phi(y)w(y)dy \geq 2\int_0^{M/2} \phi(y) dy.
\]
To show the opposite inequality, fix $\epsilon>0$ and find $w_\epsilon\in X$ so that
\begin{equation}\label{eq:w_epsilon}
	\int_0^{+\infty} \phi(y)w_\epsilon(y)dy
		\geq \sup_{w\in X} \int_0^{+\infty} \phi(y)w(y)dy - \epsilon.
\end{equation}
We assume, without loss of generality that $w_\e$ is non-negative since, otherwise, we may replace it with $|w_\e|$, which also satisfies~\eqref{eq:w_epsilon}. Let $v = 2 \chi_{[0,M/2]} - w_\epsilon\in X$ and observe that $v(y) \geq 0$ if $0<y \leq M/2$ and $v(y) \leq 0$ if $y \geq M/2$.  In addition, notice that $\int_0^{+\infty} v(y)\, dy \geq 0$.  Hence,
\[\begin{split}
	\int_0^{+\infty}\!\! \phi(y) 2 \chi_{[0,M/2]}(y) dy
		&= \int_0^{M/2}\!\! \phi(y)v(y)dy + \int_{M/2}^{+\infty}\!\!\phi(y)v(y)dy + \int_0^{+\infty} \!\!\phi(y) w_\epsilon(y) dy\\
		&\geq \int_0^{M/2}\!\! \phi(M/2)v(y)dy + \int_{M/2}^{+\infty}\!\!\phi(M/2)v(y)dy + \int_0^{+\infty}\!\! \phi(y) w_\epsilon(y) dy\\
		&\geq 0 +  \sup_{w\in X} \int_0^{+\infty} \phi(y) w(y)dy - \epsilon.
\end{split}\]
In the first inequality above, we used the sign of $v$ and the fact that $\phi$ is non-increasing.  In the second inequality, we used that $\int_0^{+\infty} v(y) dy \geq 0$ and~\eqref{eq:w_epsilon}.  Since the above is true for all $\epsilon>0$, then we obtain
\[
	2\int_0^{M/2} \phi(y)dy
		\geq \sup_{w\in X} \int_0^{+\infty} \phi(y) w(y) dy,
\]
finishing the proof.
\end{proof}


\subsection{Pointwise estimates when $K(x) \approx x^{-\alpha}$ as $x\to+\infty$}

As before, we show that \Cref{thm:acceleration}-(ii) and \Cref{thm:acceleration}-(iii) imply pointwise propagation.  A pointwise upper bound was constructed explicitly in the proof of~\Cref{thm:acceleration}-(ii) so we are concerned only with estimates from below.

First, we point out that we may argue as in \Cref{cor:acceleration_pointwise1} to obtain pointwise bounds when $u_0$ is even and radially non-increasing.  In order to see a more general argument, we show the following stronger version of \Cref{cor:acceleration_pointwise}.

\begin{corollary}\label{cor:acceleration_pointwise2}
Under the assumptions of Theorem~$\ref{thm:acceleration}$-{\rm{(iii)}}, suppose that there exist some constants $A\ge1$ and $\alpha \in (0,1)$ such that
	\[
		\frac{1}{A(1+x)^{\alpha}} \leq K(x) \leq \frac{A}{(1+x)^{\alpha}}
	\]
for all $x\in \R$. Then, for any $\delta \in (0,1/6)$, there exists a constant $C_0>0$ depending only on $u_0$, $K$ and $\delta$, such that
	\[
		\liminf_{t\to+\infty}\,t^{-1/\alpha}\Big(\frac{1}{t}\int_{0}^t \big|\{ x\in\R : u(s,x) \geq 1 - t^{-\delta}\}\big|\,ds\Big)\ge C_0.
	\]
\end{corollary}

\begin{proof}
Fix $\delta \in (0,1/6)$.  Due to \Cref{thm:acceleration} and our assumptions on $K$, we know that there is a constant $C_1\ge1$ depending on $u_0$, $K$, and $\delta$ such that, for all $t\geq 1$,
\begin{equation}\label{eq:P_alpha}
	\frac{t^{1/\alpha}}{C_1}
		\leq P(t)
		\leq C_1\,t^{1/\alpha + \delta/2}.
\end{equation}

Define the time averaged burning rate
\[
	\overline V(t) := \frac{1}{t} \int_{0}^t V(s)ds
\]
for $t>0$. Using the relationship $P' = V$ in $(0,+\infty)$, we have that
\begin{equation}\label{eq:burning_bound}
		\overline V(t)
		= \frac{1}{t} \int_0^t P'(s)ds
		= \frac{1}{t} \left(P(t) -P(0)\right)
		\leq C_1\,t^{1/\alpha -1 + \delta/2}
\end{equation}
for all $t\ge1$. For any $\e\in(0,1/2)$ and $s>0$, define the {\em good} set, the {\em bad} set, and the {\em tail} set
\[\begin{split}
	&G_\epsilon(s) = \{x\in\R : u(s,x) > 1- \epsilon\},\quad
	B_\epsilon(s) = \{x\in\R : \epsilon < u(s,x) < 1-\epsilon\},\\
	&\text{ and }\qquad
	T_\epsilon(s) = \{x\in\R : u(s,x) \leq \epsilon\}.
\end{split}\]
Our goal is to obtain lower bounds on $|G_\epsilon|$ in an averaged sense with the choice $\epsilon = t^{-\delta}$ and $t\ge1$ large enough so that $0<t^{-\delta}<1/2$.

Firstly, we notice that $B_\epsilon(t)$ cannot be too big.  In particular, we have that, for any $t>0$ and $\epsilon\in(0,1/2)$,
\[
	\epsilon(1-\epsilon)|B_\epsilon(t)|
		\leq \int_{B_\epsilon(t)} u(t,x)\,(1-u(t,x))\, dx
		\leq V(t).
\]
Using this with~\eqref{eq:burning_bound}, we have that, for every $t\ge1$ and $\epsilon\in(0,1/2)$,
\begin{equation}\label{eq:bad_set}
	\frac{2}{t} \int_{t/2}^t |B_\epsilon(s)| ds
		\leq \frac{2}{t} \int_{0}^t |B_\epsilon(s)| ds
		\leq \frac{2\,C_1\,t^{1/\alpha - 1 + \delta/2}}{\epsilon(1-\e)} 
\end{equation}

Secondly, using the upper bound for $u$~\eqref{eq:fast_super_solution} with any $p\in(1/\alpha,+\infty)$ along with the bounds on $P$~\eqref{eq:P_alpha}, it is straightforward to check that
\begin{equation}\label{eq:tail_estimate}
	\lim_{t\to+\infty} \int_{\{|x|>t^{1/\alpha + \delta/2}\}} u(t,x)\,dx = 0.
\end{equation}

With these inequalities in hand, we obtain the desired lower bound for $|G_{t^{-\delta}}(s)|$ in an averaged sense.  Indeed, using $\|u\|_\infty \leq 1$,~\eqref{eq:P_increasing},~\eqref{eq:P_alpha},~\eqref{eq:bad_set}, and  the choice $\e = t^{-\delta}$, we have, for all $t$ large enough,
\[\begin{split}
	\frac{t^{1/\alpha}}{2^{1/\alpha}C_1}
		&\leq \frac{2}{t} \int_{t/2}^t P(s)ds\\
		&\leq \frac{2}{t} \int_{t/2}^t \left[ \int_{B_\epsilon(s)} u(s,x) dx + \int_{G_\epsilon(s)} u(s,x) dx + \int_{T_\epsilon(s) \cap[-s^{1/\alpha+\delta/2},s^{1/\alpha+\delta/2}]} u(s,x) dx\right.\\
		& \qquad\qquad\left. + \int_{\{|x|>s^{1/\alpha + \delta/2}\}} u(s,x) dx\right]ds\\
		&\leq \frac{2}{t} \int_{t/2}^t \left[ |B_\epsilon(s)| + |G_\epsilon(s)| + 2\epsilon s^{1/\alpha+\delta} + \int_{\{|x|>s^{1/\alpha + \delta/2}\}} u(s,x) dx\right]ds\\
		&\leq 8C_1 t^{1/\alpha - 1 + 3\delta/2} + \frac{2}{t} \int_0^t |G_\epsilon(s)|ds + 2t^{1/\alpha - \delta/2} + \frac{2}{t}\int_{t/2}^t\int_{\{|x|>s^{1/\alpha + \delta/2}\}} u(s,x) dx ds.
\end{split}\]
Using~\eqref{eq:tail_estimate}, it follows that the fourth term tends to zero.  Since, by assumption, $3\delta/2 < 1$, then, if $t$ is sufficiently large, we may absorb the first, third, and fourth terms from the right hand side into the left hand side.  We obtain, for all $t$ sufficiently large, that
\[
	\frac{t^{1/\alpha}}{2^{1/\alpha+1}\,C_1} 
		\leq \frac{2}{t}\int_0^t |G_{t^{-\delta}}(s)| ds.
\]
This concludes the proof.
\end{proof}


\section{Proofs of Propositions~\ref{prop:heat_kernel} and~\ref{prop:fokker_planck} and Lemma~\ref{lem:harnack}}\label{sec:lemmas}


\subsection{Heat kernel estimates: proof of Proposition~\ref{prop:heat_kernel}}

In this section, we prove~\Cref{prop:heat_kernel}.  We investigate the heat kernel for the equation
\begin{equation}\label{eq:heat_equation}
	w_t + (v_x w)_x = w_{xx} \qquad \text{ in } (s,+\infty)\times\R
\end{equation}
with $s\ge0$, where $v$ satisfies~\eqref{boundsv0}. We follow the Fabes-Stroock~\cite{FabesStroock} approach.  First, we look at weighted functions
\begin{equation}\label{eq:phi_alpha_def}
	\phi_\alpha(t,x) = e^{-\alpha x} w_\alpha(t,x)
\end{equation}
for $(t,x)\in[s,+\infty)\times\R$, where $\alpha\in\R$ and $w_\alpha$ solves~\eqref{eq:heat_equation} with initial data
\begin{equation}\label{walpha}
w_\alpha(s,x)=g(x)\,e^{\alpha x},
\end{equation}
for a fixed $g\in C_c(\R)$.   Then we prove the following bound.

\begin{lem}\label{weighted_upper_bound}
For any $\delta>0$, there exist a constant $C_\delta>0$, depending on $\delta$, and a constant $\tilde R>0$, depending only on $A_0$, $A_1$, and $A_2$,  so that, for every $\alpha\in\R$, $g\in C_c(\R)$ and $t>s\ge0$,
\[
	\|\phi_\alpha(t,\cdot)\|_\infty \leq \frac{C_\delta}{\sqrt{t-s}}\,e^{R_\alpha (t-s) + \delta \tilde R (\alpha^2 + 1) (t-s)} \|g\|_1,
\]
where $R_\alpha$ is given by
\[
	R_\alpha
		= \min\left\{\inf_{\epsilon\in(0,1)}\left[\alpha^2\left(1 + \frac{A_0^2}{1-\e}\right) + \frac{A_1^2}{4\e}\right] , \alpha^2(1 + A_0^2) + \frac{A_2}{2}
		 \right\}.
\]
\end{lem}

Before beginning, we note the following abuse of notation.  Since our estimates depend only on $t-s$, we set $s = 0$ for the remainder of this section and obtain the desired bounds.  The general case is straightforward.

Let us now explain how the bound on $\Gamma$ follows from \Cref{weighted_upper_bound}.

\begin{proof}[Proof of Proposition~$\ref{prop:heat_kernel}$, from Lemma~$\ref{weighted_upper_bound}$]
Let $P_t^\alpha$ be the solution operator which gives us $\phi_\alpha$ from $g$.  Notice that then
\[
	P_t^\alpha g(x)
		= e^{-\alpha x} \int_\R \Gamma(t,0,x,y) g(y) e^{\alpha y}dy
		= \int_\R K_\alpha(t,x,y) g(y)dy
\]
for every $t>0$ and $x\in\R$, where $K_\alpha(t,x,y):= e^{-\alpha (x-y)}\Gamma(t,0,x,y)$ and $\Gamma$ is the heat kernel for~\eqref{eqGamma}, that is,~\eqref{eq:heat_equation} with Dirac mass at $y$ as initial condition.  The inequality in \Cref{weighted_upper_bound} gives us by duality that, for any $t>0$ and $(x,y)\in\R\times\R$,
\[
	0\le K_\alpha(t,x,y)
		\leq \frac{C_\delta}{\sqrt{t}}\,e^{R_\alpha t + \delta \tilde R(\alpha^2 + 1)t},
\]
which in turn gives us that
\[
	\Gamma(t,0,x,y) \leq \frac{C_\delta}{\sqrt{t}}\,e^{R_\alpha t + \delta \tilde R(\alpha^2 + 1)t+\alpha (x-y)}.
\]
This holds for all $\alpha\in\R$ so we may optimize in the following ways.  Notice that we may write each term in the infimum in $R_\alpha$ as $a_1\alpha^2 + a_2$.  More precisely the pairs $(a_1,a_2)$ of the type
\begin{equation}\label{eq:pairs}
	(a_1, a_2) = \left(1 + \frac{A_0^2}{1-\e}, \frac{A_1^2}{4\e}\right)\ \hbox{ with }0<\e<1
	\qquad \text{ or } \qquad
	(a_1,a_2) = \left(1 + A_0^2, \frac{A_2}{2}\right).
\end{equation}
Hence, in each case, choosing $\alpha = -(x-y)/(2t(\delta\tilde R + a_1))$, yields
\[
	\Gamma(t,0,x,y)
		\leq \frac{C_\delta}{\sqrt t}\, \exp\left\{ (a_2 + \delta \tilde R)t - \frac{(x-y)^2}{4t(\delta \tilde R + a_1)}\right\}.
\]
Substituting the values $(a_1,a_2)$ from~\eqref{eq:pairs} into the right hand side above, and changing $\delta\tilde{R}$ into $\delta$ (remember that $\tilde{R}$ does not depend on $\delta$) yields~\eqref{eq:kernel_bound}, finishing the proof.
\end{proof}


\subsection{Proof of \Cref{weighted_upper_bound}}

\begin{proof}[Proof of Lemma~$\ref{weighted_upper_bound}$]
We repeat that we can assume $s=0$ without loss of generality. The proof of Lemma~\ref{weighted_upper_bound} proceeds as follows.  First, we obtain an estimate on the growth of $\|\phi_\alpha(t,\cdot)\|_2$ in time $t$.  Then, using the Nash inequality, we obtain a sequence of inequalities on the $L^{2^{k+1}}$ norm in terms of the $L^{2^k}$ norm.  Iterating this procedure and using the growth in $L^2$ as a boundary condition, we obtain an $L^2 \to L^\infty$ estimate.  Using duality, this gives us an $L^1 \to L^2$ and then an $L^1 \to L^\infty$ estimate.

In order to apply duality, we obtain bounds for a slightly more general equation.  For any $v^{(1)}, v^{(2)}$ satisfying the same conditions~\eqref{boundsv0} as $v$, we investigate bounds for the equation
\begin{equation}\label{eq:heat_equation1}
	w_t + v_x^{(1)} w_x + v^{(2)}_{xx} w = w_{xx} \qquad \text{ in } (0,+\infty)\times\R.
\end{equation}
Let $\phi_\alpha(t,x) = e^{-\alpha x} w_\alpha(t,x)$ for $t\ge0$ and $x\in\R$, where $\alpha\in\R$ is fixed and $w_\alpha$ solves~\eqref{eq:heat_equation1} with initial condition
\begin{equation}\label{walpha0}
w_\alpha(0,x)=g(x)\,e^{\alpha x}
\end{equation}
and a fixed non-zero function $g\in C_c(\R)$

\subsubsection*{$L^2$ growth of $\phi_\alpha$}

Define
$$M_p(t) := \|\phi_\alpha(t,\cdot)\|_{L^{2p}(\R)}$$
for $p\ge1$ and $t\ge0$. All quantites $M_p(t)$ are positive real numbers. As mentioned above, the first step is in obtaining the inequality
\begin{equation}\label{eq:FS1}
	M_1(t) \leq e^{R_{\alpha,v^{(1)},v^{(2)}} t} M_1(0)
\end{equation}
for all $t\ge0$, with $R_{\alpha, v^{(1)},v^{(2)}}$ given in~\eqref{eq:R_alpha_v} below.

To this end, using~\eqref{eq:heat_equation} we have that $\phi_\alpha$ satisfies
\begin{equation}\label{eq:phi_alpha}
	e^{\alpha x} (\phi_\alpha)_t + v_x^{(1)}(e^{\alpha x}\phi_\alpha)_x + v_{xx}^{(2)} e^{\alpha x} \phi_\alpha = (e^{\alpha x} \phi_\alpha)_{xx} \qquad \text{ in } (0,+\infty)\times \R.
\end{equation}
Multiplying this by $e^{-\alpha x} \phi_\alpha(t,x)$ and integrating by parts gives us, for all $t>0$,
\begin{equation}\label{eq:L2_first_step}
\begin{split}
	\frac{1}{2}\,\frac{dM_1^2}{dt}(t)
		&= -\!\!\int_\R\!( (\phi_\alpha)_x\!-\!\alpha \phi_\alpha)((\phi_\alpha)_x\!+\!\alpha \phi_\alpha) dx\!-\!\!\int_\R\!\! v_{xx}^{(2)} \phi_\alpha^2 dx \!-\!\! \int_\R\!\! ((\phi_\alpha)_x \!+\! \alpha \phi_\alpha) \phi_\alpha v^{(1)}_x dx\\
		&= \alpha^2 M_1^2(t) \!-\!\! \int_\R\! |(\phi_\alpha)_x|^2 dx \!+\! 2\!\!\int_\R\! (\phi_\alpha)_x \phi_\alpha \Big(v^{(2)}_x \!-\!\frac{v^{(1)}_x}{2}\Big) dx \!-\! \alpha\! \int_\R\! \phi_\alpha^2 v_x^{(1)} dx,
\end{split}
\end{equation}
where all functions in the integrals are evaluated at $(t,x)$. Notice that all integrals converge since $w_\alpha(t,x)$ and then $\phi_\alpha(t,x)$ satisfy Gaussian estimates as $|x|\to+\infty$ locally uniformly in $t\in(0,+\infty)$, see~\cite{Aronson}. Notice also that the function $M_1$ is continuous on $[0,+\infty)$ and of class $C^1$ on $(0,+\infty)$, as are the functions $M_p$ for all $1\le p<+\infty$.

We now estimate the last two terms in~\eqref{eq:L2_first_step}.  Depending on $A_0$, $A_1$, and $A_2$, there are different ways to estimate the third term.  It may be estimated, for any $\epsilon_1 \in (0,1)$, as
\begin{equation}\label{eq:L2_v_x_split}
		2\int_\R (\phi_\alpha)_x \phi_\alpha \left( v_x^{(2)} - \frac{v_x^{(1)}}{2}\right) dx
			\leq \epsilon_1 \int_\R |(\phi_\alpha)_x|^2 dx + \frac{\left\|v_x^{(2)} - \frac{v_x^{(1)}}{2}\right\|_\infty^2}{\epsilon_1} \int_\R |\phi_\alpha|^2 dx,
\end{equation}
or it may be estimated as
\begin{equation}\label{eq:L2_v_xx}
	2\int_\R\! (\phi_\alpha)_x \phi_\alpha \left( v_x^{(2)} - \frac{1}{2} v_x^{(1)}\right) dx
		= - \int_\R\! \phi_\alpha^2 \left( v_{xx}^{(2)} - \frac{v^{(1)}_{xx}}{2}\right) dx
		\leq \left\|v_{xx}^{(2)} - \frac{v_{xx}^{(1)}}{2}\right\|_\infty \int_\R\! \phi_\alpha^2 dx.
\end{equation}
We bound the last term in~\eqref{eq:L2_first_step}, for any $\epsilon_2 \in (0,1]$, as follows:
\begin{equation}\label{eq:L2_v}
	-\alpha \int_\R \phi_\alpha^2 v^{(1)}_x dx
		= 2\alpha \int_\R \phi_\alpha (\phi_\alpha)_x v^{(1)} dx
		\leq \epsilon_2 \int_\R |(\phi_\alpha)_x|^2 dx + \frac{\alpha^2 A_0^2}{\epsilon_2} \int_\R \phi_\alpha^2 dx.
\end{equation}

Combining the estimates above gives us~\eqref{eq:FS1}.  Indeed, we obtain the two different choices of $R_\alpha$ as follows: on the one hand, using~\eqref{eq:L2_v_x_split} and~\eqref{eq:L2_v}, with $\epsilon_1 + \epsilon_2 = 1$, we obtain
\[
	\frac{1}{2}\,\frac{dM_1^2}{dt}(t)
		\leq  \left(\alpha^2 + \frac{\left\|v_x^{(2)} - \frac{v_x^{(1)}}{2}\right\|_\infty^2}{\epsilon_1} + \frac{\alpha^2 A_0^2}{1-\epsilon_1} \right) M_1^2(t);
\]
on the other hand, using~\eqref{eq:L2_v_xx} and~\eqref{eq:L2_v}, with $\epsilon_2 =1$, we obtain
\[
	\frac{1}{2}\,\frac{dM_1^2}{dt}(t)
		\leq \left( \alpha^2 + \alpha^2 A_0^2 + \left\|v_{xx}^{(2)} - \frac{v_{xx}^{(1)}}{2}\right\|_\infty\right) M_1^2(t).
\]
Defining
\begin{equation}\label{eq:R_alpha_v}
	R_{\alpha,v^{(1)},v^{(2)}}
		\!\!=\! \min\!\left\{\!
			\inf_{\epsilon_1\in(0,1)}\!\!\left[\alpha^2\!\left(\!1 \!+\! \frac{A_0^2}{1\!-\!\epsilon_1}\!\right)\!\!+\! \frac{\left\|v_x^{(2)} \!-\! \frac{v_x^{(1)}}{2}\right\|_\infty^2}{\epsilon_1}\right]\!\!,
			\alpha^2(1\!+\! A_0^2) \!+\!\! \left\|v_{xx}^{(2)} \!-\! \frac{v_{xx}^{(1)}}{2}\right\|_\infty\!
		\right\}\!,
\end{equation}
we obtain $(M_1^2)'(t)\le 2 R_{\alpha,v^{(1)},v^{(2)}}  M_1^2(t)$ for any $t>0$.  This differential inequality yields exactly~\eqref{eq:FS1}.

\subsubsection*{An $L^{p/2} \to L^p$ inequality}

We now obtain the inequality
\begin{equation}\label{eq:FS2}
	\frac{1}{2p}\,\frac{dM_p^{2p}}{dt}(t) \leq p\, \tilde R\,(\alpha^2 + 1)\,M_p^{2p}(t) - \frac{C}{p}\,\frac{M_p^{6p}(t)}{M_{p/2}^{4p}(t)}
\end{equation}
for every $t>0$ and $p=2^k$ with $k\in\N$ ($k\ge1$), where $\tilde{R}>0$ is a constant depending only on $A_0$, $A_1$ and $A_2$, and $C>0$ is a universal constant. Recall the estimates~\eqref{eq:FS1} for $M_1(t)$; here, our arguments above give us a ``boundary condition''.  With this, we may close the system of inequalities by considering $p = 2^k$ for any $k\in\N$ ($k\ge1$) and obtain the estimate using an ODE argument.

To obtain~\eqref{eq:FS2}, we begin by multiplying the equation for $\phi_\alpha$~\eqref{eq:phi_alpha} by $e^{-\alpha x}\phi_\alpha^{2p-1}(t,x)$ (for $p=2^k$, $k\in\N$, $k\ge1$) and integrating by parts, which yields:
\[\begin{split}
	\frac{1}{2p} \frac{d}{dt} \int_\R \phi_\alpha^{2p} dx
	 	&= \underbrace{-\int_\R \left((2p-1) \phi_\alpha^{2p-2} (\phi_\alpha)_x - \alpha \phi_\alpha^{2p-1}\right) \left(\alpha \phi_\alpha + (\phi_\alpha)_x\right)_xdx}_{:= I_1}\\
		&\quad \underbrace{-\int_\R \left((2p-1)\phi_\alpha^{2p-2} (\phi_\alpha)_x + \alpha \phi_\alpha^{2p-1} \right)\phi_\alpha v^{(1)}_x dx - \int_\R v_{xx}^{(2)} \phi_\alpha^{2p} dx}_{:=I_2}
\end{split}\]
for every $t>0$. First, we re-write and then estimate $I_1$ using Young's inequality to obtain
\[\begin{split}
	I_1
		&= \alpha^2 \int_\R \phi_\alpha^{2p} dx
			- (2p-1) \int_\R \phi_\alpha^{2p-2} ((\phi_\alpha)_x)^2 dx
			- 2(p-1) \alpha \int_\R \phi_\alpha^{2p-1} (\phi_\alpha)_x dx\\
		&\leq \alpha^2 p \int_\R \phi_\alpha^{2p} dx
			- \frac{1}{p} \int_\R ((\phi_\alpha^p)_x)^2 dx
\end{split}\]
Now we estimate the second term $I_2$.  To that end, we use our bounds on $v^{(1)}$ and $v^{(2)}$, together with Young's inequality, to obtain
\[\begin{split}
	I_2
		&\leq\|v_x^{(1)}\|_\infty \int_\R |\phi_\alpha|^{2p-1} |(\phi_\alpha)_x|dx + (\alpha\|v_x^{(1)}\|_\infty + \|v_{xx}^{(2)}\|_\infty) \int_\R \phi_\alpha^{2p}dx\\
		&\leq \frac{\tilde R}{2}\,(\alpha^2+1) \int_\R \phi_\alpha^{2p} dx + \frac{1}{2p^2}\int_\R ((\phi_\alpha^p)_x)^2 dx,
\end{split}\]
where $\tilde R\ge2$ is a constant depending only on $A_0$, $A_1$, and $A_2$.  For notational ease, let
$$\tilde R_\alpha := \tilde R\,(\alpha^2+1).$$
Then, since $p\ge1$, one infers that
\[
	\frac{1}{2p}\, \frac{dM_{p}^{2p}}{dt}(t)
		\leq p\,\tilde R_{\alpha}\,M_p^{2p}(t) - \frac{1}{2p} \int_\R ((\phi_\alpha^p)_x)^2 dx
\]
for all $t>0$. From the Nash inequality, there is a universal constant $C>0$ such that
\begin{equation}\label{nash}
2\,C\,\|\psi\|_{L^2(\R)}^6\le\|\psi'\|_{L^2(\R)}^2\,\|\psi\|_{L^1(\R)}^4
\end{equation}
for all $\psi\in H^1(\R)\cap L^1(\R)$. Applying this inequality to $\psi=\phi_\alpha^p(t,\cdot)$ yields~\eqref{eq:FS2}.

\subsubsection*{From~\eqref{eq:FS2} to an $L^\infty$ bound}

Letting now $G_p(t) = M_p(t) e^{-\tilde R_\alpha p t}$, we have that, for every $t>0$ and $p=2^k$ with $k\in\N$ ($k\ge1$),
\begin{equation}\label{ineqGp}
	\frac{1}{4p}\,\frac{dG_p^{-4p}}{dt}(t)
		\geq \frac{C}{p}\,\frac{e^{4p^2 \tilde R_\alpha t}}{M_{p/2}^{4p}(t)}.
\end{equation}
Now, we define
$$\overline M_p(t) = \sup_{s \in (0,t]} s^{(p-1)/4p}M_p(s)$$
with the intention of leveraging the fact that $\overline M_p$ is non-decreasing in $t$ (notice that the supremum could be taken over $[0,t]$ without any change).  Fix any $\delta \in (0,1)$.  Then the above equation~\eqref{ineqGp} becomes, after substituting, integrating, and estimating the integral,
\[\begin{split}
	\frac{1}{G_p^{4p}(t)}
		&\geq 4\,C \int_0^t\frac{s^{p-2}e^{4p^2 \tilde R_\alpha s}}{s^{p-2}M_{p/2}^{4p}(s)} ds
		\geq \frac{4\,C}{\overline M_{p/2}^{4p}(t)} \int_{(1-\delta/p^2)t}^t s^{p-2} e^{4p^2 \tilde R_\alpha s}ds\\
		&\geq 4\,C\,\frac{t^{p-1}}{p}\,e^{(1- \delta/p^2) 4p^2 \tilde R_\alpha t} \left(1 - \left(1 - \frac{\delta}{p^2}\right)^{p-1}\right) \overline M_{p/2}^{-4p}(t),
\end{split}\]
for every $t>0$ and $p=2^k$ with $k\in\N$, $k\ge1$. There exists a constant $C_\delta > 0$, depending only on $\delta$, such that $1 - (1-\delta/p^2)^{p-1} \geq 1/(C_\delta p)$ for all $p=2^k$ with $k\in\N$, $k\ge1$.  Hence we have that
\[
	G_p^{4p}(t)
		\leq C_\delta\,\frac{p^2}{t^{p-1}}\,e^{-(1- \delta/p^2) 4p^2 \tilde R_\alpha t}\,\overline M_{p/2}^{4p}(t),
\]
for every $t>0$ and $p=2^k$ with $k\in\N$, $k\ge1$, where we have absorbed the universal constant $C$ given by~\eqref{nash} into $C_\delta$. Re-writing this in terms of $M_p$ yields
\begin{equation}\label{eq:kernel_last_step1}
	M_p(t)
		\leq C_\delta^{1/4p} \left(\frac{p^2}{t^{p-1}}\right)^{1/4p} e^{\delta \tilde R_\alpha t/ p}\,\overline M_{p/2}(t).
\end{equation}
Fixing $t>0$, considering the above inequality at any $s\in(0,t]$, multiplying by $s^{(p-1)/4p}$ and taking the supremum over $(0,t]$, we obtain
\begin{equation}\label{eq:kernel_last_step2}
	\overline M_p(t)
		\leq C_\delta^{1/4p} p^{1/(2p)} e^{\delta \tilde R_\alpha t/ p}\,\overline M_{p/2}(t)
\end{equation}
for every $t>0$ and $p=2^k$ with $k\in\N$, $k\ge1$. Plugging~\eqref{eq:kernel_last_step2} into~\eqref{eq:kernel_last_step1} $k$ times with $p = 2^k, 2^{k-1}, \dots, 2^1$, we obtain, for every $t>0$,
\[\begin{split}
	\|\phi_\alpha(t,\cdot)\|_\infty
		\le \limsup_{k\to+\infty} M_{2^k}(t) & \le\limsup_{k\to+\infty}\Big(t^{-(2^k-1)/(4\cdot 2^k)}\overline{M}_{2^k}(t)\Big)\\
		& \leq t^{-1/4}\limsup_{k\to+\infty}\left(C_\delta^{\sum_{\ell=1}^k 2^{-2-\ell}} \left(\prod_{\ell=1}^k 2^\frac{\ell}{2^{\ell+1}}\right) e^{\delta \tilde R_\alpha t \sum_{\ell=1}^k 2^{-\ell}}\, \overline M_1(t)\right).
\end{split}\]
Using the summability of $\ell\,2^{-\ell}$, the definition of $\overline M_1(t)$ and~\eqref{eq:FS1}, we get that, for every $t>0$,
\begin{equation}\label{eq:FS3}
	\|\phi_\alpha(t,\cdot)\|_\infty\leq \frac{C_\delta'}{t^{1/4}}\, e^{(R_{\alpha,v^{(1)},v^{(2)}} + \delta \tilde R_\alpha) t}\,\|\phi_\alpha(0,\cdot)\|_2=\frac{C_\delta'}{t^{1/4}}\, e^{(R_{\alpha,v^{(1)},v^{(2)}} + \delta \tilde R_\alpha) t}\,\|g\|_2,
\end{equation}
where $C_\delta'>0$ is a constant depending only on $\delta$. 

\subsubsection*{From an $L^2 \to L^\infty$ bound to an $L^1 \to L^\infty$ bound}

To conclude we use a standard technique. Let $S_{\alpha,v^{(1)},v^{(2)}}(t;s)$ be the solution operator sending the initial data $g$ at time $s\ge0$ to $\phi_\alpha$ at time $t>s$ where $\phi_\alpha$ is defined by~\eqref{eq:phi_alpha_def},~\eqref{eq:heat_equation1} and~\eqref{walpha0}.  Then~\eqref{eq:FS3} with the choice $v^{(1)} = v^{(2)} = v$ implies that
\begin{equation}\label{eq:solution_bound}
	\|S_{\alpha,v,v}(t;s)\|_{L^2 \to L^\infty}
		\leq \frac{C_\delta'\,e^{(R_{\alpha,v,v} + \delta \tilde R_\alpha)(t-s)}}{(t-s)^{1/4}}
\end{equation}
for all $t>s\ge0$. On the other hand,~\eqref{eq:FS3} with the choice $v^{(1)} = -v$ and $v^{(2)} = 0$ and with replacing $\alpha$ by $-\alpha$ yields that
\begin{equation}\label{eq:adjoint_bound}
	\|S_{-\alpha,-v,0}(t;s)^*\|_{L^1 \to L^2}
		\leq \frac{C_\delta'\,e^{(R_{-\alpha,-v,0}+ \delta \tilde R_{-\alpha})(t-s)}}{(t-s)^{1/4}}
\end{equation}
for all $t>s\ge0$, where $S_{-\alpha,-v,0}(t;s)^*$ is the adjoint operator of $S_{-\alpha,-v,0}(t;s)$, or the solution operator to the adjoint equation. Since, by a straightforward computation $S_{-\alpha,-v,0}(t;s)^* = S_{\alpha,v,v}(t;s)$, and since $R_{-\alpha,-v,0}=R_{\alpha,-v,0}$ and $\tilde{R}_{-\alpha}=\tilde{R}_{\alpha}$, we have that, for problem~\eqref{eq:heat_equation}-\eqref{walpha},
\[\begin{split}
	\|\phi_\alpha(t,\cdot)\|_\infty
		&= \|S_{\alpha,v,v}(t;t/2)S_{\alpha,v,v}(t/2;0)g\|_\infty
		= \|S_{\alpha,v,v}(t;t/2)S_{-\alpha,-v,0}(t/2;0)^*g\|_\infty\\
		&\leq \frac{2^{1/4}C_\delta'e^{(R_{\alpha,v,v}\!+\!\delta \tilde R_\alpha)t/2}}{t^{1/4}} \|S_{-\alpha,-v,0}(t/2;0)^*g\|_2
		\leq \frac{2^{1/2}(C_\delta')^2e^{((R_{\alpha,v,v}\!+\!R_{\alpha,-v,0})/2\!+\!\delta \tilde R_\alpha)t}}{t^{1/2}} \|g\|_1
\end{split}\]
for all $t>0$. To conclude, we simply note that
\[
	\frac{R_{\alpha, v,v} + R_{\alpha,-v,0}}{2}
		\leq \min\left\{\inf_{\epsilon\in(0,1)}\left[\alpha^2\left(1 + \frac{A_0^2}{1-\e}\right) + \frac{A_1^2}{4\e}\right] , \alpha^2(1 + A_0^2) + \frac{A_2}{2}
		 \right\} =: R_\alpha.
\]
\end{proof}

We note that, in the last step, we could have replaced $R_\alpha$ with a sharper, more complicated bound on $(R_{\alpha,v,v} + R_{\alpha,-v,0})/2$.  Since this does not provide any benefits when considering the asymptotic limit $A_i \to 0$ for all $i\in\{0,1,2\}$, we omit it.


\subsection{Upper bounds on the tails: proof of \Cref{prop:fokker_planck}}

In order to prove \Cref{prop:fokker_planck}, we begin by stating a bound due to Hill.  We use this bound in the sequel, as well, to derive our local-in-time Harnack inequality.

\begin{lem}[{\cite[Theorem~2.1]{Hill}}] \label{lem:weak_heat_kernel}
	Suppose that $v:(0,+\infty)\times\R^n\to\R^n$ is a bounded vector field with $\||v|\|_{L^{\infty}((0,+\infty)\times\R^n)}\le A<+\infty$. Let $\Gamma(t,s,x,y)$ be the fundamental solution to
\begin{equation}\label{gammav}
\begin{cases}
		\Gamma_t + v\cdot\nabla_x\Gamma = \Delta_x \Gamma \qquad \qquad \text{ in } (s,+\infty)\times \R^n,\\
		\Gamma(t=s,s,x,y) = \delta_y(x)
	\end{cases}
\end{equation}
for any $t > s\ge0$ and $y\in\R^n$. Then
\[
	\Gamma(t,s,x,y)
		\leq \left(\frac{1}{\sqrt{4\pi(t-s)}}+\frac{A}{2}\right)^n
\]
for all $t>s\ge0$ and $(x,y)\in\R^n\times\R^n$, and
$$\Gamma(t,s,x,y)\leq\left(\!\frac{1}{\sqrt{4\pi(t\!-\!s)}}\!+\!\frac{A}{2}\!\right)^{n-1}\!\!\!\!\times\left(\!\frac{1}{\sqrt{4\pi(t\!-\!s)}}\!+\!\frac{A\sqrt{t\!-\!s}}{\sqrt{4\pi}(|x\!-\!y|\!-\!A(t\!-\!s))}\!\right)e^{-\frac{(|x\!-\!y|\!-\!A(t\!-\!s))^2}{4(t\!-\!s)}}$$
if $|x-y|>A\,(t-s)$. Further, if $|x-y| > A\,\sqrt{n}\,(t-s)$ then
\begin{equation}\label{boundshill}
	\frac{e^{-\frac{(|x-y| + A\,\sqrt{n}\,(t-s))^2}{4(t-s)}}}{(16\pi (t-s))^{n/2}}
		\leq \Gamma(t,s,x,y).
\end{equation}
\end{lem}

We briefly mention how to obtain these bounds from \cite[Theorem~2.1]{Hill}.  We first show how to obtain the lower bound~\eqref{boundshill}. Notice that this bound is invariant under rotation of $x-y$ and under translation in $(t,s)$ and $(x,y)$, and that the bound on $|v|$ is also invariant by translation and rotation of the frame. We may then assume $s=0<t$, $y=0$ and $|x_i| = |x|/\sqrt n> A\,t$ for all $i = 1,\dots, n$ since we may otherwise rotate and translate the entire system. Hill \cite[Theorem~2.1]{Hill} proves that
\[
	\Gamma(t,0,x,0)
		\geq \prod_{i=1}^n \left(\frac{1}{\sqrt{4\pi t}}\,e^{-\frac{(|x_i| + A_it)^2}{4t}} - \frac{A_i}{4}\erfc\left(\frac{|x_i|}{2\sqrt t} + \frac{A_i\sqrt t}{2} \right) \right),
\]
where $\erfc(z):= 2 \pi^{-1/2} \int_z^\infty e^{-s^2}ds$ for any $z \in \R$ and where $A_i:=\|v_i\|_{L^{\infty}((0,+\infty)\times\R^n}\le A$ denotes the $L^{\infty}$ norm of the $i$-th component $v_i$ of $v$. Using standard estimates, for $z>0$, we have $\erfc(z) \leq \pi^{-1/2} e^{-z^2}/z$.  Hence, we have that
\[\begin{split}
	\Gamma(t,0,x,0)
		&\geq \prod_{i=1}^n \left(\frac{1}{\sqrt{4\pi t}}e^{-\frac{(|x_i| + A_i  t)^2}{4t}} - \frac{A_i}{4}\,\frac{2}{\sqrt{\pi} \left(\frac{|x_i|}{\sqrt t} + A_i\sqrt t\right)}e^{- \frac{(|x_i| + A_it)^2}{4t}} \right)\\
		&\geq \frac{e^{-\frac{(|x| + A\,\sqrt{n}\,t)^2}{4t}}}{(4 \pi t)^{n/2}} \prod_{i=1}^n \left(1 - \frac{A_i}{\frac{|x_i|}{t} + A_i}\right)
		\geq  \frac{e^{-\frac{(|x| + A \,\sqrt{n}\,t)^2}{4t}}}{(16 \pi t)^{n/2}},
\end{split}\]
where in the second-to-last inequality we used that $|x_i| > A t \geq A_i t$ (implying in particular that all factors in the product are positive) and that
$$(|x_1|+A_1t)^2+\cdots+(|x_n|+A_nt)^2\le|x|^2+2\,|x|\,A\,\sqrt{n}\,t+n\,A^2\,t^2=(|x|+A\,\sqrt{n}\,t)^2.$$
This is exactly the bound claimed above.

To get the upper bounds, we may assume without loss of generality that $s=0<t$, $y=0$ and $|x_1|=|x|$ (hence, $x_2=\cdots=x_n=0$). Hill \cite[Theorem~2.1]{Hill} shows that
\[
	\Gamma(t,0,x,0)
		\leq \prod_{i=1}^n \left(\frac{1}{\sqrt{4\pi t}}\,e^{-\frac{(|x_i| - A_it)^2}{4t}} + \frac{A_i}{4}\erfc\left(\frac{|x_i|}{2\sqrt t} - \frac{A_i\sqrt t}{2} \right) \right)\le\left(\frac{1}{\sqrt{4\pi t}}+\frac{A}{2}\right)^n.
\]
Furthermore, if $|x|=|x_1|>A\,t$, then $|x_1|-A_1\,t\ge|x|-A\,t>0$ and
$$\begin{array}{rcl}
\Gamma(t,0,x,0) & \leq & \displaystyle\left(\frac{1}{\sqrt{4\pi t}}e^{-\frac{(|x_1| - A_1t)^2}{4t}} + \frac{A_1}{4}\erfc\left(\frac{|x_1|}{2\sqrt t} - \frac{A_1\sqrt t}{2} \right) \right)\times\left(\frac{1}{\sqrt{4\pi t}}+\frac{A}{2}\right)^{n-1}\vspace{3pt}\\
& \leq & \displaystyle\left(\frac{1}{\sqrt{4\pi t}}+ \frac{A\,\sqrt{t}}{\sqrt{4\pi}\,(|x|-At)} \right)\times e^{-\frac{(|x| - A\,t)^2}{4t}} \times\left(\frac{1}{\sqrt{4\pi t}}+\frac{A}{2}\right)^{n-1}.\end{array}$$
These bounds are the upper bounds claimed in Lemma~\ref{lem:weak_heat_kernel}.

\begin{proof}[Proof of \Cref{prop:fokker_planck}]
The proof of \Cref{prop:fokker_planck} is a straightforward application of \Cref{lem:weak_heat_kernel}.  Indeed, letting $\Gamma$ be the fundamental solution of~\eqref{eq:kfp}.  Then, we notice that $\tilde \Gamma(t,s,x,y) := \Gamma(t,s,y,x)$ is the fundamental solution of the adjoint operator of~\eqref{eq:kfp}, $\partial_t - v\partial_x - \partial_{xx}$; see, for example,~\cite[Theorem~10]{Aronson}. Hence, by extending the field $v$ with, say, $v=0$ in $(T,+\infty)\times\R$ (the extended field $v$ still satisfies $\|v\|_{L^{\infty}((0,+\infty)\times\R}\le A$), we apply \Cref{lem:weak_heat_kernel} to $\tilde{\Gamma}$  to obtain, for all $T\ge t>s\ge0$ and $|x-y| \geq A(t-s) + 1$,
\[
	\Gamma(t,s,x,y) = \tilde \Gamma(t,s,y,x)
		\leq \left(\frac{1}{\sqrt{4\pi(t-s)}}+\frac{A\sqrt{t-s}}{\sqrt{4\pi}(|x-y|-A(t-s))}\right)e^{-\frac{(|x-y|-A(t-s))^2}{4(t-s)}}.
\]
Using this, along with the fundamental solution representation of $u$ yields, for all $|x|\ge A\,T+a+1$,
\[\begin{split}
	u(T,x)
		&= \int_\R \Gamma(T,0,x,y)\,u_0(y)\, dy\\
		&\leq \int_\R \left(\frac{1}{\sqrt{4\pi T}} + \frac{A\sqrt{T}}{\sqrt{4\pi}\,(|x-y|-A\,T)} \right) e^{-\frac{(|x-y|- A\,T)^2}{4T}}\,\ind_{[-a,a]}(y)\,dy\\
		&\leq \frac{a}{\sqrt{\pi}}\left(\frac{1}{\sqrt{T}} + \frac{A\sqrt{T}}{|x|-A\,T-a} \right) e^{-\frac{(|x|-A\,T-a)^2}{4T}},
\end{split}\]
which concludes the proof.
\end{proof}


\subsection{The local-in-time Harnack inequality: proof of \Cref{lem:harnack}}\label{sec:harnack}

The final technical lemma to prove is \Cref{lem:harnack}.  We do that here. Our proof proceeds as follows: first, we notice that the heat kernel bound due to Hill~\cite{Hill} on operators of the form $\partial_t + v\cdot \nabla - \Delta$ is sharp in the spatial decay of the tails even if the bound is quite weak when $|x|$ is small.  The decay in the tails is crucial in the first step of \Cref{lem:harnack} where we show that $u(t,x)$ and $u(t+s,y)^{1/p}$ may be compared at any two points $x,y$.  Finally, we use this to bootstrap to the gradient Harnack estimate, finishing the proof of \Cref{lem:harnack}.  In this last step we use crucially that our new Harnack inequality has the shift forward in time for any $s\in[0,s_0]$.

\begin{proof}[Proof of Lemma~$\ref{lem:harnack}$] 
As mentioned above, we first prove the inequality involving only $u$.  To this end, fix $t_0>0$, $s_0\ge0$, $R>0$, $p \in (1,+\infty)$ and $t \geq t_0 > 0$, and let $q$ be the conjugate exponent to $p$, $\delta = \min\{t_0/2, 1\}$, $\alpha = (1+p)/(2p)>0$, and
$$A =\sup_{t>0}\Big(\|\,|F(u(t,\cdot))|\,\|_{L^{\infty}(\R^n)}+\|\nabla\cdot F(u(t,\cdot)\|_{L^{\infty}(\R^n)}\Big).$$
Notice that $\alpha\, p > 1$, that $\alpha < 1$, that $\alpha$ depends only on $p$ while $\delta$ depends only on $t_0$, and that $A$ is a real number from the assumptions on $F$ and $u$. Define $\overline u:(s',x)\mapsto\overline{u}(s',x)$ that solves
\[\begin{cases}
	\overline u_{s'} + F(u(s',\cdot)) \cdot \nabla \overline u = \Delta \overline u, \qquad &\text{ in } (t - \delta,+\infty) \times \R^n\\
	\overline u(t-\delta,\cdot) = u(t-\delta,\cdot) &\text{ in } \R^n.
\end{cases}\]
It follows from the maximum principle that $0\le\overline{u}\le\|u\|_\infty$ in $(t-\delta,+\infty)\times\R^n$. A straightforward computation shows that the functions
$$u^-(s',x)=\min\{1,\|u\|_{\infty}^{-1}\}\,e^{-A(s'-t+\delta)}\,\overline u(s',x)\ \hbox{ and }\ u^+(s',x)=e^{(1+A)(s'-t+\delta)}\,\overline u(s',x)$$
are, respectively, a sub- and super-solution of~\eqref{eqF} for $(s',x) \in (t-\delta,+\infty)\times \R^n$, in the sense that
$$u^-_{s'}+F(u(s',\cdot))\cdot\nabla u^-+\nabla\cdot F(u(s',\cdot))\,u^-\le\Delta u^-+u^-(1-u^-)$$
and
$$u^+_{s'}+F(u(s',\cdot))\cdot\nabla u^++\nabla\cdot F(u(s',\cdot))\,u^+\ge\Delta u^++u^+(1-u^+)$$
in $(t-\delta,+\infty)\times\R^n$. Furthermore, $u^-(t-\delta,\cdot)\le u(t-\delta,\cdot)=u^+(t-\delta,\cdot)$ in $\R^n$. Hence, we have that, for all $(s',x) \in(t-\delta,+\infty)\times\R^n$,
\begin{equation}\label{eq:u_overline_u}
	\min\{1,\|u\|_{\infty}^{-1}\}\,e^{-A(s' -t+\delta)}\, \overline u(s',x)
		\leq u(s',x)
		\leq e^{(1+A)(s'-t+\delta)}\,\overline u(s',x).
\end{equation}
Let $\Gamma(t,s,x,y)$ be the fundamental solution of~\eqref{gammav} with $v(t,x)=F(u(t,\cdot))(x)$.  Then, using~\eqref{eq:u_overline_u} and H\"older's inequality, we write
\[\begin{split}
	u(t,x)
		&\leq e^{(1+A)\delta}\,\overline u(t,x)
		= e^{(1+A)\delta} \int_{\R^n} \Gamma(t,t-\delta,x,z)\,u(t-\delta,z)\,dz\\
		&\leq e^{(1+A)\delta}\,\|u\|_\infty^{1/q} \int_{\R^n} \left(u(t-\delta,z)\,\Gamma(t,t-\delta,x,z)^{\alpha p}\right)^{1/p} \left( \Gamma(t,t-\delta,x,z)^{(1-\alpha) q}\right)^{1/q}dz\\
		&\leq e^{(1+A)\delta}\,\|u\|_\infty^{1/q}\,\|\Gamma(t,t-\delta,x,\cdot)^{1-\alpha}\|_q \left(\int_{\R^n} u(t-\delta,z) \Gamma(t,t-\delta,x,z)^{\alpha p} dz\right)^{1/p}.
\end{split}\]
Note that \Cref{lem:weak_heat_kernel} along with the fact that $\alpha < 1$ implies that $\|\Gamma(t,t-\delta,x,\cdot)^{1-\alpha}\|_q$ is bounded by a constant depending only on $n$, $A$, $\delta$, $\alpha$ and $q$, and then only on $t_0$, $p$, $A$ and $n$. Hence, we are finished with the proof of~\eqref{*14} if we show that there exists a constant $C_0>0$ depending only on $t_0$, $s_0$, $R$, $p$, $A$ and $n$ such that
\begin{equation}\label{eq:comparable_kernel}
	\Gamma(t,t-\delta, x, z)^{\alpha p} \leq C_0\,\Gamma(t+s,t-\delta, y, z)
\end{equation}
for all $s\in[0,s_0]$, $|x-y|\le R$ and $z\in\R^n$. Indeed, were this the case, then the above, along with~\eqref{eq:u_overline_u}, implies that
\begin{equation}\label{eq:p_harnack}
\begin{split}
	u(t,x)
		&\leq C_0^{1/p}\,e^{(1+A)\delta}\,\|u\|_\infty^{1/q} \left(\int_{\R^n} u(t-\delta,z)\,\Gamma(t+s,t-\delta,x,z)\,dz\right)^{1/p}\\
		&= C_0^{1/p}\,e^{(1+A)\delta}\,\|u\|_\infty^{1/q}\,\overline u(t+s,x)^{1/p}\\
		&= C_0^{1/p}\,e^{(1+A)\delta +A(\delta+s)/p}\,\|u\|_\infty^{1/q}\,\left( e^{-A(\delta+s)} \,\overline u(t+s,x)\right)^{1/p}\\
		&\leq C_0^{1/p}\,e^{(1+A)\delta +A(\delta+s_0)/p}\,\max\{\|u\|_\infty^{1-1/p},\|u\|_\infty\}\,u(t+s,x)^{1/p}.
\end{split}\end{equation}

We now prove~\eqref{eq:comparable_kernel}.  We assume that $z = 0$, though the general case follows similarly.  We fix any $s\in[0,s_0]$ and any $x$, $y$ in $\R^n$ such that $|x-y|\le R$. There are two cases to consider.  If  $|x|, |y| \leq\sqrt n\,A\,\delta + R+1$, then we use well-known heat kernel bounds from, e.g.,~\cite[Theorem 10]{Aronson}, which state that there exists $C_1\ge1$, depending only on $\delta$, $s_0$, $A$, and $n$ such that, for any $x',y' \in \R^n$ and any $0 \leq s' < t'$, such that $t' - s' \leq s_0+\delta$, then
\begin{equation}\label{eq:Aronson}
	\frac{1}{C_1(t'-s')^{n/2}} e^{-C_1\frac{|x'-y'|^2}{t'-s'}}
		\leq \Gamma(t',s', x',y')
		\leq \frac{C_1}{(t'-s')^{n/2}} e^{-\frac{|x'-y'|^2}{C_1(t'-s')}}.
\end{equation}
From~\eqref{eq:Aronson}, it is clear that there exists a constant $C_2>0$ depending only on $\delta$, $s_0$, $A$, $\alpha$, $p$, $n$, and $R$, and hence only on $t_0$, $s_0$, $R$, $p$, $A$, and $n$, such that
\begin{equation}\label{eq:Gamma_1}
	\frac{\Gamma(t,t-\delta,x,0)^{\alpha p}}{\Gamma(t+s,t-\delta,y,0)}
		\leq C_2
\end{equation}
if $|x|,|y|\le \sqrt n\,A\,\delta +R+1$.  We note that the bounds in \cite[Theorem~10]{Aronson} are not sharp enough as $|x| \to \infty$ to be useful in the regime $|x| > \sqrt n\,A\,\delta+R+1$.

If  either $|x|$ or $|y|$ is larger than $\sqrt{n}\,A\,\delta + R+1$, then both $|x|, |y| > \sqrt n\,A\,\delta+1\ge A\,\delta+1$ because $|x-y| \leq R$.  Then, applying the bounds in \Cref{lem:weak_heat_kernel} yields:
\begin{equation}\label{eq:Gamma_alpha}
	\frac{\Gamma(t,t-\delta,x,0)^{\alpha p}}{\Gamma(t+s,t-\delta,y,0)}
		\leq \frac{\left(\!\frac{1}{\sqrt{4\pi\delta}}\!+\!\frac{A}{2}\!\right)^{\alpha p(n-1)}\!\!\!\!\times\left(\!\frac{1}{\sqrt{4\pi\delta}}\!+\!\frac{A\sqrt{\delta}}{\sqrt{4\pi}(|x|\!-\!A\delta)}\!\right)^{\alpha p}e^{-\frac{\alpha p(|x|\!-\!A\delta)^2}{4\delta}}}
		{(16\pi (\delta + s))^{-n/2}\,e^{-\frac{(|y| + A\,\sqrt{n}\,(s+\delta))^2}{4(s+\delta)}}}.
\end{equation}
Let $\epsilon > 0$ be a constant to be determined and $A_1 = |x-y| + A\,\sqrt{n}\,(s+\delta)$.  Notice that $A_1$ is bounded by a constant depending only on $A$, $n$, $s_0$, $t_0$, and $R$ (remember that $|x-y|\le R$) and that $|y| + A\,\sqrt{n}\,(s+\delta) \leq |x| + A_1$.  Using~\eqref{eq:Gamma_alpha} and Young's inequality, there exists a constant $C_3>0$ depending only on $t_0$, $s_0$, $p$, $A$, and $n$, such that
\[\begin{split}
	\frac{\Gamma(t,t-\delta,x,0)^{\alpha p}}{\Gamma(t+s,t-\delta,y,0)}
		&\leq C_3 \exp\left\{ - \alpha p \frac{|x|^2 - 2|x|A\delta + A^2\delta^2}{4\delta} + \frac{|x|^2 + 2|x|A_1 + A_1^2}{4(s+\delta)}\right\}\\
		&\leq C_3 \exp\left\{ - \alpha p \frac{(1- \e)|x|^2 + \left(1 -\frac{1}{\e}\right)A^2\delta^2}{4\delta} + \frac{(1+\epsilon)|x|^2 + \left(1 + \frac{1}{\e}\right)A_1^2}{4(s+\delta)}\right\}\\
		&\leq C_3 \exp\left\{ - |x|^2\left(\frac{\alpha p(1-\e)}{4\delta} - \frac{1+\e}{4(\delta + s)}\right) + \frac{\alpha p A^2\delta^2}{4\delta \e} + \frac{\left(1 + \frac{1}{\e}\right)A_1^2}{4(s+\delta)}\right\}.
\end{split}\]
Since $\alpha p > 1$ and $\delta \leq \delta + s$, it follows that we may choose $\e>0$ small enough, depending only on $p$, such that $\alpha p (1-\e) \geq 1+\e$ and hence
\[
	\frac{\alpha p(1-\e)}{4\delta} - \frac{1+\e}{4(\delta + s)} \geq 0.
\]
Using this inequality, we see that there exists $C_4$ depending only on $t_0$, $s_0$, $p$, $A$, $n$, and $R$, such that
\[
	\frac{\Gamma(t,t-\delta,x,0)^{\alpha p}}{\Gamma(t+s,t-\delta,y,0)} \leq C_4
\]
for all $s\in[0,s_0]$, $|x-y|\le R$ and $\max\{|x|, |y|\}>\sqrt n\,A\,\delta+R+1$, as desired.  The combination of this with~\eqref{eq:Gamma_1} implies~\eqref{eq:comparable_kernel}.  This yields~\eqref{eq:p_harnack} and~\eqref{*14}.

We now show how to obtain a gradient bound on $u$ from~\eqref{*14}.  Let $R = |x-y|$.  The local $L^q$ parabolic estimates, see e.g.~\cite[Theorem~7.22]{Lieberman} along with the anisotropic Sobolev embedding for $q = n+3$ (see e.g.~\cite[Lemma A3]{Engler}), implies that there exists a constant $C_{A,\delta,n,R}>0$ that depends only on $A$, $\delta$, $n$ and $R$ such that, for any $(t,x) \in (t_0,+\infty)\times \R^n$,
\[\begin{split}
	|\nabla u(t,x)|
		&\leq C_{A,\delta,n,R}\,\left(\|u\|_{L^q([t-\delta,t]\times B_R(x))} + \|u(1-u)\|_{L^q([t-\delta,t]\times B_R(x))} \right)\\
		&\leq 2\,C_{A,\delta,n,R}\,\|u\|_{L^\infty([t-\delta,t]\times B_R(x))}\,(1 + \|u\|_\infty).
\end{split}\]
Applying~\eqref{*14} to $\|u\|_{L^\infty([t-\delta,t]\times B_R(x))}$ implies that, for any $p\in (1,\infty)$, there exists a constant $C_0>0$ that depends only on $t_0$, $R$, $p$, $A$, and $n$ such that
\[
	|\nabla u(t,x)|
		\leq C_0\,\max\{\|u\|_\infty^{1-1/p},\|u\|_\infty\}\, (1 + \|u\|_\infty)\,u(t,y)^{1/p}
\]
for all $|x-y|\le R$, which finishes the proof.
\end{proof}

 
\bibliographystyle{abbrv}
\bibliography{refs2}

\begin{bibdiv}
\begin{biblist}

\bib{Aronson}{article}{
      author={Aronson, D.~G.},
       title={Non-negative solutions of linear parabolic equations},
        date={1968},
     journal={Ann. Scuola Norm. Sup. Pisa (3)},
      volume={22},
       pages={607\ndash 694},
}

\bib{BKV}{article}{
      author={Belk, M.},
      author={Kazmierczak, B.},
      author={Volpert, V.},
       title={Existence of reaction-diffusion-convection waves in unbounded
  strips},
        date={2005},
        ISSN={0161-1712},
     journal={Int. J. Math. Math. Sci.},
      number={2},
       pages={169\ndash 193},
}

\bib{BenAmar}{article}{
      author={Ben~Amar, M.},
       title={Collective chemotaxis and segregation of active bacterial
  colonies},
        date={2016},
     journal={Scient. Rep.},
      volume={6},
}

\bib{BerestyckiConstantinRyzhik}{article}{
      author={Berestycki, H.},
      author={Constantin, P.},
      author={Ryzhik, L.},
       title={Non-planar fronts in {B}oussinesq reactive flows},
        date={2006},
        ISSN={0294-1449},
     journal={Ann. Inst. Henri Poincar\'e, Anal. Non Lin\'eaire},
      volume={23},
       pages={407 \ndash  437},
  url={http://www.sciencedirect.com/science/article/pii/S0294144905000326},
}

\bib{BerestyckiHamelNadin}{article}{
      author={Berestycki, H.},
      author={Hamel, F.},
      author={Nadin, G.},
       title={Asymptotic spreading in heterogeneous diffusive excitable media},
        date={2008},
        ISSN={0022-1236},
     journal={J. Funct. Anal.},
      volume={255},
       pages={2146\ndash 2189},
         url={http://dx.doi.org/10.1016/j.jfa.2008.06.030},
}

\bib{BHR_LogDelay}{article}{
      author={Bouin, E.},
      author={Henderson, C.},
      author={Ryzhik, L.},
       title={The {B}ramson logarithmic delay in the cane toads equations},
        date={2016},
     journal={Preprint},
        note={https://arxiv.org/abs/1610.03285},
}

\bib{ConstantinKiselevObermanRyzhik}{article}{
      author={Constantin, P.},
      author={Kiselev, A.},
      author={Oberman, A.},
      author={Ryzhik, L.},
       title={Bulk burning rate in passive-reactive diffusion},
        date={2000},
     journal={Arch. Ration. Mech. Anal.},
      volume={154},
       pages={53\ndash 91},
         url={http://dx.doi.org/10.1007/s002050000090},
}

\bib{ConstantinLewickaRyzhik}{article}{
      author={Constantin, P.},
      author={Lewicka, M.},
      author={Ryzhik, L.},
       title={Travelling waves in two-dimensional reactive {B}oussinesq systems
  with no-slip boundary conditions},
        date={2006},
     journal={Nonlinearity},
      volume={19},
       pages={2605\ndash 2615},
         url={http://dx.doi.org/10.1088/0951-7715/19/11/006},
}

\bib{CRRV}{article}{
      author={Constantin, P.},
      author={Roquejoffre, J.-M.},
      author={Ryzhik, L.},
      author={Vladimirova, N.},
       title={Propagation and quenching in a reactive {B}urgers-{B}oussinesq
  system},
        date={2008},
        ISSN={0951-7715},
     journal={Nonlinearity},
      volume={21},
       pages={221\ndash 271},
         url={http://dx.doi.org/10.1088/0951-7715/21/2/003},
}

\bib{Crooks}{article}{
      author={Crooks, E.C.M.},
       title={Travelling fronts for monostable reaction-diffusion systems with
  gradient-dependence},
        date={2003},
        ISSN={1079-9389},
     journal={Adv. Diff. Equations},
      volume={8},
       pages={279\ndash 314},
}

\bib{CrooksMascia}{article}{
      author={Crooks, E.C.M.},
      author={Mascia, C.},
       title={Front speeds in the vanishing diffusion limit for
  reaction-diffusion-convection equations},
        date={2007},
        ISSN={0893-4983},
     journal={Diff. Int. Equations},
      volume={20},
       pages={499\ndash 514},
}

\bib{Engler}{article}{
      author={Engler, H.},
       title={Global smooth solutions for a class of parabolic
  integrodifferential equations},
        date={1996},
        ISSN={0002-9947},
     journal={Trans. Amer. Math. Soc.},
      volume={348},
       pages={267\ndash 290},
         url={http://dx.doi.org/10.1090/S0002-9947-96-01472-9},
}

\bib{FabesStroock}{article}{
      author={Fabes, E.B.},
      author={Stroock, D.W.},
       title={A new proof of {M}oser's parabolic {H}arnack inequality using the
  old ideas of {N}ash},
        date={1986},
        ISSN={0003-9527},
     journal={Arch. Ration. Mech. Anal.},
      volume={96},
       pages={327\ndash 338},
}

\bib{Fisher}{article}{
      author={Fisher, R.},
       title={The wave of advance of advantageous genes},
        date={1937},
     journal={Ann. Eugenics},
      volume={7},
       pages={355\ndash 369},
}

\bib{Gartner}{article}{
      author={G{\"a}rtner, J.},
       title={Location of wave fronts for the multidimensional {KPP} equation
  and {B}rownian first exit densities},
        date={1982},
        ISSN={0025-584X},
     journal={Math. Nachr.},
      volume={105},
       pages={317\ndash 351},
}

\bib{Grima}{article}{
      author={Grima, R.},
       title={Strong-coupling dynamics of a multicellular chemotactic system},
        date={2005},
     journal={Phys. Rev. Letters},
      volume={95},
       pages={128103},
}

\bib{Henderson_Boussinesq}{article}{
      author={Henderson, C.},
       title={Pulsating fronts in a 2{D} reactive {B}oussinesq system},
        date={2014},
        ISSN={0360-5302},
     journal={Comm. Part. Diff. Equations},
      volume={39},
       pages={1555\ndash 1595},
         url={http://dx.doi.org/10.1080/03605302.2013.850726},
}

\bib{Hill}{article}{
      author={Hill, A.T.},
       title={Estimates on the heat kernel of parabolic equations with
  advection},
        date={1997},
        ISSN={0036-1410},
     journal={SIAM J. Math. Anal.},
      volume={28},
       pages={1309\ndash 1316},
         url={http://dx.doi.org/10.1137/S003614109630104X},
}

\bib{KeatingBonner}{article}{
      author={Keating, M.T.},
      author={Bonner, J.T.},
       title={Negative chemotaxis in cellular slime molds},
        date={1977},
     journal={J. Bacteriology},
      volume={130},
       pages={144\ndash 147},
}

\bib{KPP}{article}{
      author={Kolmogorov, {A.N.}},
      author={Petrovskii, {I.G.}},
      author={Piskunov, {N.S.}},
       title={{\'E}tude de l'\'equation de la chaleur de mati\`ere et son
  application \`a un probl\`eme biologique},
        date={1937},
     journal={Bull. Moskov. Gos. Univ. Mat. Mekh.},
      volume={1},
       pages={1\ndash 25},
}

\bib{Lewicka}{article}{
      author={Lewicka, M.},
       title={Existence of traveling waves in the {S}tokes-{B}oussinesq system
  for reactive flows},
        date={2007},
     journal={J. Diff. Equations},
      volume={237},
       pages={343\ndash 371},
         url={http://dx.doi.org/10.1016/j.jde.2007.03.019},
}

\bib{LewickaMucha}{article}{
      author={Lewicka, M.},
      author={Mucha, P.},
       title={On the existence of traveling waves in the 3{D} {B}oussinesq
  system},
        date={2009},
     journal={Comm. Math. Phys.},
      volume={292},
       pages={417\ndash 429},
         url={http://dx.doi.org/10.1007/s00220-009-0904-3},
}

\bib{Lieberman}{book}{
      author={Lieberman, G.M.},
       title={Second order parabolic differential equations},
   publisher={World Scientific Publishing Co., Inc., River Edge, NJ},
        date={1996},
        ISBN={981-02-2883-X},
         url={http://dx.doi.org/10.1142/3302},
}

\bib{MalhamXin}{article}{
      author={Malham, S.},
      author={Xin, J.},
       title={Global solutions to a reactive {B}oussinesq system with front
  data on an infinite domain},
        date={1998},
     journal={Comm. Math. Phys.},
      volume={193},
       pages={287\ndash 316},
         url={http://dx.doi.org/10.1007/s002200050330},
}

\bib{MimuraOhara}{article}{
      author={Mimura, M.},
      author={Ohara, K.},
       title={Standing wave solutions for a {F}isher type equation with a
  nonlocal convection},
        date={1986},
        ISSN={0018-2079},
     journal={Hiroshima Math. J.},
      volume={16},
       pages={33\ndash 50},
         url={http://projecteuclid.org/euclid.hmj/1206130536},
}

\bib{NadinPerthameRyzhik}{article}{
      author={Nadin, G.},
      author={Perthame, B.},
      author={Ryzhik, L.},
       title={Traveling waves for the {K}eller-{S}egel system with {F}isher
  birth terms},
        date={2008},
        ISSN={1463-9963},
     journal={Interfaces Free Bound.},
      volume={10},
       pages={517\ndash 538},
         url={http://dx.doi.org/10.4171/IFB/200},
}

\bib{SalakoShen1}{article}{
      author={Salako, R.},
      author={Shen, W.},
       title={Global existence and asymptotic behavior of classical solutions
  to a parabolic-elliptic chemotaxis system with logistic source on
  $\mathbb{R}^n$},
     journal={{https://arxiv.org/abs/1608.02031}},
}

\bib{SalakoShen2}{article}{
      author={Salako, R.},
      author={Shen, W.},
       title={Spreading speeds and traveling waves of a parabolic-elliptic
  chemotaxis system with logistic source on $\mathbb{R}^n$},
     journal={{https://arxiv.org/abs/1609.05387}},
}

\bib{Sengupta}{article}{
      author={Sengupta, A.},
      author={van Teeffelen, S.},
      author={L{\"o}wen, H.},
       title={Dynamics of a microorganism moving by chemotaxis in its own
  secretion},
        date={2009},
     journal={Phys. Rev. E},
      volume={80},
       pages={031122},
}

\bib{Tello}{article}{
      author={Tello, J.I.},
       title={Mathematical analysis and stability of a chemotaxis model with
  logistic term},
        date={2004},
        ISSN={0170-4214},
     journal={Math. Methods Appl. Sci.},
      volume={27},
       pages={1865\ndash 1880},
         url={http://dx.doi.org/10.1002/mma.528},
}

\bib{TelloWinkler}{article}{
      author={Tello, J.I.},
      author={Winkler, M.},
       title={A chemotaxis system with logistic source},
        date={2007},
        ISSN={0360-5302},
     journal={Comm. Part. Diff. Equations},
      volume={32},
       pages={849\ndash 877},
         url={http://dx.doi.org/10.1080/03605300701319003},
}

\bib{TexierPicardVolpert}{article}{
      author={Texier-Picard, R.},
      author={Volpert, V.},
       title={Probl\`emes de r\'eaction-diffusion-convection dans des cylindres
  non born\'es},
        date={2001},
        ISSN={0764-4442},
     journal={C. R. Acad. Sci. Paris S\'er. I Math.},
      volume={333},
       pages={1077\ndash 1082},
         url={http://dx.doi.org/10.1016/S0764-4442(01)02178-4},
}

\bib{Uchiyama}{article}{
      author={Uchiyama, K.},
       title={The behavior of solutions of some semilinear diffusion equation
  for large time},
        date={1978},
     journal={J. Math. Kyoto Univ.},
      volume={18},
       pages={453\ndash 508},
}

\bib{VCKRR}{article}{
      author={Vladimirova, N.},
      author={Constantin, P.},
      author={Kiselev, A.},
      author={Ruchayskiy, O.},
      author={Ryzhik, L.},
       title={Flame enhancement and quenching in fluid flows},
        date={2003},
        ISSN={1364-7830},
     journal={Combust. Theory Model.},
      volume={7},
       pages={487\ndash 508},
         url={http://dx.doi.org/10.1088/1364-7830/7/3/303},
}

\bib{VladimirovaRosner}{article}{
      author={Vladimirova, N.},
      author={Rosner, R.},
       title={Model flames in the {B}oussinesq limit: the effects of feedback},
        date={2003},
        ISSN={1539-3755},
     journal={Phys. Rev. E},
      volume={67},
       pages={066305, 10},
         url={http://dx.doi.org/10.1103/PhysRevE.67.066305},
}

\bib{Winkler}{article}{
      author={Winkler, M.},
       title={Chemotaxis with logistic source: very weak global solutions and
  their boundedness properties},
        date={2008},
     journal={J. Math. Anal. Appl.},
      volume={348},
       pages={708\ndash 729},
}

\bib{ZakiAndrewInsall}{article}{
      author={Zaki, M.},
      author={Andrew, N.},
      author={Insall, R.H.},
       title={Entamoeba histolytica cell movement: a central role for
  self-generated chemokines and chemorepellents},
        date={2006},
     journal={Proc. Nat. Acad. Sci.},
      volume={103},
       pages={18751\ndash 18756},
}

\end{biblist}
\end{bibdiv}
\end{document}